\documentclass[a4paper, 10pt, twoside, notitlepage]{amsart}
\usepackage[utf8]{inputenc}
\usepackage{color}
\usepackage{amsmath} 
\usepackage{amssymb} 
\usepackage{amsthm}
\usepackage{graphicx}
\usepackage{esint}
\usepackage[colorlinks=true,linkcolor=blue]{hyperref}
\usepackage[export]{adjustbox}
\usepackage{thmtools}

\usepackage{enumitem}

\theoremstyle{plain}
\newtheorem{thm}{Theorem}
\newtheorem{prop}{Proposition}[section]

\newtheorem{lem}[prop]{Lemma}

\newtheorem{defi}[prop]{Definition}
\newtheorem{rmk}[prop]{Remark}

\newtheorem{alg}[prop]{Algorithm}
\newtheorem{claim}{Claim}

\newcommand {\R} {\mathbb{R}} 
 \newcommand {\N} {\mathbb{N}}

\newcommand {\diam} {\text{diam}}

\DeclareMathOperator{\diag}{diag}
\DeclareMathOperator{\tr}{tr}

\DeclareMathOperator {\dist} {dist}

\DeclareMathOperator {\intconv} {intconv}
\DeclareMathOperator {\conv}{conv}
\DeclareMathOperator {\inte} {int}

\DeclareMathOperator {\Skew} {Skew}

\DeclareMathOperator {\rank} {rank}
\DeclareMathOperator {\Per} {Per}
\DeclareMathOperator {\sgn} {sgn}

\begin{document}

\title[Higher Regularity of Convex Integration Solutions]{Higher Sobolev Regularity of Convex Integration Solutions in Elasticity: The Dirichlet Problem with Affine Data in $\inte(K^{lc})$}

\author{Angkana R\"uland }
\author{Christian Zillinger}
\author{Barbara Zwicknagl}

\address{
Mathematical Institute of the University of Oxford, Andrew Wiles Building, Radcliffe Observatory Quarter, Woodstock Road, OX2 6GG Oxford, United Kingdom }
\email{ruland@maths.ox.ac.uk}

\address{
Department of Mathematics,
University of Southern California,
Los Angeles, CA 90089-2532, US}
\email{zillinge@usc.edu}

\address{
Institut f\"ur Mathematik,
Universit\"at W\"urzburg,
Emil-Fischer-Stra{\ss}e 40,
97074 W\"urzburg Germany  }
\email{barbara.zwicknagl@mathematik.uni-wuerzburg.de}

\begin{abstract}
In this article we continue our study of higher Sobolev regularity of flexible convex integration solutions to differential inclusions arising from applications in materials sciences. We present a general framework yielding higher Sobolev regularity for Dirichlet problems with affine data in $\inte(K^{lc})$. This allows us to simultaneously deal with linear and nonlinear differential inclusion problems. We show that the derived higher integrability and differentiability exponent has a lower bound, which is independent of the position of the Dirichlet boundary data in $\inte(K^{lc})$. 
As applications we discuss the regularity of weak isometric immersions in two and three dimensions as well as the differential inclusion problem for the geometrically linear hexagonal-to-rhombic and the cubic-to-orthorhombic phase transformations occurring in shape memory alloys.  
\end{abstract}

\thanks{
A.R. acknowledges a Junior Research Fellowship at Christ Church.
B.Z. acknowledges support from the DFG through CRC 1060 ``The mathematics of emergent effects''.}

\maketitle
\tableofcontents

\section{Introduction}
\label{sec:intro}

In this article, we continue our investigation of higher regularity properties of convex integration solutions, which was started in \cite{RZZ16}. We hence analyse regularity properties on a \emph{Sobolev scale} of solutions to \emph{$m$-well problems}, which are motivated by materials sciences, in particular by shape-memory alloys.\\ 

Shape-memory alloys are materials which undergo a first order diffusionless solid-solid phase transformation in which the underlying crystalline lattice loses some of its symmetries. Here, in general, the \emph{austenite}, which is the high temperature phase, has many symmetries, whiles the low temperature phase, the \emph{martensite}, loses some of these. As a consequence of this loss of symmetry several variants of martensite coexist at temperatures below the critical transformation temperature $\theta_c$. \\

Often these materials are modelled in the framework of the \emph{phenomenological theory of martensite} \cite{B3}, which adopts a variational point of view. Here energy functionals of the type
\begin{align}
\label{eq:min}
\mathcal{E}(\nabla u, \theta):=\int\limits_{\Omega} W(\nabla u, \theta) dx
\end{align}
are minimized (e.g. subject to certain boundary conditions), where $\Omega \subset \R^{3}$ corresponds to the \emph{reference configuration} (which is often chosen to be the undeformed austenite configuration at a fixed temperature), $u:\Omega \rightarrow \R^3$ denotes the \emph{deformation} (and $\nabla u$ the associated \emph{deformation gradient}) and $\theta: \Omega \rightarrow \R$ is the \emph{temperature}. The \emph{stored energy function} $W$ reflects the symmetry properties of the material under consideration. In particular, at high temperatures, it has a single energy minimum corresponding to the austenite phase, while at low temperatures it has various energetically equivalent minima corresponding to the variants of martensite. As the energies have to be invariant with respect to the material symmetries and in addition have to satisfy the requirement of frame indifference, the minimization of energies as in \eqref{eq:min} in general leads to very complex non-quasiconvex problems. \\

Hence instead of analysing \eqref{eq:min}, we in the sequel fix a temperature $\theta<\theta_c$ and consider only \emph{exactly stress-free} deformations $u$, i.e. deformations such that
\begin{align*}
\nabla u(x) \in K \mbox{ for a.e. } x \in \Omega,
\end{align*}
where $K=K(\theta)$ denotes the absolute minima of $W$ at temperature $\theta$. In the setting of transformations in shape-memory alloys this leads to an \emph{$m$-well problem}
\begin{align}
\label{eq:incl_nonlin}
\nabla u(x) \in \bigcup\limits_{j=1}^m SO(3)U_j \mbox{ for a.e. } x \in \Omega,
\end{align}
where $U_j=U_j^t$ are positive definite matrices modelling the variants of martensite. The $SO(3)$-invariance is a result of frame indifference.\\

Due to the nonlinear structure of $SO(3)$, it is often convenient to carry out a further simplification step and to linearise the problem around the identity. This then leads to a \emph{geometrically linearised $m$-well problem} for the displacement, which reads
\begin{align}
\label{eq:incl_lin}
\nabla u \in \{e^{(1)}, \dots, e^{(m)}\} + \Skew(3)
\end{align}
for $e^{(1)},\dots,e^{(m)} \in \R^{3 \times 3}_{sym}$.
Here due to its linear structure, it is often easier to handle $\Skew(3)$
invariance (which is the linearisation of $SO(3)$) than to work with the full
$SO(3)$ symmetry.\\

Considering the Dirichlet problem for the differential inclusions  \eqref{eq:incl_nonlin}, \eqref{eq:incl_lin} within a suitable class of (for instance) affine boundary conditions results in interesting behaviour: If the matrix space geometry of the set $K$ is suitable (in the sense that the lamination or rank-one convex hulls $K^{lc}$ or $K^{rc}$ are sufficiently large), it can be shown by means of convex integration \cite{MS} or the Baire category approach \cite{DaM12} that there is a very large set of solutions with very strong non-uniqueness properties. In a Baire category sense, the solution set is residual in a suitable topology (for instance in the space $W^{1,\infty}$ equipped with the $L^{\infty}$ topology). In this sense the differential inclusions are very \emph{flexible}. 

Since the inclusion problems \eqref{eq:incl_nonlin}, \eqref{eq:incl_lin} however have physical origins, it is a natural question whether all of these solutions are really the \emph{physically significant} ones or whether they are only \emph{mathematical artefacts}. In this context, it is known for specific problems \cite{DM1, K1, R16}, that for certain inclusion problems surface energy constraints, which mathematically correspond to regularity assumptions on $\nabla u$ rule out many of these multiple ``wild" solutions and the problems \eqref{eq:incl_nonlin}, \eqref{eq:incl_lin} become very \emph{rigid}. In fact only very few solutions exist for the problems studied in \cite{DM1, DM2, K1, R16} under $BV$ (or $BV$ type) constraints on $\nabla u$. Hence, a strong dichotomy between the rigid (for $\nabla u \in BV$) and the flexible (for $\nabla u \in L^{\infty}$) behaviour is present.

In this article we seek to study the described dichotomy further by investigating the regularity properties of convex integration solutions, showing that the flexible regime also exists on a Sobolev $W^{s,p}$ scale beyond the mere $\nabla u \in L^{\infty}$ bounds. This approach can be viewed complementary to the studies of rigidity of laminates or branching structures, or to quantitative rigidity estimates \cite{CO,CO1,ChermisiConti2010,JerrardLorent2013,TS}.\\

\subsection{The main result}
In studying higher Sobolev regularity properties in the flexible regime, we simultaneously consider both \emph{geometrically linearised} and \emph{nonlinear $m$-well problems}, if their underlying matrix space geometries are sufficiently ``simple", thus allowing us to focus on the analytical aspects of the problem. In this context we are in particular interested in the following three model problems:
\begin{itemize}
\item[(a)] \emph{Weak isometric immersions.} We consider the inclusion problem
\begin{align*}
\nabla u(x) \in O(n) \mbox{ a.e. in } \Omega,
\end{align*}
where $n\in\{2,3\}$ and $\Omega \subset \R^n$ is an open, bounded Lipschitz set. Here Liouville's theorem ensures that solutions with $C^1$ regularity are rigid, while classical results of Gromov prove the flexibility of the differential inclusion for $\nabla u \in L^{\infty}$ \cite{G, DaM12, MS, KSS15, S}. If, however, 
\begin{itemize}
\item $n=2$,
\item or if $n=3$ and additionally zero Dirichlet data are assumed,
\end{itemize}
stronger results are available: An ``origami" convex integration scheme due to Dacorogna, Marcellini and Paolini \cite{DMP08a, DMP08b, DMP08c, DMP10} shows that it is possible to construct solutions which are in any $W^{s,p}$ with $s\in(0,1)$, $p\in[1,\infty)$ and $sp\in(0,1)$. Hence, in these cases the ``complete" dichotomy is understood. The only obstruction ruling out flexible convex integration solutions is the presence of trace estimates, which requires ``high" Sobolev regularity.
This lack of rigidity can be viewed as a consequence of the very flexible structure of $K$ and the presence of multiple rank-one connections.
\item[(b)] \emph{The hexagonal-to-rhombic phase transformation.} The hexagonal-to-rhombic phase transformation can be viewed as a model setting for a very flexible differential inclusion in the context of shape-memory alloys:
\begin{align*}
\nabla u &\in K_h:=\left\{ 
\begin{pmatrix}
1 & 0 \\
0& -1
\end{pmatrix},
 \frac{1}{2}\begin{pmatrix}
-1 & \sqrt{3} \\
\sqrt{3}& 1
\end{pmatrix},
\frac{1}{2}\begin{pmatrix}
-1 & -\sqrt{3}\\
-\sqrt{3}& 1
\end{pmatrix} \right\} \\
& \quad \quad \qquad+ \Skew(2) \mbox{ a.e. in } \Omega.
\end{align*}
Due to its flexible structure, in \cite{RZZ16} we studied the
hexagonal-to-rhombic phase transformation as a model problem (with physical
significance \cite{KK91}, \cite{CPL14}) and derived higher
order Sobolev regularity for a class of convex integration solutions. While the flexibility of the transformation makes it an interesting test case to study convex integration solutions, we remark that there is no known rigidity result complementing the flexible regime.
\item[(c)] \emph{The cubic-to-orthorhombic phase transformation.} The inclusion problem for the cubic-to-orthorhombic transformation is given by
\begin{align*}
\nabla u \in K_{co}:=\{e^{(1)},\dots, e^{(6)}\} + \Skew(3) \mbox{ a.e. in } \Omega, 
\end{align*}
with
\begin{equation*}
\begin{split}
&e^{(1)}:= \begin{pmatrix}
1 & \delta & 0\\
\delta& 1 & 0\\
0 & 0 & -2
\end{pmatrix},
e^{(2)}:= \begin{pmatrix}
1 & -\delta & 0\\
-\delta& 1 & 0\\
0 & 0 & -2
\end{pmatrix},
e^{(3)}:= \begin{pmatrix}
1 & 0 & \delta\\
0& -2 & 0\\
\delta & 0 & 1
\end{pmatrix},\\
&e^{(4)}:= \begin{pmatrix}
1 & 0 & -\delta\\
0& -2 & 0\\
-\delta & 0 & 1
\end{pmatrix},
e^{(5)}:= \begin{pmatrix}
-2 & 0 & 0\\
0& 1 & \delta\\
0 & \delta & 1
\end{pmatrix},
e^{(6)}:= \begin{pmatrix}
-2 & 0 & 0\\
0& 1 & -\delta\\
0 & -\delta & 1
\end{pmatrix}.
\end{split}
\end{equation*} 
In addition to being a full three-dimensional inclusion (which makes the matrix space geometry harder than in (ii)), the cubic-to-orthorhombic phase transformation is a good model problem, since, due to the results in \cite{R16}, it is known that this transformation displays a dichotomy between rigidity and flexibility. Hence, it is particularly interesting to study finer properties of the arising convex integration solutions to understand whether the presence of the dichotomy gives rise to ``hidden" regularity constraints for convex integration solutions. 
\end{itemize}

In the context of these model problems our main result can be formulated as follows:

\begin{thm}
\label{thm:reg}
Let $n\in\{2,3\}$ and let $\Omega \subset \R^n$ be a bounded Lipschitz domain. Consider
\begin{equation}
\label{eq:incl}
\begin{split}
&\nabla u \in K \mbox{ a.e. in } \Omega,\\
& \nabla u  = M \mbox{ in } \R^n \setminus \overline{\Omega},
\end{split}
\end{equation}
for $M \in \inte(K^{lc})$. Then, if 
\begin{itemize}
\item[(a)] $n=2$ and $K=O(2)$ or $K= K_{h}$, 
\item[(b)] or if $n=3$ and $K=O(3)$ or $K=K_{co}$, 
\end{itemize}
there exists $ \theta_0 \in (0,1)$, which depends on $n$ and $K$ but not on $M$, such that for all $s\in (0,1)$, $p\in (1,\infty)$ with $0<s p < \theta_0$ there exist solutions $u \in W^{1,\infty}_{\text{loc}}(\R^n)$ to \eqref{eq:incl} with $\nabla u \in W^{s,p}(\Omega)$.
\end{thm}

\begin{rmk}
\label{rmk:char}
We remark that a similar result is true on the level of the associated characteristic functions, c.f. Proposition \ref{prop:char_reg}.
\end{rmk}

Let us comment on the result of Theorem \ref{thm:reg}:
\begin{itemize}
\item Theorem \ref{thm:reg} improves the result in \cite{RZZ16} in various directions: Firstly, the regularity exponent $sp$ is \emph{independent} of the position of the Dirichlet boundary data $M\in \R^{n\times n}$ in $\inte(K^{lc})$. Secondly, our argument allows us to deal with both the linearised and the nonlinear problems simultaneously. Thirdly, Theorem \ref{thm:reg} extends the higher regularity result for convex integration solutions from two to three dimensions, which allows us to deal with the particularly interesting model setting of the cubic-to-orthorhombic phase transformation. 
\item The restriction in Theorem \ref{thm:reg} to the dimensions $n=2,3$ is a consequence of the fact that we only prove suitable covering results in two and three dimensions (see Section \ref{sec:covering}). Although we believe that these results remain valid in higher dimensions, we opted to avoid the associated difficulties and to restrict our attention to the physically relevant regimes of $n=2,3$. We stress that apart from the covering results in Section \ref{sec:covering} all other arguments of our analysis are valid in any dimension.
\item In the case of $K=O(2)$, or $K=O(3)$ and zero boundary conditions, our results are worse than the ones by Dacorogna, Marcellini and Paolini \cite{DMP08a,DMP08b,DMP08c,DMP10}. This is mainly due to the fact that our building block constructions are not optimally fitted to the specific problem at hand, but can be used for a larger class of problems. We expect that the results of Dacorogna, Marcellini and Paolini extend to the case of $O(n)$ with boundary data in $\intconv(O(n))$.
\item As already noted, our estimates on the Sobolev exponents are \emph{not} optimal. Showing that in the two-dimensional case the value $sp$ can be \emph{uniformly} chosen independently of the position in $\inte(K^{lc})$ however is a major improvement with respect to \cite{RZZ16}. We believe that in order to obtain qualitatively improved estimates one has to exploit the finer structure of the underlying specific problem (similarly as in the $O(n)$ cases).
\end{itemize}

We emphasize that our overall set-up is more general than the results explained in Theorem \ref{thm:reg} in the sense that we show higher Sobolev regularity of solutions to more general 
differential inclusions of the type \eqref{eq:incl}, if they obey several structural assumptions \ref{item:A1}-\ref{item:A4} which are discussed in Section
\ref{sec:assumptions}. These assumptions are such that they lead to a similar structure as the ones of the model problems from Theorem \ref{thm:reg}.\\

\subsection{Outline of the article}
The remainder of the article is organized as follows: In Section \ref{sec:assumptions} we formulate a collection of assumptions \ref{item:A1}-\ref{item:A4}. In Sections \ref{sec:convex_int}-\ref{sec:quant} we show that these conditions suffice to deduce higher regularity for solutions to \eqref{eq:incl}. To this end, relying on the assumptions \ref{item:A1}-\ref{item:A4}, we formulate and analyse a suitable convex integration algorithm in Section \ref{sec:convex_int} (Algorithm \ref{alg:convex_int}, Lemma \ref{lem:well-def}). In Section \ref{sec:quant} we then complement this with the suitable $L^1$ and $BV$ estimates (Lemmas \ref{lem:L1_in}, \ref{lem:BV_in}), which follow from the requirements in \ref{item:A1}-\ref{item:A4}. This allows us to conclude the existence of higher regularity convex integration solutions in the general framework outlined in Section \ref{sec:assumptions} (Proposition \ref{prop:char_reg}, Theorem \ref{thm:reg_gen}). In Sections \ref{sec:Ex}-\ref{sec:covering}, we discuss the examples (a)-(c) from above, which satisfy the assumptions \ref{item:A1}-\ref{item:A4} and which can hence be dealt with by means of the outlined quantitative convex integration scheme. Here the presentation is split into two main parts: In Section \ref{sec:Ex} we first show the validity of the assumptions \ref{item:A1}-\ref{item:A4} for specific diamond-shaped domains. Then in Section \ref{sec:covering} we extend this to more general domains by presenting several covering strategies in two and three dimensions.

\section{Assumptions on the Matrix Space Geometry of $K$}
\label{sec:assumptions}

In the following we introduce a collection of assumptions which we impose on the set $K
\subset \R^{n\times n}$ in order to be able to construct convex integration
solutions with higher Sobolev regularity.
Here, the typical sets $K$ which we have in mind are given by
\begin{align*}
  K_{n}=O(n)=SO(n)\cup SO(n)
  \diag(-1,1,\dots,1), \ K_{l}= \{e^{(1)},\dots,e^{(m)}\}+ \Skew(n),
\end{align*}
where $e^{(i)},e^{(j)}$ are pairwise symmetrised rank-one connected.
We note that the connected components of $K$ are given by orbits of suitable
representatives under the action of a group $G$, where $G$ may be bounded
($G=SO(n)$) or unbounded ($G=\Skew(n)$).
In this section we introduce a unified description of these and similar settings and
briefly recall fundamental notions such as lamination convexity.

\begin{defi}
\label{defi:lam_conv}
Let $K \subset \R^{n\times n}$. Then $K^{lc}:=\bigcup\limits_{l=0}^{\infty} R_l(K)$, where $R_0(K)=K$ and for $l\geq 1$ 
\begin{align*}
R_l(K)&:= \{M \in \R^{n\times n}: M= \lambda A + (1-\lambda) B \mbox{ for some } \lambda \in (0,1), \\
& \quad \quad \ \rank(A-B)=1, \ A,B \in R_{l-1}(K) \}
\end{align*}
are the laminates of order at most $l$.
\end{defi}

We recall that if $U \subset \R^{n\times n}$ is open, then also $U^{lc}$ is open and that $K^{lc} \subset \conv(K)$, where, in general, the inclusion is strict. In the examples in Section \ref{sec:Ex}, we however additionally also have the opposite inclusion, i.e. the
convex and lamination convex hulls of $K$ coincide. This implies that in all our applications, we will mainly concentrate on the analytical side of the convex integration scheme and do not have to focus on an underlying complicated geometry in matrix space (for instance when verifying the conditions \ref{item:A3tilde}-\ref{item:A4} below).\\

As in \cite{MS} the construction of our solutions to the differential inclusion \eqref{eq:incl} proceeds iteratively by solving auxiliary open inclusion problems, which approximate the inclusion \eqref{eq:incl} increasingly well. To this end, we use the notion of an in-approximation (in a slightly modified version with respect to \cite{MS}): 

\begin{defi}
\label{defi:in_approx}
Let $\{U_k\}_{k\in \N}$ be a sequence of open sets $U_k \subset \R^{n\times n}$. Then the sequence is an \emph{in-approximation of $K$} if the following conditions are satisfied:
\begin{itemize}
\item[(i)] we have the inclusion $U_k \subset U_{k+1}^{lc}$,
\item[(ii)] $U_k \rightarrow K$ in the sense that if $V_k \in U_k$ and $V_k \rightarrow V$, then $V \in K$.
\end{itemize}
\end{defi}

We emphasize that in contrast to \cite{MS}, we do not assume that the sets $U_k$ are necessarily bounded. This is motivated by the desire to deal with the geometrically nonlinear and linear settings simultaneously. In order to compensate this potential lack of compactness in matrix space, we however require boundedness in our construction (see Assumption (A5) below).

\subsection{Assumptions on the differential inclusion}
\label{sec:assume}
We next specify the assumptions which our differential inclusion \eqref{eq:incl} has to obey. They should be read as conditions on the set $K$.

First we require that $G \subset \R^{n\times n}$ is a (continuous Lie) group
acting on $K^{lc} \subset \R^{n \times n}$ through
\begin{align*}
h: K^{lc} \times G \rightarrow K^{lc}.
\end{align*}
In this context, for $g\in G$, $M\in K^{lc}$ we use the shorthand notation $g M= h(M,g) \in K^{lc}$. Also we write $GM$
to denote the orbit of $M \in K^{lc}$ under $h$.

We then impose the following conditions on $K \subset \R^{n \times n}$. In particular, the connected components of $K \subset \R^{n \times n}$ are then assumed to be compatible with the group action:
\begin{enumerate}[label=(A\arabic*)]
\item\label{item:A1}There exist $M_1,\dots,M_m \in K$ such that
\begin{align*}
K=\bigcup_{l=1}^{m} G M_{l},
\end{align*}
and this union is disjoint. Suppose further that there exists $c_{1}>0$ such
  that $\dist(G M_{i}, GM_{j}) \geq c_{1}$ if $i \neq j$.

We assume that the (relative) interior $\inte(K^{lc}) \subset \R^{n\times n}$ of the lamination convex hull of $K$ is non-empty. In the following we do not distinguish between the relative interior (with respect to a subset of $\R^{n\times n}$) and the interior of $\R^{n \times n}$, but always assume that this is used in a consistent way, i.e. we always mean the interior or the relative interior.

\item \label{item:A2} There exist sets $\tilde{U}_{k}^{j}, U_{k}^{j} \subset
  \inte (K^{lc})$ for $k \in \N$ and
  $j \in \{0, \dots, m\}$ such that:
  \begin{align*}
    \tilde{U}_{k}^{j}&= G\tilde{U}_{k}^{j}, U_{k}^{j}= GU_{k}^{j} \mbox{ for all }  j,k , \\
   \inte(K^{lc})&= \bigcup_{k,j} \tilde{U}_{k}^{j}, \\
    \tilde{U}^m_k&=U_{k}^{m}= U_{k+1}^0,
  \end{align*}
  and such that for any $k$,
  \begin{align*}
  \tilde{U}_{k}^{0},\tilde{U}_{k}^{1}, \dots,
  \tilde{U}_{k}^{m}=U_{k+1}^{0}, U_{k+1}^{1}, \dots,
  U_{k+1}^{m}=U_{k+2}^{0}, U_{k+2}^{1}, \dots 
  \end{align*}
  is an in-approximation of $K$.
  Further assume that for every $j$, the sequence $\sup\{\dist(M, K): M \in U_{k}^{j}\}$ tends to zero
    as $k \rightarrow \infty$.
  
\item \label{item:A3tilde} There is a replacement construction associated with
  $\tilde{U}_{k}^{j}$ for the given class $\mathcal{C}$ of domains (e.g. right-angled rectangles, c.f. Section \ref{sec:covering}). That is, there exists $v_1 \in (0,1)$ such that for any
  domain $\tilde{\Omega} \in \mathcal{C}$ and any $M \in \tilde{U}^{j}_{k}, j<m,$ there exists a
  piecewise affine function $u:\tilde{\Omega} \rightarrow \R^n$, such that $\nabla u$
  has only finitely many level sets $\tilde{\Omega}_1, \dots, \tilde{\Omega}_N$ and there exists a set $\tilde{\Omega}_g \subset \tilde{\Omega}$ which is the union of finitely many of the level sets of $\nabla u$ such that
  \begin{align*}
&u(x) = M x \mbox{ on } \partial \tilde{\Omega} \cup (\tilde{\Omega} \setminus \tilde{\Omega}_g),\\
& \nabla u \in \tilde{U}^{j+1}_k  \mbox{ in } \tilde{\Omega}_g,\\
&|\tilde{\Omega}_g| \geq v_1 |\Omega|.
  \end{align*}
  Furthermore, we assume that the level sets $\tilde{\Omega}_{1},\dots, \tilde{\Omega}_{N}$ of $\nabla u$ are
  all (finite unions of) elements in $\mathcal{C}$ and satisfy the following
  perimeter estimates:
  \begin{itemize}
  \item[(i)] For all the level sets contained in $\tilde{\Omega}_{g}$, we obtain an
    estimate by
    \begin{align*}
      \sum_{\overline{\Omega} \subset \tilde{\Omega}_{g}, \overline{\Omega} \in \{\tilde{\Omega}_1,\dots, \tilde{\Omega}_N\}} \Per(\overline{\Omega}) \leq C_{0} \Per(\tilde{\Omega}),
    \end{align*}
    where $C_{0}=C_{0}(j,k)\geq 1$.
    \end{itemize}
For the level sets in $\tilde{\Omega} \setminus \tilde{\Omega}_g$ we assume that one of the following two perimeter estimates holds:
    \begin{itemize}
  \item[(ii)] If $\tilde{\Omega} \in \mathcal{C}^{1}\subset \mathcal{C}$ is in a special
    class (self-similar structure in our application), then also
    $\tilde{\Omega}\setminus \tilde{\Omega}_{g} \in \mathcal{C}^{1}$ and
    \begin{align*}
      \sum_{\overline{\Omega} \subset \tilde{\Omega}\setminus \Omega_{g}, \overline{\Omega} \in \{\tilde{\Omega}_1,\dots, \tilde{\Omega}_N\}} \Per(\overline{\Omega}) \leq C_{2} \Per(\tilde{\Omega}),
    \end{align*}
    where $C_{2}\geq 1$ is a uniform constant. 
  \item[(iii)] If $\tilde{\Omega} \in \mathcal{C}\setminus \mathcal{C}^{1}$, there
    exists a splitting $\tilde{\Omega}^{[1]}\cup \tilde{\Omega}^{[2]}= \tilde{\Omega}\setminus
  \tilde{\Omega}_{g}$ such that $\tilde{\Omega}^{[1]} \in \mathcal{C}^{1}$ and
  $\tilde{\Omega}^{[2]} \in \mathcal{C}$ such that
  \begin{align*}
      \sum_{\overline{\Omega} \subset \tilde{\Omega}^{[1]}, \overline{\Omega} \in \{\tilde{\Omega}_1,\dots, \tilde{\Omega}_N\}} \Per(\overline{\Omega}) &\leq C_{0} \Per(\tilde{\Omega}), \\
      \sum_{\overline{\Omega} \subset \tilde{\Omega}^{[2]}, \overline{\Omega} \in \{\tilde{\Omega}_1,\dots, \tilde{\Omega}_N\}} \Per(\overline{\Omega}) &\leq C_{2} \Per(\tilde{\Omega}),
    \end{align*}
    where the constants $C_0,C_2$ are the ones from (i), (ii). In particular, $C_0=C_0(j,k)$, while $C_2$ is uniform and thus independent of $j,k$.
  \end{itemize}
 \item \label{item:A3} There is a replacement construction associated with
  $U_{k}^{j}$ for the given class $\mathcal{C}$ of domains. That is, there exists $v_1 \in (0,1)$ such that for any
  domain $\tilde{\Omega} \in \mathcal{C}$ and any $M \in U^{j}_{k}$, $j<m$, there exists a
  piecewise affine function $u:\tilde{\Omega} \rightarrow \R^n$, such that $\nabla u$
  has only finitely many level sets and $\tilde{\Omega}_g \subset \tilde{\Omega}$ such that
  \begin{align*}
&u(x) = M x \mbox{ on } \partial \tilde{\Omega} \cup (\tilde{\Omega} \setminus \tilde{\Omega}_g),\\
& \nabla u \in U^{j+1}_k  \mbox{ in } \tilde{\Omega}_g,\\
&|\tilde{\Omega}_g| \geq v_1 |\tilde{\Omega}|.
  \end{align*}
  Furthermore, the level sets $\tilde{\Omega}_{1},\dots, \tilde{\Omega}_{N}$ of $\nabla u$ are
  all (finite unions of) elements in $\mathcal{C}$ and
  \begin{align*}
    \sum_{l=1}^{N} \Per(\tilde{\Omega}_{l}) \leq C_{1} \Per(\tilde{\Omega}).
  \end{align*}
  Here, $C_{1}\geq 1$ is required to be uniform in $k$ and $j$.

\item \label{item:A4} For the construction in \ref{item:A3}, there exist constants
  $0<c_{2}<1$ and $C_{3}>1$
  and sets $\tilde{\Omega}_{g}^{\star} \subset \tilde{\Omega}_g$ with $|\tilde{\Omega}_{g}^{\star}|\geq (1-C_{3} c_{2}^{k})|\tilde{\Omega}_g|$ and such that
\begin{align*}
|\nabla u - M|\leq C_{3} c_{2}^{k} \mbox{ in } \tilde{\Omega}_{g}^{\star}.
\end{align*}
Furthermore, for both \ref{item:A3} and \ref{item:A3tilde} the function $\nabla u$ is uniformly bounded:
\begin{align*}
  |\nabla u| \leq C_{3} \mbox{ in } \tilde{\Omega}.
\end{align*}
\end{enumerate}

We emphasize that in \ref{item:A1}-\ref{item:A4} and also in the sequel, we always use the notion of ``a piecewise affine function" in the sense of ``a piecewise affine and continuous function".\\

Let us comment on the assumptions from \ref{item:A1}-\ref{item:A4}:
\begin{itemize}
\item The requirement in \ref{item:A1} states that, up to group actions, we are interested
in an $m$-well problem. In the present framework we can simultaneously deal with the geometrically nonlinear and linear theory of elasticity by setting $G=SO(n)$ or $G=\Skew(n)$. Here
the unboundedness of $\Skew(n)$ (and of possibly other unbounded continuous Lie groups) leads to several technical issues. 
For instance, the conditions on the distances of the connected components in
  \ref{item:A1} and the convergence of $\sup\{\dist(M,K): M
  \in U_{k}^{j}\}$ to zero in \ref{item:A2} are
  imposed to ensure that for a convergent sequence of matrices $N_{l}$, the
  mapping to the closest connected component $G M_j$ of $K$ remains constant for large $l$. 
  These conditions hold for any in-approximation by convergence and continuity, if
one additionally assumes that $K$ is a closed, bounded set and hence compact.
Using the boundedness assumption in \ref{item:A4}, we may reduce to this
setting, even if $G$ and thus $K$ are unbounded.
The additional assumptions in \ref{item:A1}, \ref{item:A2} dealing with the potential unboundedness of $G$ could hence be omitted. But as they allow us to simplify notation for
instance in Section \ref{sec:L1}, where we then do not have to distinguish
between $K^{lc}$ and $K^{lc} \cap B_{C_3}(0)$, we opted for including them.

If additional constraints, e.g. a trace constraint or a determinant constraint, are taken into account, we always work with the relative interior of $K^{lc}$. All conditions on the \emph{interior} of $K^{lc}$ should then be read as conditions on the \emph{relative interior} of $K^{lc}$. In the examples, which we discuss in Section \ref{sec:Ex}, this enters in the investigation of the (geometrically) linearised hexagonal-to-rhombic and the cubic-to-orthorhombic phase transformations (c.f. Section \ref{sec:lin_trans}).

\item Condition \ref{item:A2} provides an in-approximation which is invariant under the action of the symmetry group.
  The sets $\tilde{U}_{k}^{j}$ in \ref{item:A2} are auxiliary sets which ensure
  that we can start with any initial datum $M \in \inte(K^{lc})= \bigcup_{k,j}
  \tilde{U}_{k}^{j}$.
  Using the construction in \ref{item:A3tilde}, after a small number of steps we
  then obtain that $\nabla u \in U^{j_0}_{k_0}$ for some $j_0, k_0$.
  We remark that the constant $C_0\geq 1$ in \ref{item:A3tilde} may depend on $k$ and $j$ and thus on the initial data.
  However, on the sets $U_{k}^{j}$ we require \emph{uniform} estimates, which are
  essential to obtain convergence in $W^{s,p}$ with $s,p$ \emph{independent} of the
  initial data.
\item 
The constructions \ref{item:A3tilde} and \ref{item:A3} allow us to carry out the convex
integration scheme of Müller and {\v{S}}ver{\'a}k \cite{MS}. In addition to the requirements which are also needed for a non-quantitative convex integration scheme, the perimeter estimates in \ref{item:A3} and \ref{item:A3tilde} provide a first quantitative ingredient.
Verifying it in our applications requires
fine control of the scales involved as discussed in Section \ref{sec:BV} (and Section \ref{sec:covering}). 
In particular, in the model situations (a)-(c) from Section \ref{sec:intro} we have to consider various different scenarios for the underlying geometry of the covering, which we have hence formalized in splitting the assumption \ref{item:A3tilde} into the three cases \ref{item:A3tilde} (i)-(iii).

\item The last requirement \ref{item:A4} ensures very strong control over the group action in our
  convergence estimate. That is, not only does $\dist(G \nabla u_{k}, M)$ tend to
  zero, but also $\nabla u_{k}$ forms a Cauchy sequence with an exponential
  convergence rate in $L^{1}$. This provides the second main quantitative ingredient in our scheme. It is used in combination with the BV estimate to derive quantitative higher regularity estimates by interpolation \cite{CDDD03}.
\end{itemize}

\subsection{The higher regularity result}
\label{sec:thm2}

Under the assumptions from Section \ref{sec:assume}, we can then construct convex integration solutions
using the algorithm described in Section \ref{sec:convex_int}. More precisely, we show that the conditions collected in Section \ref{sec:assume} imply the following higher regularity result:

\begin{thm}
\label{thm:reg_gen}
Let $\Omega \subset \R^n$ be a bounded domain with $\Omega \in \mathcal{C}$ and
suppose that the conditions formulated in Section \ref{sec:assume} hold. Then there
exists $ \theta_0 \in (0,1)$ depending only on the dimension $n\in \N$ and the constants $C_1, C_2, v_1, c_2$ in \ref{item:A1}-\ref{item:A4} such that for all values $s\in (0,1)$, $p\in (1,\infty)$ with $0<s p < \theta_0$ and for any $M_0 \in \inte(K^{lc})$ there exist solutions $u \in W^{1,\infty}_{\text{loc}}(\R^n)$ of 
\begin{align}
\label{eq:incl_K}
\begin{split}
\nabla u &\in K \mbox{ in } \Omega,\\
\nabla u &= M_0 \mbox{ in } \R^3 \setminus \Omega,
\end{split}
\end{align}
such that $\nabla u \in W^{s,p}_{loc}(\R^n)$. Moreover, for some constant 
$C>1$ which depends on $C_0,C_1,C_2,C_3,v_1,c_2,n,\Omega$, we have
\begin{align*}
\|\nabla u\|_{W^{s,p}(\Omega)} \leq C.
\end{align*}
\end{thm}

Similarly as in Theorem \ref{thm:reg} we here show that
the constructed solutions to the differential inclusion \eqref{eq:incl_K} exhibit $W^{s,p}$ regularity, where the achieved regularity exponent is
\emph{independent} of the choice of the initial data (though the size of the norm may
depend on it). In particular, this improves the results from \cite{RZZ16} significantly, if the Dirichlet data in \eqref{eq:incl} are close to the boundary of the corresponding convex hulls. 
\\

In Section \ref{sec:Ex} we  will show that for our model cases (a)-(c) from Section \ref{sec:intro} the conditions \ref{item:A1}-\ref{item:A4} are satisfied. Here in particular the derivation of \emph{uniform} bounds for $C_1,C_2>1$ requires careful covering arguments, which could be significantly simplified if the uniformity of the result was given up. An application of Theorem \ref{thm:reg_gen} to the model cases (a)-(c) from Section \ref{sec:intro} then entails Theorem \ref{thm:reg}.

\section{The Convex Integration Algorithm}
\label{sec:convex_int}

In the whole following section we always suppose that the assumptions of Section \ref{sec:assume} hold.
With these at hand, we proceed to the definition of our quantitative convex integration algorithm.
Here, we first provide a construction for the case when $\Omega$ is in a
  given class of domains $\mathcal{C}$. This is then extended in Section
  \ref{lem:Lip} to general Lipschitz domains.

\subsection{The formulation of the convex integration algorithm}
We construct solutions to the differential inclusion at hand by abiding to the following construction rule:

\begin{alg}
\label{alg:convex_int}
Let $\Omega \subset \R^n$ be such that $\Omega \in \mathcal{C}$. Let $U_{k}^j$ and $\tilde{U}_{k}^{j}$ be as in (A2). Let $M_0 \in \inte(K^{lc})$.
\begin{itemize}
\item[(a)] Data: For $k\in \N\cup\{0\}$ we consider tuples $(u_k,\hat{\Omega}_k, l_k, j_k, q_k)$, where
\begin{itemize}
\item $u_k: \Omega \rightarrow \R$ is a piecewise affine, uniformly (in $k$) bounded Lipschitz function.
\item $\hat{\Omega}_k=\{\Omega^{k,1},\dots,\Omega^{k,i_k}\} \subset \mathcal{C}$ is a collection of (up to null-sets) disjoint sets covering $\Omega$. We have that $\nabla u_k|_{\Omega^{k,j}}=const$ for all $j\in\{1,\dots,i_k\}$. 
\item $(l_k, j_k): \hat{\Omega}_k \rightarrow (\N\cup\{-1,0\}) \times
  \{0,\dots,n-1\}$ denotes the depth of the iteration. For any $\tilde{\Omega} \in
  \hat{\Omega}_k$ we construct $(l_k,j_k)$ such that either $\nabla u_k \in U_{l_k}^{j_k}$ or
  $\nabla u_k \in \tilde{U}_{m_0}^{j_k}$, where $m_0 \in \N$ is defined in Step (b).
\item $q_k:\hat{\Omega}_k \rightarrow \N\cup\{0\} $ denotes the number of times the function $u_0$ (defined in (b) below) has been modified on a given domain.
\end{itemize}
\item[(b)] Initialization: We set 
\begin{align*}
  &u_0(x) := M_0 x,  \ \hat{\Omega}_0:= \{\Omega\},
\ l_0:=-1, \ q_0:=0. 
\end{align*}
Further, we define 
\begin{align*}
  m_0&:=\min \{m\in \N: M_0 \in \tilde{U}_{m}^0\}, \\
  j_0&:=0.
\end{align*}

\item[(c)] Replacement construction: Assume now that for $k\geq 0$ the tuple
  $(u_k, \hat{\Omega}_k, l_k,j_k, q_k)$ is given and let
  $\Omega^{k,i}\in \hat{\Omega}_k$. We then distinguish two situations:
\item[($c_1$)] 
Assume that $l_k(\Omega^{k,i})=-1$. Let $j_{k,i}:=j_k(\Omega^{k,i})$ and suppose that on the domain $\Omega^{k,i}$ it holds
that $\nabla u_k \in \tilde{U}_{m_0}^{j_{k,i}}$. We apply the replacement
construction from Assumption \ref{item:A3tilde}. This returns
\begin{itemize}
\item[(i)] a piecewise affine function $w: \Omega^{k,i} \rightarrow \R^n$ such that on a subset $\Omega^{k,i}_g \subset \Omega^{k,i}$, which consists of a union of
elements of $\mathcal{C}$, and which satisfies $|\Omega^{k,i}_g|\geq v_1 |\Omega^{k,i}|$, it holds that  
\begin{align*}
\nabla w(x) \in \tilde{U}_{m_0}^{j_{k,i}+1}  \mbox{ for a.e. } x \in \Omega^{k,i}_g. 
\end{align*}
Moreover, $w(x) = u_k(x) \mbox{ for a.e. } x\in (\Omega^{k,i} \setminus \Omega^{k,i}_g) \cup \partial \Omega^{k,i}$.
\item[(ii)] a collection $\hat{\Omega}_{k+1,i}:=\{\Omega^{k+1,1},\dots,\Omega^{k+1,r_i}\} \subset \mathcal{C}$ of pairwise (up to null sets) disjoint domains, which are the level sets of $\nabla w$.
\end{itemize}
We define the following set-functions (the remaining ones for the tuple $(u_k,
\hat{\Omega}_k, l_k, j_k, q_k)$ do not differ from those which
occur in the case ($c_2$) and are hence given in a unified way  below, c.f. \eqref{eq:iterstep}):
\begin{align*}
&l_{k+1,i}: \hat{\Omega}_{k+1,i} \rightarrow \N\cup\{0,-1\}, \\
  &l_{k+1,i}(\Omega^{k+1,r}) =
    \begin{cases}
      m_0+1 &\mbox{ if } j_{k+1}(\Omega^{k+1,r})=0 \mbox{ but } j_{k}(\Omega^{k,i})\neq 0,\\
l_k(\Omega^{k,i}) &\mbox{ else},
    \end{cases}
\end{align*}
where the function $j_{k+1}$ is defined in \eqref{eq:iterstep} at the end of the algorithm after step ($c_2$). 

\item[($c_2$)] 
Assume that $l_{k,i}:=l_k(\Omega^{k,i}) \neq -1$ and abbreviate $j_{k,i}:=j_k(\Omega^{k,i})$. Suppose that on the domain $\Omega^{k,i}$ it holds
that 
\begin{align*}
\nabla u_k \in U_{l_{k,i}}^{j_{k,i}}. 
\end{align*}
On $\Omega^{k,i}$ we thus apply the replacement construction from Assumption \ref{item:A3}. This returns
\begin{itemize}
\item[(i)] a piecewise affine function $w:\Omega^{k,i} \rightarrow \R^n$ and a subset $\Omega^{k,i}_{g} \subset \Omega^{k,i}$, which is a union of 
elements from $\mathcal{C}$, which satisfies $|\Omega^{k,i}_{g}|\geq v_1|\Omega^{k,i}|$, and on which 
\begin{align*}
\nabla w \in U_{l_{k,i}}^{j_{k,i}+1}.
\end{align*} 
For a.e. $x\in \Omega^{k,i} \setminus \Omega^{k,i}_g$ it holds that 
\begin{align*}
\nabla w = \nabla u_k \in  U_{l_{k,i}}^{j_{k,i}}.
\end{align*}
Furthermore, $w(x) = u_k(x)$ for $x\in \partial \Omega^{k,i}$.
\item[(ii)] a collection $\hat{\Omega}_{k+1,i}:=\{\Omega^{k+1,1},\dots,\Omega^{k+1, r_i}\}$ of pairwise (up to null-sets) disjoint domains, which are level sets of $\nabla w$.
\end{itemize}
We then define the set function
\begin{align*}
&l_{k+1,i}: \hat{\Omega}_{k+1,i} \rightarrow \N\cup\{0,-1\}, \\
  &l_{k+1,i}(\Omega^{k+1,r}) =
    \begin{cases}
      l_{k}(\Omega^{k,i})+1  &\mbox{ if } j_{k+1}(\Omega^{k+1,r})=0 \mbox{ but } j_{k}(\Omega^{k,i})\neq 0,\\
      l_k(\Omega^{k,i}) &\mbox{ else}.
    \end{cases}
\end{align*}
\end{itemize}
In both cases ($c_1$) and ($c_2$) we set
\begin{eqnarray}
  \label{eq:iterstep}
&j_{k+1,i}: \hat{\Omega}_{k+1,i} \rightarrow \{0,\dots,m\}, \nonumber\\
  &j_{k+1,i}(\Omega^{k+1,r}) =
    \begin{cases}
      j_{k}(\Omega^{k,i})+1 \mbox{ mod } m & \mbox{ if } \Omega^{k+1,r}\subset \Omega^{k,i}_g,\nonumber\\
      j_{k}(\Omega^{k,i}) &\mbox{ else},
    \end{cases}\nonumber\\
&q_{k+1,i}: \hat{\Omega}_{k+1,m} \rightarrow \R, \\
  &q_{k+1,i}(\Omega^{k+1,r}) =
    \begin{cases}
      q_{k}(\Omega^{k,i})+1 &\mbox{ if } \Omega^{k+1,i} \subset \Omega^{k,i}_g,\nonumber\\
      q_k(\Omega^{k,i}) &\mbox{ else}.
    \end{cases}
\end{eqnarray}
and set $u_{k+1}|_{\Omega^{k,i}}:=w$. 
Finally, for $f\in\{j,l,q\}$ we define
\begin{align*}
&\hat{\Omega}_{k+1}:= \bigcup\limits_{j=1}^{i_k} \hat{\Omega}_{k+1,j},\\
& f_{k+1}: \hat{\Omega}_{k+1} \rightarrow \R,\ f_{k}(\Omega_{m}) := f_{k,i}(\Omega_m) \mbox{ if } \Omega_m \in \hat{\Omega}_{k,i}.
\end{align*}
\end{alg}

\begin{rmk}
\label{rmk:alg}
We make the following observations: Assuming that Algorithm \ref{alg:convex_int} is well-defined and with slight abuse of notation, identifying the set functions $l_k, q_k$ with functions on $\Omega$ by setting $l_k(x)=l_{k}(\Omega^{k,i})$ for $x\in \Omega^{k,i}$ (which is a.e. well-defined), we have that
\begin{itemize}
\item $l_k$ is an increasing function in $k \in \N$,
\item $q_k$ is an increasing function in $k\in \N$.
\end{itemize}
\end{rmk}

\subsection{Well-definedness of the algorithm}

In order to construct the desired solutions of the differential inclusion \eqref{eq:incl_K}, we seek to follow the prescription of Algorithm \ref{alg:convex_int}. To this end, we however first have to ensure its well-definedness. In order to simplify notation, we thus introduce the \emph{descendent} of a domain $\widetilde{\Omega} \in \hat{\Omega}_k$.

\begin{defi}
Let $\widetilde{\Omega}\in \hat{\Omega}_k$ for some $k\geq 0$. Then we say that
$\Omega \in \hat{\Omega}_l$ with $l>k$ is a \emph{descendant of $\widetilde{\Omega}$}, if 
$\Omega \subset \widetilde{\Omega}$. We denote the set of all descendants of $\widetilde{\Omega}$ by $\mathcal{D}(\widetilde{\Omega})$. 
\end{defi}

With this notation available, we discuss the well-definedness of Algorithm \ref{alg:convex_int}:

\begin{lem}
\label{lem:well-def}
Let $\Omega \subset \R^n$ with $\Omega \in \mathcal{C}$. Let further $M_0\in \inte(K^{lc})$ and construct the tuple $(u_k, \hat{\Omega}_k, l_k,j_k,q_k)$ as in Algorithm \ref{alg:convex_int}. 
Then, beginning with the initialization step, only the cases ($c_1$), ($c_2$) can occur in the course of the algorithm. In particular, the cases ($c_1$), ($c_2$) cover all possible situations. Moreover, the function $u_{k+1}$ defined by restriction on the sets $\widetilde{\Omega} \in \hat{\Omega}_k$ is piecewise affine, in particular it is continuous.
Furthermore, if $q_{k}(\tilde{\Omega}) \geq m$, then $\nabla u_{k}|_{\tilde{\Omega}}$ is in case ($c_{2}$).
\end{lem}

\begin{proof}  
We note that for $k=0$, we have that $l_0=-1$ and $\nabla u =M_0 \in
\tilde{U}_{m_0}^0$. Hence, we start the algorithm in the case ($c_1$).
The Assumption \ref{item:A3tilde} then implies that this remains unchanged as long as
$l_k(\tilde{\Omega})=-1$. Thus, initially, the claim is true.
\\
It therefore remains to show the induction step, i.e. that if the claim is true at the $k$-th iteration step, it is then also true for the iteration step $k+1$.

In order to observe this, let $\tilde{\Omega}_g := \bigcup\limits_{j=1}^{i_k}\Omega^{k,j}_g $.
By construction, for a.e. $x\in \Omega \setminus \tilde{\Omega}_g$ it holds that $\nabla
u_{k+1}(x)=\nabla u_k(x)$. Since on $\Omega \setminus \tilde{\Omega}_g$ we have $j_{k+1}=j_k, l_{k+1}=l_k, q_{k+1}=q_k$ in Algorithm \ref{alg:convex_int}, we may invoke the inductive hypothesis and conclude that on $\Omega \setminus \tilde{\Omega}_g$ the claim is true. 

Hence, we only need to consider sets $\tilde{\Omega} \in \tilde{\Omega}_{g}\cap
\hat{\Omega}_k$. Fix such a set and abbreviate $j_k:=j_{k}(\tilde{\Omega})$, $l_k:=l_k(\tilde{\Omega})$.
If on $\tilde{\Omega}$ the case $(c_2)$ occurs, i.e. if $l_{k}>-1$ and if for
a.e. $x\in \tilde{\Omega}$ we have $\nabla u_k(x) \in U^{j_k}_{l_k}$, then by
Assumption \ref{item:A3} for almost every $x\in \tilde{\Omega}$
\begin{align*}
  \nabla u_{k+1}(x)=\nabla w(x)  \in U^{j_k+1}_{l_k}=U^{j_{k+1}}_{l_{k+1}},
\end{align*}
where we used that $U^{(m-1)+1}_{l_k}= U^{m}_{l_k}=U^{0}_{l_k+1}$. Hence for all $\bar{\Omega}\in \mathcal{D}(\tilde{\Omega})\cap \hat{\Omega}_{k+1}$ we have that in the iteration step $k+1$ we are in the case ($c_2$). This in particular shows that once a domain reaches the case ($c_2$) its descendants will always remain in this case.

If in step $k$ and on the domain $\tilde{\Omega}$ we are in case ($c_1$), i.e.
for a.e. $x\in \tilde{\Omega}$ it holds $\nabla u_k(x) \in
\tilde{U}^{j_k}_{m_0}$, then an application of Assumption \ref{item:A3tilde} in Algorithm \ref{alg:convex_int} ensures that for a.e. $x\in \tilde{\Omega}$
\begin{align*}
  \nabla u_{k+1}(x)=\nabla w(x)  \in \tilde{U}^{j_k+1}_{m_0}.
\end{align*}
If $j_k+1 <m$, we thus obtain $ \nabla u_{k+1}(x) \in \tilde{U}^{j_{k+1}}_{m_0}$. Combined with the prescription of $l_{k+1}$ we conclude that in step $k+1$ we are again in the case ($c_1$).
If instead $j_k+1=m$, then 
\begin{align*}
  \tilde{U}^{(m-1)+1}_{m_0}= \tilde{U}^{m}_{m_0}=U_{m_0}^m=U_{m_0 +1}^0=U_{l_{k+1}}^0.
\end{align*}
As in this case for any $\bar{\Omega} \in \mathcal{D}(\tilde{\Omega})\cap \hat{\Omega}_{k+1}$ it holds $l_{k+1}(\bar{\Omega})=m_0+1$, in step $k+1$ we are in the case ($c_2$).
This concludes the induction argument showing that the cases ($c_1$), ($c_2$) cover all possibilities which appear in the iteration Algorithm \ref{alg:convex_int}.

The fact that $u_{k+1}$ is piecewise affine and continuous follows from the assumptions \ref{item:A3tilde}, \ref{item:A3} concerning the existence of piecewise affine replacement constructions.
Finally, we note that if $q_{k}(\tilde{\Omega})\geq m$, then
  $l_{k}(\tilde{\Omega}) \neq -1$, which implies that $\nabla u_{k}$ is in case
  ($c_{2}$) and by the above considerations remains in the case ($c_{2}$) in all
  subsequent steps.
\end{proof}

\section{Quantitative $L^1$ and $BV$ Estimates}
\label{sec:quant}

As in the previous section, in the whole of the following section we always suppose that the assumptions of Section \ref{sec:assume} hold.
Given these, we derive the desired quantitative estimates of Theorem
\ref{thm:reg_gen} for the convex integration solutions, which were constructed in
Algorithm \ref{alg:convex_int}. Here we argue in two steps: In Proposition \ref{prop:char_reg} we first prove the result for the underlying characteristic functions (c.f. Definition \ref{defi:character}). In Section \ref{sec:def} we then extend it to the deformation itself. In both cases, we rely on an interpolation result between suitable $L^p$ and $BV$ estimates. The critical Sobolev exponent is determined by the competition between the convergence of the $L^p$ and the growth of the $BV$ norms. For the convenience of the reader, we recall a variant of the interpolation result of \cite{CDDD03} in the form in which we are going to use it in the sequel (c.f. also Theorem 2 in \cite{RZZ16}):

\begin{prop}[\cite{CDDD03}, Remark 2.2 in \cite{RZZ16}]
\label{prop:CDDD}
Let $u\in L^{\infty}(\R^n)\cap BV(\R^n)\cap L^1(\R^n)$ and let $\theta_0\in (0,1)$ and $\tilde{\theta}=\theta_0 q^{-1}$ for some $q\in (1,\infty)$. Then,
\begin{align}
\label{eq:CDDD}
\|u\|_{W^{\tilde{\theta},q}} \leq \|u\|_{L^{\infty}(\R^n)}^{1-\frac{\tilde{\theta}}{\theta_0}}
\left( \|u\|_{L^{1}}^{1-\theta_0} \|u\|_{BV(\R^n)}^{\theta_0} \right)^{\frac{\tilde{\theta}}{\tilde{\theta}_0}}.
\end{align}
\end{prop}

With this at hand, we introduce the \emph{characteristic functions associated with the
connected component of $K$} introduced in \ref{item:A1}.

\begin{defi}[Characteristic functions]
\label{defi:character}
Let $\Omega \subset \R^n$ with $\Omega \in \mathcal{C}$ and let $M_1,\dots, M_m$ be the matrices from \ref{item:A1}. Let $u_k:\Omega \rightarrow \R^n$ be the mapping obtained in the $k$-th iteration step of the Algorithm \ref{alg:convex_int}. Then, for $j=1,\dots,m$ we define \emph{the characteristic functions $\chi_k^{(j)}:\R^n \rightarrow \{0,1\}$ at step $k \in \N$ associated with $M_1,\dots, M_m$} by
\begin{align*}
\chi_k^{(j)}(x)=
\left\{
\begin{array}{ll}
1 &\mbox{ if } x \in \Omega \mbox{ and } \dist(\nabla u_k(x), GM_j)< \dist(\nabla u_k(x), \bigcup_{i\neq j} G M_{i}),\\
0 &\mbox{ else},
\end{array} \right.
\end{align*}
if $j<m$, and $\chi_k^{(m)}(x)=1- \sum\limits_{j=1}^{m-1}\chi_k^{(j)}(x)$.\\
Similarly, for a given solution $ u \in W^{1,\infty}(\Omega) $ to \eqref{eq:incl} we define \emph{the associated underlying characteristic functions by}
\begin{align*}
\chi_u^{(j)}(x)=
\left\{
\begin{array}{ll}
1 &\mbox{ if } x \in \Omega \mbox{ and } \dist(\nabla u(x), GM_j)< \dist(\nabla u(x),\bigcup_{i\neq j} G M_{i} ),\\
0 &\mbox{ else},
\end{array} \right.
\end{align*}
if $j<m$, and $\chi_u^{(m)}(x)=1-\sum\limits_{j=1}^{m-1}\chi_u^{(j)}(x)$.
\end{defi}

With this notation at hand, we can formulate a replacement result for the level sets of $\nabla u$:

\begin{prop}
\label{prop:char_reg}
Let $\Omega \subset \R^n$ with $\Omega \in \mathcal{C}$. There exists $ \theta_0
\in (0,1)$ such that for all values $s\in (0,1)$, $p\in (1,\infty)$ with $0<s p < \theta_0$ and for any $M_0 \in  \inte(K^{lc})$ there exist solutions $u \in W^{1,\infty}(\Omega)$ of \eqref{eq:incl_K}
with underlying characteristic functions $\chi^{(1)}_u,\dots,\chi^{(m)}_u \in W^{s,p}(\Omega)$.
The constant $\theta_0$ only depends on $n\in \N$ and the set $K$, but not on the choice of $M_0\in  \inte(K^{lc})$.
\end{prop}

\begin{rmk}
\label{rmk:frac}
We remark that as a direct consequence of Proposition \ref{prop:char_reg} we obtain a bound on the packing dimension of the sets $\Omega_i:=\{x\in \Omega: \chi_u^{(i)}(x)=1\}$. This follows as in Remark 2.3 in \cite{RZZ16}.
\end{rmk}

Proposition \ref{prop:char_reg} will be derived as a consequence of the $L^1$ and $BV$ estimates from Sections \ref{sec:L1} and \ref{sec:BV}.

\subsection{The $L^1$ bound}
\label{sec:L1}

We first discuss the $L^1$ estimate. In contrast to the piecewise affine convex
integration scheme that was used in \cite{RZZ16}, none of the level sets of the
gradient becomes ``stable" after a finite number of iteration steps. In spite of
this, the replacement constructions of \ref{item:A3tilde}, \ref{item:A3} allow us to conclude similar decay properties for the $L^1$ norm of differences of the characteristic functions $\chi_k^{(1)}, \dots, \chi_k^{(m)}$ (see Figure \ref{fig:tree}).

\begin{lem}[$L^1$ control]
\label{lem:L1_in}
Let $\Omega \subset \R^n$ with $\Omega \in \mathcal{C}$ and let the constants $c_{2}\in(0,1)$ and $v_{1}\in(0,1)$ be as in \ref{item:A3tilde}-\ref{item:A4}. 
Further suppose that $u_k:\Omega \rightarrow \R^n$ is the deformation, which is obtained in the $k$-th step of the convex integration scheme from Algorithm \ref{alg:convex_int}. 
Then, for $i\in\{1,\dots,m\}$ and for some constant $C>1$ depending only on the
uniform constants in \ref{item:A1}-\ref{item:A4}, it holds that
\begin{align*}
\|\chi_{k+1}^{(i)}-\chi_{k}^{(i)}\|_{L^1(\Omega) }
\leq C \tilde{c}^{k}|\Omega|,
\end{align*}
where $\tilde{c}: = v_{1}c_{2}^{1/m}+(1-v_{1})\in (0,1).$
\end{lem}

\begin{figure}[t]
  \centering
 \includegraphics[width=0.5\linewidth, page=1]{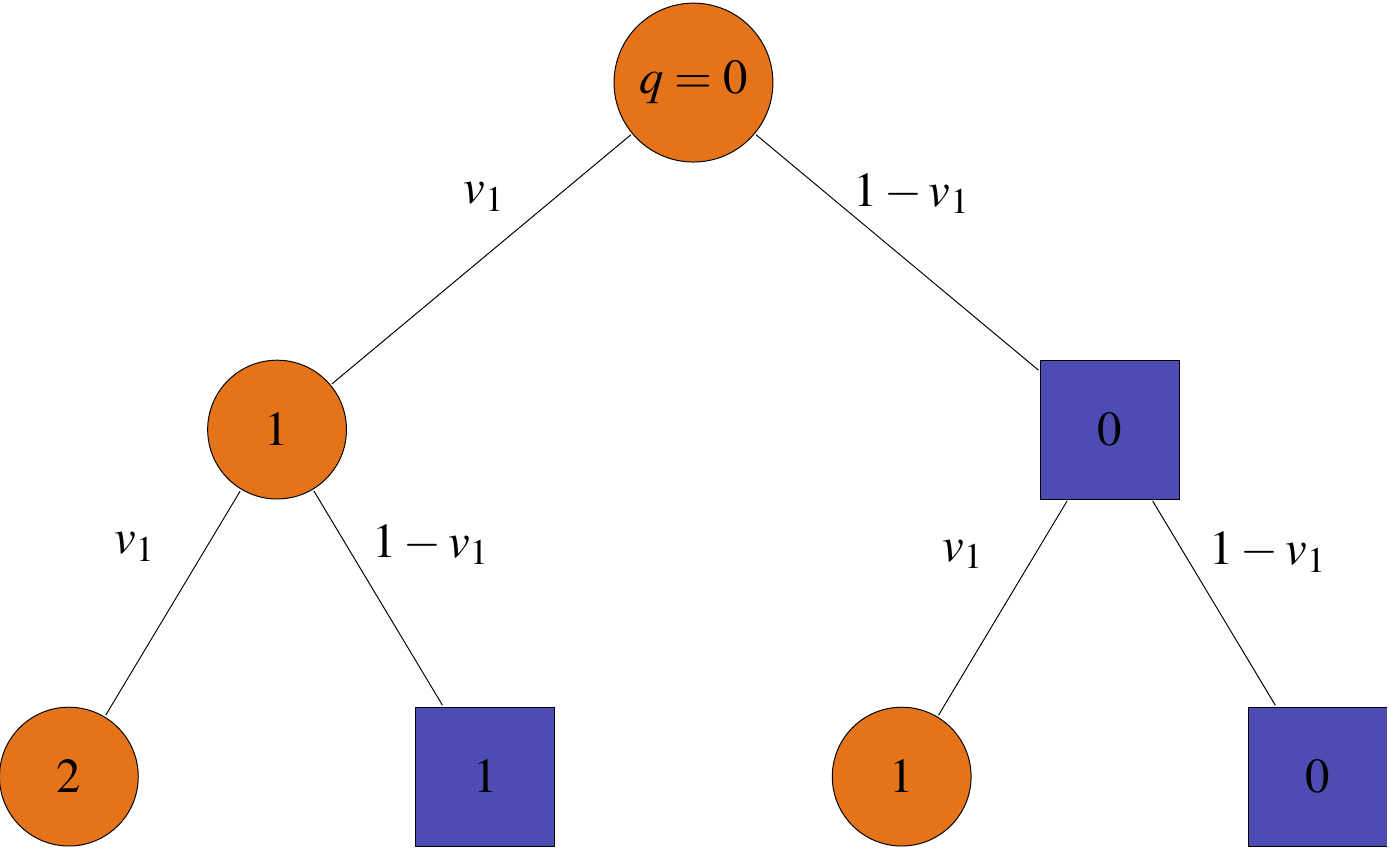}
  \caption{The binomial tree, which yields the desired $L^1$ bound in the case that the volume fractions are always given by $v_1$ and $1-v_1$. In this case the distribution of good versus bad sets at a certain stage would be binomially distributed. In the general case in which the ``good" volume fraction is only bounded below by $v_1$, the distribution is not exactly binomial, but still retains similar properties (see Lemma \ref{lem:L1_in}).}
  \label{fig:tree}
\end{figure}

\begin{proof}
    We claim that it suffices to show that
    \begin{align}
      \label{eq:3}
      \|\chi_{k+1}^{(i)}- \chi_{k}^{(i)}\|_{L^{1}(\tilde{\Omega})} \leq C c_{2}^{q_{k}(\tilde{\Omega})/m} |\tilde{\Omega}|
    \end{align}
    for each $\tilde{\Omega} \in \hat{\Omega}_{k}$.

    Indeed, assume that \eqref{eq:3} holds and sum over all $\tilde{\Omega} \in
    \hat{\Omega}_{k}$ to obtain that
    \begin{align*}
      \|\chi_{k+1}^{(i)}-\chi_{k}^{(i)}\|_{L^{1}(\Omega)}= \sum_{\tilde{\Omega} \in
      \hat{\Omega}_{k}}\|\chi_{k+1}^{(i)}-\chi_{k}^{(i)}\|_{L^{1}(\tilde{\Omega})} \\
      \leq C \sum_{\tilde{\Omega} \in
      \hat{\Omega}_{k}} c_{2}^{q_{k}(\tilde{\Omega})/m}|\tilde{\Omega}| =:C\mathcal{E}_{k}[c_{2}^{q/m}].
    \end{align*}
    We here interpret our partition of $\Omega$ into smaller sets as a ``Bernoulli-type"
    experiment with probabilities equal to the volume fractions (c.f. Figure
    \ref{fig:tree}) and think of $\mathcal{E}_{k}[c_{2}^{q/m}]$ as the expected
    value of the variable $c_2^{q_k(\cdot)/m}$.
    We remark that if our constructions in \ref{item:A3tilde}, \ref{item:A3} always used the same volume fraction $v=v_1$ in each
    step, then $q_k$ would be binomially distributed and hence
    \begin{align*}
      \mathcal{E}_{k}[c_{2}^{q/m}]= \sum_{l=0}^{k} {k \choose l} v^{l}_1(1-v_1)^{k-l} c_{2}^{l/m}|\Omega|= (1-v_1+c^{1/m}_{2}v_1)^{k}|\Omega| \leq \tilde{c}^{k}|\Omega|.
    \end{align*}
    In the present slightly more general case, this identity is replaced by an estimate which
    follows by first noting that
    \begin{align*}
      \mathcal{E}_{0}[c_{2}^{q/m}]= \mathcal{E}_{0}[1]=|\Omega|,
    \end{align*}
    and by secondly showing that 
    \begin{align}
      \label{eq:6}
      \mathcal{E}_{k+1}[c_{2}^{q/m}] \leq \tilde{c} \mathcal{E}_{k}[c_{2}^{q/m}],
    \end{align}
    which combined imply the result by induction on $k$.

    Let thus $k \in \N$, then
\begin{align*}
      \mathcal{E}_{k+1}[c_{2}^{q/m}] &= \sum_{\tilde{\Omega} \in
      \hat{\Omega}_{k+1}} c_{2}^{q_{k+1}(\tilde{\Omega})/m}|\tilde{\Omega}| 
      =\sum_{\overline{\Omega} \in
        \hat{\Omega}_{k}} \sum_{\tilde{\Omega} \in \mathcal{D}(\overline{\Omega})\cap \hat{\Omega}_{k+1}} c_{2}^{q_{k+1}(\tilde{\Omega})/m}|\tilde{\Omega}|\\
      &= \sum_{\overline{\Omega} \in
        \hat{\Omega}_{k}} \left(\sum_{\tilde{\Omega} \in \hat{\Omega}_{k+1}: \tilde{\Omega} \subset \overline{\Omega}_{g}}c_{2}^{q_{k+1}(\tilde{\Omega})/m}|\tilde{\Omega}| +
        \sum_{\tilde{\Omega} \in \hat{\Omega}_{k+1}: \tilde{\Omega} \subset (\overline{\Omega}\setminus \overline{\Omega}_{g})}c_{2}^{q_{k+1}(\tilde{\Omega})/m}|\tilde{\Omega}| \right) \\
       &= \sum_{\overline{\Omega} \in
        \hat{\Omega}_{k}} \left(\sum_{\tilde{\Omega} \in \hat{\Omega}_{k+1}: \tilde{\Omega} \subset \overline{\Omega}_{g}}c_{2}^{(q_{k}(\overline{\Omega})+1)/m}|\tilde{\Omega}| +
         \sum_{\tilde{\Omega} \in \hat{\Omega}_{k+1}: \tilde{\Omega} \subset (\overline{\Omega}\setminus \overline{\Omega}_{g})}c_{2}^{q_{k}(\overline{\Omega})/m}|\tilde{\Omega}| \right) \\
      &= \sum_{\overline{\Omega} \in \hat{\Omega}_{k}}  c_{2}^{q_{k}(\overline{\Omega})/m} \left(c_{2}^{1/m} |\overline{\Omega}_{g}| + |\overline{\Omega}\setminus \overline{\Omega}_{g}| \right) \\
      &= \sum_{\overline{\Omega} \in \hat{\Omega}_{k}} c_{2}^{q_{k}(\overline{\Omega})/m} |\overline{\Omega}| \left(c_{2}^{1/m} \frac{|\overline{\Omega}_{g}|}{|\overline{\Omega}|} + 1 \cdot \frac{ |\overline{\Omega}\setminus \overline{\Omega}_{g}|}{|\overline{\Omega}|}\right) \\
&\leq \sum_{\overline{\Omega} \in \hat{\Omega}_{k}} c_{2}^{q_{k}(\overline{\Omega})/m} |\overline{\Omega}| 
\max\limits_{v\geq v_1} \left( c_2^{1/m} v + (1-v) \right) \\      
&= \sum_{\overline{\Omega} \in \hat{\Omega}_{k}} c_{2}^{q_{k}(\overline{\Omega})/m} |\overline{\Omega}| 
\left( c_2^{1/m} v_1 + (1-v_1) \right)\\
      & = \tilde{c} \mathcal{E}_{k}[c_{2}^{q/m}],
    \end{align*}
where we used that $|\overline{\Omega}_g|\geq v_1 |\overline{\Omega}|$ by the conditions in \ref{item:A3tilde}, \ref{item:A3}.

It remains to present the argument for \eqref{eq:3} which is deduced from \ref{item:A4} in the following way:
    Let $l_{0} \in \N$ to be fixed later and consider the sets $\tilde{\Omega}$
    with $q_{k}(\tilde{\Omega})=l \leq l_{0}$.
    Then, 
    \begin{align*}
      \|\chi_{k+1}^{(i)}- \chi_{k}^{(i)}\|_{L^{1}(\tilde{\Omega})} \leq |\tilde{\Omega}| \leq C c_{2}^{l/m} |\tilde{\Omega}|,
    \end{align*}
    provided $C$ is chosen such that $C c_{2}^{l_{0}/m}\geq 1$.
    
    We may hence focus on sets such that $q_{k}(\tilde{\Omega})\geq l_{0}$.
    We claim that, if $l_{0}$ is chosen sufficiently large, the conditions \ref{item:A1},
    \ref{item:A2} and \ref{item:A4} ensure the following implication: 
    \begin{align}
    \label{eq:claim_a}
     \mbox{If } |\nabla u_{k+1}-\nabla u_k| \leq C c_{2}^{k+1} \mbox{on } \tilde{\Omega}_{g}^{\star}, \mbox{ then } \chi_{k+1}^{(i)}= \chi_{k}^{(i)} \mbox{ on } \tilde{\Omega}_{g}^{\star}.
    \end{align}
    Using \ref{item:A4}, we then estimate 
     \begin{align*}
     \|\chi_{k+1}^{(i)}- \chi_{k}^{(i)}\|_{L^{1}(\tilde{\Omega})} 
     &=\|\chi_{k+1}^{(i)}- \chi_{k}^{(i)}\|_{L^{1}(\tilde{\Omega}_{g})} 
       = \|\chi_{k+1}^{(i)}- \chi_{k}^{(i)}\|_{L^{1}(\tilde{\Omega}_{g}\setminus \tilde{\Omega}_{g}^{\star})}
\\ 
& \leq  |\tilde{\Omega}_{g}\setminus \tilde{\Omega}_{g}^{\star}| \leq C_3 c_2^{\max(l_{k}(\tilde{\Omega}),m_{0})}|\tilde{\Omega}| \leq C_3 c_2^{q_{k}(\tilde{\Omega})/m}|\tilde{\Omega}|.
    \end{align*}

    It remains to prove the claim in \eqref{eq:claim_a}, which follows from a triangle inequality.
    More precisely, choose $l_{0}$ sufficiently large such that
    $\sup\{\dist(M, K): M \in U_{i}^{j}\} \leq \frac{c_{1}}{3}$ for all $i
    \geq l_{0}/m$, where $c_{1}$ denotes the
    distance of the wells, as given in \ref{item:A1}.
    Then $\nabla u_{k} \in U_{l_{k}}^{j_{k}}$ implies that $\dist(\nabla u_k, G M_{i})<
    \frac{c_{1}}{3}$ for some $i \in \{1, \dots, m\}$.
    
    After possibly further increasing $l_{0}$, condition \ref{item:A4} then yields
    \begin{align*}
      |\nabla u_{k+1}- \nabla u_{k}| < C c_{2}^{l_{k}(\tilde{\Omega})} < \frac{c_{1}}{3}.
    \end{align*}
    
    Assume then for the sake of contradiction, that $\nabla u_{k+1}$ is closest
    to a different well $GM_{j}$, then
    \begin{align*}
      c_{1} &\leq \dist (GM_{i}, GM_{j}) \leq \dist(GM_{i}, \nabla u_k) + \dist(\nabla u_k, \nabla u_{k+1}) + \dist(\nabla u_{k+1}, GM_{j}) \\
      & < \frac{c_{1}}{3}+\frac{c_{1}}{3}+\frac{c_{1}}{3} <c_{1},
    \end{align*}
    which yields a contradiction.
\end{proof}

\begin{rmk}
\label{rmk:existence_pointwise_limit} 
As a consequence of Lemma \ref{lem:L1_in} we infer that $\{\chi_k^{(i)}\}$, $i\in\{1,\dots,m\}$, is a Cauchy sequence in $L^1(\Omega)$: Indeed, for $k_0,k_1 \in \N$ with $k_0<k_1$ we have that
\begin{align*}
\|\chi_{k_1}^{(i)}-\chi_{k_0}^{(i)}\|_{L^1(\Omega)}
\leq C|\Omega| \sum\limits_{j=k_0}^{k_1} \tilde{c}^{j} \leq C |\Omega| \frac{1}{1-\tilde{c}}\tilde{c}^{k_0} \rightarrow 0 \mbox{ as } k_0 \rightarrow \infty.
\end{align*}
In particular, there exist functions $\chi^{(i)} \in \{0,1\}$, $i\in\{1,\dots,m\}$, such that $\chi_k^{(i)} \rightarrow \chi^{(i)}$ as $k \rightarrow \infty$. \\
If $u\in W^{1,\infty}(\Omega)$ denotes the associated solution to \eqref{eq:incl} we have that $\chi^{(i)}=\chi^{(i)}_u$ for $i \in \{1,\dots,m\}$.
\end{rmk}

\subsection{The BV bound}
\label{sec:BV}

The $BV$ estimate for the characteristic functions $\chi_k^{(i)}$, $i \in \{1,\dots,m\}$, is obtained analogously to the estimates from \cite{RZZ16}. Here we estimate very crudely without taking into account a lot of structure of the underlying deformation. We only rely on the observation that the gradient is piecewise constant, that all jump heights between different values of $\nabla u_k$ are uniformly bounded by assumption \ref{item:A4} and that the overall size of the jump set of $\nabla u_k$ is bounded by the perimeter of the sets in $\hat{\Omega}_k$.

\begin{lem}[$BV$ control]
\label{lem:BV_in}
Let $\Omega \subset \R^n$ with $\Omega \in \mathcal{C}$. Let $U_k^j, \tilde{U}_k^j$ be the in-approximation from \ref{item:A3tilde}.
Let $u_k:\Omega \rightarrow \R^n$ denote the map, which is obtained in the $k$-th step of the convex integration scheme from Algorithm \ref{alg:convex_int} with prescribed boundary data $M_0 \in \inte(K^{lc})$.
Then, for $i\in\{1,\dots,m\}$ and $C_{0},C_{1},C_{2}$ as in \ref{item:A3} and \ref{item:A3tilde}
\begin{align*}
\| \chi_{k+1}^{(i)}-\chi_k^{(i)}\|_{BV(\Omega)} \leq \frac{\max(C_{0},C_1,C_2)^{m+1}}{\max(C_{1},C_{2})^{m+1}} \max(C_{1},C_{2})^{k} 3^{k} |\Omega|.
\end{align*}
\end{lem}

\begin{rmk}
\label{rmk:dep}
We emphasize that -- as $C_1,C_2$ are uniform constants -- the $BV$ estimate in Lemma \ref{lem:BV_in} is such that the exponentially growing constant does \emph{not} depend on the position of the boundary data $M_0$ in matrix space. It thus provides the basis of the \emph{uniform} dependence of the regularity modulus $sp$ in Proposition \ref{prop:char_reg}.
\end{rmk}

\begin{proof}
  We bound
    \begin{align*}
      \| \chi_{k+1}^{(i)}-\chi_k^{(i)}\|_{BV(\Omega)} \leq \sum_{\tilde{\Omega} \in \hat{\Omega}_{k}} \Per(\tilde{\Omega}),
    \end{align*}
    and use the perimeter growth bounds of \ref{item:A3tilde} and \ref{item:A3}
    to control the right-hand-side sum.

  More precisely, for $\tilde{\Omega} \in \hat{\Omega}_{k}$ denote 
  \begin{align*}
    \delta_{k}(\tilde{\Omega})=
    \begin{cases}
      \frac{1}{C_{0}}, & \mbox{ if } \tilde{\Omega} \subset \Omega_{g} \cup \Omega^{[1]} \mbox{ is in one of the cases } \ref{item:A3tilde} (i), (iii), \\
      \frac{1}{C_{2}}, & \mbox{ if } \tilde{\Omega} \subset \Omega^{[2]} \mbox{ is in case } \ref{item:A3tilde} (iii) \mbox{ or } \tilde{\Omega} \in \mathcal{C}^{1} \mbox{ is in case } \ref{item:A3tilde} (ii), \\
      \frac{1}{C_{1}}, & \mbox{ if } \nabla u_{k}|_{\tilde{\Omega}} \mbox{ is in case } \ref{item:A3}.
    \end{cases}
  \end{align*}
  With this convention the estimates of \ref{item:A3tilde} and \ref{item:A3}
  then imply that 
  \begin{align*}
    \sum_{\Omega' \in \mathcal{D}(\tilde{\Omega})\cap \hat{\Omega}_{k+1}} \delta_{k}(\tilde{\Omega}) \Per(\Omega') \leq  3 \Per(\tilde{\Omega}),
  \end{align*}
  where the factor $3$ is a consequence of the fact that we consider three separate cases.

  Iterating this estimate in $k$, we thus obtain
  \begin{align*}
    \sum\limits_{\Omega^{1} \in \mathcal{D}(\Omega) \cap \hat{\Omega}_1}\sum\limits_{\Omega^{2} \in \mathcal{D}(\Omega^{1}) \cap \hat{\Omega}_2} \dots \sum\limits_{\Omega^{k+1} \in \mathcal{D}(\Omega^{k}) \cap \hat{\Omega}_{k+1}} \Per(\Omega^{k+1}) \delta_{1}(\Omega^{1}) \dots \delta_{k}(\Omega^{k}) \leq 3^{k}\Per(\Omega).
  \end{align*}

  We then claim that for each summand for at most $m+1$ indices $i \in \{1, \dots, k\}$ it holds
  that $\delta_{i}(\Omega^{i})=\frac{1}{C_{0}}$ and thus
  \begin{align*}
    \sum_{\tilde{\Omega} \in \hat{\Omega}_{k}}\Per(\tilde{\Omega}) \leq C_{0}^{m+1} \max(C_{1},C_{2})^{k-m-1} 3^{k}\Per(\Omega),
  \end{align*}
  where we for simplicity of notation assumed that $C_{0}\geq
  \max(C_{1},C_{2})$.

  In order to prove this claim, we show that if $i_{1}<i_{2}< \dots $ is a
  sequence of indices such that
  $\delta_{i_{\cdot}}(\Omega_{i_{\cdot}})=\frac{1}{C_{0}}$, 
then it follows that
  \begin{align}
    \label{eq:7}
    q_{i_{l+1}}(\Omega^{i_{l+1}})\geq q_{i_{l}}(\Omega^{i_{l}})+1.
  \end{align}
  Assuming that this is true, we obtain that $q_{i_{l}}(\Omega^{i_{l}}) \geq l-1$ for any $l\geq 2$.
Moreover, if $q_{i}(\Omega^{i})\geq m$, Lemma
  \ref{lem:well-def} entails that the algorithm is in case
  ($c_{2}$) for this set and all its descendants. Thus, in this case
  $\delta_{i}(\Omega^{i})\neq \frac{1}{C_{0}}$. Hence, it follows that if \eqref{eq:7} holds, then we have $l\leq m+1$ as claimed.

  It remains to prove \eqref{eq:7}. Thus let $\Omega_{i_l}\in \hat{\Omega}_{i_l}$. We distinguish the following cases:
  \begin{itemize}
  \item  If $\Omega^{i_{l}} \subset \tilde{\Omega}_g$ for some set $\tilde{\Omega}\in  \hat{\Omega}_{i_l-1}$ is in the case \ref{item:A3tilde} (i),
  then $q_{i_{l+1}}(\Omega^{i_{l+1}})\geq q_{i_{l}+1}(\Omega^{i_{l}+1})= q_{i_{l}}(\Omega^{i_{l}})+1$ by the update
  step of $(c_{1})$ in Algorithm \ref{alg:convex_int}.
  \item
  If $\Omega^{i_{l}}$ is in the case 
  \ref{item:A3tilde} (iii) and $\Omega^{i_{l}} \subset \tilde{\Omega}^{[1]}$ (where we used the notation introduced in \ref{item:A3tilde}), then $\Omega^{i_l +1}$ will be in the case \ref{item:A3tilde} (ii).
Hence, for
  $i_{l+1}$ to exist, we have to exit this case at some step $i \in (i_{l}, i_{l+1})$. But then necessarily $q_{i+1}(\Omega^{i+1})= q_{i}(\Omega^i)+1\geq q_{i_l}(\Omega^{i_l})+1$. Therefore, by monotonicity \eqref{eq:7} holds.
\item To conclude the argument, we note that $\Omega^{i_l}$ can never be in the case \ref{item:A3tilde} (iii) with $\Omega^{i_l} \subset \tilde{\Omega}^{[2]}$, as in this case we would have $\delta_{i_l}(\Omega^{i_l})=C_2^{-1}$, contradicting the defining property of the sequence.
  \end{itemize}
 This concludes the argument for \eqref{eq:7} and thus the proof. 
\end{proof}

Combining the results of Lemmas \ref{lem:L1_in} and \ref{lem:BV_in} we infer the proof of Proposition \ref{prop:char_reg}:

\begin{proof}[Proof of Proposition \ref{prop:char_reg}]
The proof follows by applying the interpolation result of Proposition \ref{prop:CDDD}. Indeed, choosing $\theta_0 \in (0,1)$ such that
\begin{align*}
\tilde{c}^{1-\theta_0}3^{\theta_0} \max(C_{1},C_{2})^{\theta_0}=1,
\end{align*}
and $s\in (0,1), q \in (1,\infty)$ such that $0 < s q =\theta_1 < \theta_0$, we have by Lemmas \ref{lem:L1_in} and \ref{lem:BV_in} that for $i\in\{1,\dots,m\}$
\begin{align*}
&  \|\chi_{k+1}^{(i)}-\chi_k^{(i)}\|_{W^{s,q}(\R^n)}
\leq C_{\theta_{1}}  \|\chi_{k+1}^{(i)}-\chi_k^{(i)}\|_{L^{1}}^{1-\theta_{1}} \|\chi_{k+1}^{(i)}-\chi_k^{(i)}\|_{BV}^{\theta_{1}}
  \\
&\leq C_{\theta_1} C^{1-\theta_1}|\Omega|^{1-\theta_1}\left( \frac{\max(C_{0},C_1,C_2)}{\max(C_{1},C_{2})} \right)^{\theta_{1} (m+1)} \left(\tilde{c}^{1-\theta_1}3^{\theta_1}\max(C_{1},C_{2})^{\theta_1} \right)^{k}.
\end{align*}
Since $\theta_{1}<\theta_{0}$, $\tilde{c}<1$ and $C_1, C_2\geq 1$,  this sequence is exponentially decreasing and thus
a telescopic sum ensures the convergence of $\chi_k^{(i)}$ in
$W^{s,p}(\R^n)$.
\end{proof}

\subsection{Convergence and regularity}
\label{sec:def}

As the final part of the discussion of the general set-up, we present the proof of Theorem \ref{thm:reg_gen} and explain the derivation of the $W^{s,p}$ estimates for the gradient $\nabla u$.

\begin{proof}[Proof of Theorem \ref{thm:reg_gen}]
  Since $|\nabla u_{k}|<C_{3}$ by \ref{item:A4} and since $\nabla u_k$ is piecewise constant, by the same argument as in Lemma \ref{lem:BV_in} it holds
  that
  \begin{align*}
    \|\nabla u_{k+1}-\nabla u_{k}\|_{BV(\Omega)} 
    &\leq \|\nabla u_{k+1}-\nabla u_{k}\|_{L^{\infty}(\Omega)} \sum_{\Omega^{k} \in\hat{\Omega}_{k}} \Per(\Omega^{k}) \\
    &\leq 2C_{3}C_{0}^{m+1} \max(C_{1},C_{2})^{k-m-1} 3^{k}\Per(\Omega).
  \end{align*}
  We further claim that
  \begin{align}
    \label{eq:4}
    \|\nabla u_{k+1}-\nabla u_{k}\|_{L^{1}(\Omega)} \leq 2C_{3} C \tilde{c}^{k},
  \end{align}
  with $C, \tilde{c}$ as in Lemma \ref{lem:L1_in}.
  The result then follows by interpolation (using Proposition \ref{prop:CDDD}) as in Proposition
  \ref{prop:char_reg}.

 We remark that for a general domain $\tilde{\Omega} \in \hat{\Omega}_k$ it does not hold that 
\begin{align*} 
 \|\nabla u_{k+1}-\nabla
  u_{k}\|_{L^{1}(\tilde{\Omega})} \leq
   C_{3} \sum_{i=1}^m\|\chi_{k+1}^{(i)}-\chi_{k}^{(i)}\|_{L^{1}(\tilde{\Omega})},
   \end{align*}
  and hence the $L^{1}$ convergence does not follow as a corollary of Lemma
  \ref{lem:L1_in}.
  However, in order to establish \eqref{eq:4}, we can follow the same approach as in Lemma
  \ref{lem:L1_in} and claim that
  \begin{align}
    \label{eq:5}
    \|\nabla u_{k+1}-\nabla u_{k}\|_{L^{1}(\tilde{\Omega})} \leq C_{3} C c^{q_{k}(\tilde{\Omega})/m} |\tilde{\Omega}|,
  \end{align}
  for any $\tilde{\Omega} \in \hat{\Omega}_{k}$.
 
  Summing \eqref{eq:5} over all $\tilde{\Omega} \in \hat{\Omega}_{k}$ this then implies that
  \begin{align*}
    \|\nabla u_{k+1}-\nabla u_{k}\|_{L^{1}(\Omega)} \leq C_{3}C \mathcal{E}_{k}[c^{q/m}]
  \end{align*}
with $\mathcal{E}_{k}[c^{q/m}]$ being bounded as in the proof of Lemma \ref{lem:L1_in}, which would thus conclude the argument for \eqref{eq:4}.

    It hence remains to prove \eqref{eq:5}. As before, choosing $C$ sufficiently large, this estimate is true for
  $q_{k}(\tilde{\Omega}) \leq l_{0}$ since
  \begin{align*}
    \|\nabla u_{k+1}-\nabla u_{k}\|_{L^{1}(\tilde{\Omega})} \leq 2C_{3} |\tilde{\Omega}|.
  \end{align*}
  For $q_{k}(\tilde{\Omega})\geq l_{0}$, we instead make use of \ref{item:A4} and our
  construction:
  \begin{align*}
    & \quad  \|\nabla u_{k+1}-\nabla u_{k}\|_{L^{1}(\tilde{\Omega})}
    =  \|\nabla u_{k+1}-\nabla u_{k}\|_{L^{1}(\tilde{\Omega}_{g})} \\
    &= \|\nabla u_{k+1}-\nabla u_{k}\|_{L^{1}(\tilde{\Omega}_{g}^{\star})} + \|\nabla u_{k+1}-\nabla u_{k}\|_{L^{1}(\tilde{\Omega}_{g} \setminus \tilde{\Omega}_{g}^{\star})}  \\
    &\leq \|\nabla u_{k+1}-\nabla u_{k}\|_{L^{\infty}(\tilde{\Omega}_{g}^{\star})}|\tilde{\Omega}| + 2C_{3} |\tilde{\Omega}_{g} \setminus \tilde{\Omega}_{g}^{\star}| \\
    &\leq C_{3} c_{2}^{l_{k}(\tilde{\Omega}^{\star})}|\tilde{\Omega}| + 2 C_{3}^2 c_{2}^{l_{k}(\tilde{\Omega})} |\tilde{\Omega}| \\
    &\leq (C_{3}+2C_{3}^2) c_{2}^{q_{k}(\tilde{\Omega})/m} |\tilde{\Omega}|.
  \end{align*}
  Here, we used that $\nabla u_{k}|_{\tilde{\Omega}}$ is a constant matrix $M$
  and hence the first estimate in \ref{item:A4} is applicable, while
  $|\tilde{\Omega}\setminus \tilde{\Omega}_{g}^{\star}|$ is controlled by the second estimate. 
  Choosing $C$ possibly larger such that $C \geq C_{3}+2C_{3}^2$ thus establishes
  \eqref{eq:5}, which concludes the proof.
\end{proof}

\section{Examples}
\label{sec:Ex}

\subsection{The case $K=O(n)$}

In this section (in combination with the covering results of Section \ref{sec:covering}), we verify that the set $K=O(n)$ with $n=2,3$ 
satisfies the conditions \ref{item:A1}-\ref{item:A4} of Section
\ref{sec:assumptions}, which then implies the first part of Theorem \ref{thm:reg} (c.f. Section \ref{sec:proof1}).\\

We first note that
\begin{align*}
  K=SO(n)Id \cup SO(n)
  \diag(-1,1,\dots,1),
\end{align*}
has a two-well structure with $G=SO(n)$, which acts on $K$ (and $K^{lc}$) by multiplication from the left. In order to further describe the geometry of $O(n)$ and its properties, we briefly
recall the singular
value decomposition of a matrix.

\begin{prop}[\cite{DaM12}, Theorem 7.1]
\label{prop:singular_val_decomp}
Let $A \in \R^{n \times n}$. Then there exist $V_1,V_2 \in O(n)$ and a diagonal matrix $D(A)=\diag(\sigma_1(A),\dots,\sigma_n(A))$ with $0\leq \sigma_1(A)\leq \dots \leq \sigma_n(A)$ such that $A=V_1DV_2$. In particular $O(n)=\{A\in \R^{n\times n}: \sigma_i(A)=1 \mbox{ for all } i \in \{1,\dots,n\}\}$.
Moreover, the function $\R^{n\times n}\ni A \mapsto \sigma_n(A)$ is convex.
\end{prop}

If there is no danger of confusion, we will in the sequel also simply write $\sigma_j$, $j\in\{1,\dots,n\}$, to denote the \emph{singular values} of a matrix $A$.

\begin{rmk}
We remark that the full decomposition $A=V_1 D V_2$ is not necessarily unique.
  For instance, for $Id=V_{1}^{T}Id V_{1}$ many choices of $V_{1}$ and $V_{2}$
  are possible.
  However, the diagonal matrix $D$ is unique as
  we require that
  \begin{align*}
    0\leq \sigma_1 \leq \sigma_2 \leq\dots \leq \sigma_n.
  \end{align*}
\end{rmk}

Based on the singular value decomposition we introduce the following equivalence relation on $\R^{n\times n}$.

\begin{defi}
\label{defi:equiv}
Let $M,N \in \R^{n\times n}$ we write $M \sim N$ if $D(M)=D(N)$, where $D(M), D(N)$ are the diagonal matrices containing the singular values of $M,N$ as defined in Proposition \ref{prop:singular_val_decomp}.
\end{defi}

\begin{rmk}
\label{rmk:equiv}
We observe that for any matrix $M \in \R^{n\times n}$ and any $N \in O(n)M:=\{N\in \R^{n\times n}: N=V M, \ V \in O(n)  \}$ it holds that $M \sim N$. Indeed, this follows from noting that if $M=V_1DV_2$, then $N=(V V_1)DV_2$ for some $V\in O(n)$.

By definition, we have
\begin{align*}
  O(n)=\{A \in \R^{n \times n}: A \sim Id\}.
\end{align*}
\end{rmk}

In our iterative convex integration construction of Algorithm \ref{alg:convex_int}, it is important to obtain precise control on the potential change of the corresponding singular values. In order to quantify this, we recall the Lipschitz dependence of the singular values:

\begin{lem}[\cite{GvL}, Corollary 8.6.2, p.449]
\label{lem:cont}
Let $A, E \in \R^{n\times n}$. Let $\sigma_1(A),\dots,\sigma_n(A)$ and  $\sigma_1(A+E),\dots,\sigma_n(A+E)$ denote the singular values of $A$ and $A+E$. Then we have that for $k\in\{1,\dots,n\}$
\begin{align*}
|\sigma_k(A+E)-\sigma_k(A)| \leq \sigma_n(E)=\|E\|_2.
\end{align*}
\end{lem}

Using the singular value decomposition, we infer that $O(n)^{lc}$ is very large.

\begin{lem}[\cite{DaM12}, Theorem 7.16]
\label{lem:sing_val_char}
We have that
\begin{align*}
\conv(O(n)) = O(n)^{lc} = \{A \in \R^{n\times n}: \sigma_n(A) \leq 1\}=R_n(O(n)).
\end{align*}
In particular, we have
\begin{align*}
\intconv(O(n))= \inte(O(n)^{lc})=\{A \in \R^{n\times n}: \sigma_n(A) < 1\}.
\end{align*}
\end{lem}

\begin{proof}
We first show that $\conv(O(n)) = O(n)^{lc} = \{A \in \R^{n\times n}: \sigma_n(A) \leq 1\}$.
By the singular value decomposition (Proposition \ref{prop:singular_val_decomp}), it suffices to prove the result for $O(n)\cap \diag(n,\R)=:O_d(n)$, where $\diag(n,\R)$ denotes the $n\times n$ diagonal matrices. As the mapping $A \mapsto \sigma_n(A)$ is convex, we directly infer the inclusion
\begin{align*}
O_d(n)^{lc} \subset \conv(O_d(n)) \subset \{A \in \R^{n\times n}: \sigma_n(A) \leq 1\}.
\end{align*}
It thus suffices to prove that $\{A \in \R^{n\times n}_d: \sigma_n(A) \leq 1\} \subset O_d(n)^{lc} $.
This follows directly by considering rank-one connections in $O_d(n)$. Let thus $M\in \{A \in \R^{n\times n}_d: \sigma_n(A) \leq 1\}$. By a suitable permutation of coordinates and premultiplication, we may assume that
$M=\diag(\sigma_1,\dots, \sigma_n)$ with $0\leq \sigma_1 \leq \dots \leq \sigma_n \leq 1 $. Then,
\begin{align*}
M= \frac{1+\sigma_1}{2}\diag(1,\sigma_2,\dots,\sigma_n) + \frac{1-\sigma_1}{2} \diag(-1,\sigma_2,\dots,\sigma_n).
\end{align*}
Iterating this in the remaining components yields the claim and also implies that $O(n)^{lc} \subset R_n(O(n))$.
\end{proof}

Combining Lemma \ref{lem:sing_val_char} with the fact that the two components 
$$SO(n), SO(n)\diag(-1,1,\dots,1)$$ 
are disjoint and compact, implies that property \ref{item:A1} is satisfied by $K=O(n)$.

\subsubsection{Construction of an in-approximation for $O(n)$}

Due to the symmetries of the group $O(n)$, it will in the sequel be convenient to work with an in-approximation for $O(n)$, which only depends on the singular values of a given matrix.

\begin{lem}[In-approximation for $O(n)$]
\label{lem:in_approx}
For $k \in \N$ and $\kappa \in (0,1)$ define the set
\begin{align*}
  I_{k,\kappa}= [c_{k}- \kappa d_{k}, c_{k}+\kappa d_{k}]\subset (-1,1), 
\end{align*}
where $c_{k}=1 -3 \cdot 2^{-(k+2)}$, $d_{k}=2^{-(k+2)}$.
Let now $\kappa_{0}=1/4$ and define
\begin{align*}
  U_{k}^{j}&:= \{M \in \R^{n \times n}:
             \exists D=\diag(\mu_{1}, \dots, \mu_{n}): M \sim D, \mu_{1},\dots, \mu_{j} \in I_{k, \kappa_{0}j/n},\\
           & \quad \mu_{j+1}, \dots, \mu_{n} \in I_{k-1, \kappa_{0}(1+j/n)}  \} \subset \inte(K^{lc}), \\
  \tilde{U}_{k}^{j}&:= \{M \in \R^{n \times n}: \exists D=\diag(\mu_{1}, \dots, \mu_{n}): M \sim D, \mu_{1},\dots, \mu_{j} \in I_{k, \kappa_{0}j/n}, \\
           & \quad  |\mu_{j+1}|, \dots, |\mu_{n}| \leq 1- 2 d_{k} +d_{k} \kappa_{0}j/n \}\subset \inte(K^{lc}).
\end{align*}
Then the sequence
\begin{align}
  \label{eq:1}
  \tilde{U}_{k}^{0},\tilde{U}_{k}^{1}, \dots,
  \tilde{U}_{k}^{n}=U_{k+1}^{0}, U_{k+1}^{1}, \dots,
  U_{k+1}^{n}=U_{k+2}^{0}, U_{k+2}^{1}, \dots 
\end{align}
is an in-approximation for $O(n)$.
\end{lem}

\begin{proof}
For $k\in \N$ by definition of the sets $\tilde{U}_{k}^{n}=U_{k}^{n}=U_{k+1}^{0}$. 
We show that the properties (i), (ii) from Definition \ref{defi:in_approx} are satisfied.
  Here, the property (ii) is immediately true by the definition of the sets $U_k^j$, by Remark \ref{rmk:equiv} and by Proposition \ref{prop:singular_val_decomp}. It remains to argue that (i) holds.

  Let thus $M \in \tilde{U}_{k}^{j}$ or $M \in U_{k}^{j}$ and without loss of
  generality assume that $j<m$ (otherwise replace $(k,m)\mapsto (k+1,0)$).
  Then after multiplying with elements of $G$ from the left and right we may
  assume that $M=\diag(\mu_{1}, \dots, \mu_{n})$.
  This matrix can then be expressed as a convex combination of
  $\tilde{M}=\diag(\mu_{1}, \dots, \mu_{j}, c_k, \mu_{j+2}, \dots, \mu_{n})$
  and $M'=\diag(\mu_{1}, \dots, \mu_{j}, -c_k, \mu_{j+2}, \dots, \mu_{n})$,
  with $\tilde{M},M'
  \in \tilde{U}^{j+1}_{k}$ or $U_{k}^{j+1}$, respectively, and $\tilde{M}-M'= \diag(0,\dots,0,2c_k,0,\dots,0)$ is rank-one.
\end{proof}

We further note that $\bigcup\limits_{k=1}^{\infty}\tilde{U}_{k}^{0}= \bigcup\limits_{k=1}^{\infty} \{M \in \R^{n\times n}: \sigma(M)\leq 1-2 d_{k}\}$ yields a covering of $\inte(K^{lc})$.
Moreover, since the mapping 
$$\R^{n\times n} \ni M \mapsto (\sigma_1(M),\dots,\sigma_n(M))\in \R^{n}_+$$ 
is invariant under multiplication with
  elements of $G=SO(n)$, it holds that $\tilde{U}_{k}^{j}= G \tilde{U}_{k}^{j}$ and
  $U_{k}^{j}=GU_{k}^{j}$. Similarly, we may additionally also multiply with
  $G$ from the right.
  
 Combined with these observations Lemma \ref{lem:in_approx} shows that for $m=n$ the condition \ref{item:A2} is satisfied for $K=O(n)$.

\subsubsection[The nonlinear replacement construction]{The nonlinear replacement construction without constraints}
\label{sec:nonl-repl-constr}

In the sequel, we recall replacement constructions for unconstrained, bounded
differential inclusions following \cite{C} (which has the advantage of being a
symmetric construction, c.f. also \cite{MS} and \cite{CT05}). This will imply the estimates stated in \ref{item:A3tilde}, \ref{item:A3}, \ref{item:A4} for the case $K=O(n)$ and for a special class of suitably ``adapted" domains. In Section \ref{sec:covering} we will then extend this to an arbitrary domain in the class $\mathcal{C}$.

We first present the two-dimensional construction (Lemma \ref{lem:MS}) and then extend this to three dimensions in Lemma \ref{lem:3D_O(n)}. In principle, the same argument could be used to apply this construction in arbitrary dimensions. As we are however mainly interested in $n=2,3$, we do not further pursue this here.

\begin{figure}[t]
  \includegraphics[width=0.4\linewidth, page=2]{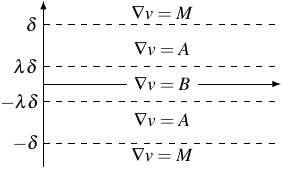}
 \includegraphics[width=0.55\linewidth, page=3]{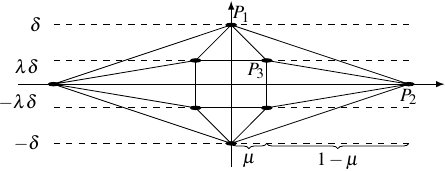}
  \caption{Starting from a double laminate $v$ we construct a domain $\Omega$ in the shape of a diamond.}
  \label{fig:kite}
\end{figure}

\begin{lem}
\label{lem:MS}
Let $A,B \in \R^{2\times 2}$ with $A-B = a \otimes n$ for $a\in \R^2 \setminus \{0\}$, $n\in S^{1}$. Let further $\lambda \in (0,1)$ and assume that 
\begin{align*}
C = (1-\lambda)A + \lambda B.
\end{align*} 
Then for each $\delta \in (0,1/2)$ there exist 
\begin{itemize}
\item[(i)] a domain $\Omega_{\delta}$ in the shape of a diamond,
\item[(ii)] a map $u:\Omega_{\delta} \rightarrow \R^2$ with the properties
\begin{align}
\label{eq:dist_grad_2D}
\dist(\nabla u, \{A,B\}) &\leq \epsilon:= 4\sqrt{2} \delta \lambda(1-\lambda)|a|,\\
\label{eq:dist_infty_2D}
|u-Cx| &\leq \epsilon:= 4\sqrt{2} \delta \lambda (1-\lambda)|a|.
\end{align}
Furthermore, there exist $\Omega'$ (which is a union of level sets of $\nabla u$) and constants $v_A, v_B\in (1/2,2)$ such that
\begin{align*}
  \dist(\nabla u, A) &\leq \epsilon \mbox{ on } \Omega' \mbox{ and } |\Omega'|=(1-\lambda)v_A |\Omega_{\delta}|, \\
  \dist(\nabla u, B) &\leq \epsilon \mbox{ on } \Omega \setminus \Omega' \mbox{ and } |\Omega\setminus \Omega'|=\lambda v_B |\Omega_{\delta}|.
\end{align*}
\end{itemize}
\end{lem}

\begin{proof}
We follow the proof of \cite{C} (without considering the additional determinant constraint). Without loss of generality we may first suppose that $0<\lambda\leq 1-\lambda$.
By carrying out a suitable translation, rescaling and rotation in matrix space, we may without loss of generality assume that
\begin{align*}
A = - \lambda a \otimes e_2, \ B= (1-\lambda) a \otimes e_2, \ C = 0,  \ |a|=1.
\end{align*}
Then on the domain $\Omega=[-1,1]\times [-\delta,\delta]$ we begin with a deformation corresponding to a simple laminate
\begin{align*}
v(x) = \left\{
\begin{array}{ll}
-\lambda (x_2 + \delta) a \mbox{ if } x_2 \in [-\delta,-\lambda \delta],\\
(1-\lambda) a x_2 \mbox{ if } x_2 \in [-\lambda \delta, \lambda \delta],\\
-\lambda (x_2 - \delta)a \mbox{ if } x_2 \in [\lambda \delta, \delta].
\end{array}
\right.
\end{align*}
We next consider the triangle spanned by $P_1=(0,\delta)$, $P_2 =(1,0)$, $P_3 = (\mu, \lambda \delta)$ for some $\mu\in(0,1/2)$ and modify the deformation there by interpolating linearly between the values of $v(P_1), v(P_2), v(P_3)$ (see Figure \ref{fig:kite}). Since $v(P_1)=v(P_2)=0$, the resulting new deformation $ \tilde{v}$ vanishes along the whole line connecting $P_1$ and $P_2$. \\
We claim that by choosing $\mu\in(0,1/2)$ suitably we can ensure that 
\begin{align}
\label{eq:close}
|\nabla \tilde{v}-A|\leq 4\sqrt{2}\lambda (1-\lambda)\delta. 
\end{align}
Carrying out similar constructions in the other quadrants and using the odd and even symmetry of $u$ with respect to the $x_2$- and $x_1$-axes, respectively, thus entails \eqref{eq:dist_grad_2D} in each of the interpolation triangles. Defining the resulting function to be the desired deformation $u$ and the resulting diamond shaped domain as $\Omega_{\delta}$, we then also obtain the validity of the boundary data and of \eqref{eq:dist_infty_2D} (the latter by an application of the fundamental theorem in combination with \eqref{eq:dist_grad_2D}).
We thus focus on proving \eqref{eq:close}. To this end, we note that since $v(P_1)=v(P_2)=0$, $v(P_3)=(1-\lambda)\lambda \delta a$, we have
\begin{align*}
\nabla \tilde{v} \begin{pmatrix} -1 \\ \delta \end{pmatrix} = \begin{pmatrix} 0 \\ 0 \end{pmatrix},\
\nabla \tilde{v} \begin{pmatrix} \mu-1 \\ \lambda \delta \end{pmatrix} = (1-\lambda) \lambda \delta a.
\end{align*}
The first condition immediately implies that $\nabla \tilde{v} = b \otimes \begin{pmatrix} \delta \\ 1 \end{pmatrix}$, while the second condition yields $b= c a$ with $c= \frac{(1-\lambda)\lambda }{\lambda -1+\mu}$. Choosing $\mu=(1-\lambda)\delta$ and recalling that $\delta \in (0,1/2)$ then leads to
\begin{align*}
|\nabla \tilde{v}-A| 
&= \lambda \left| a \otimes \left( \begin{pmatrix} 0 \\ 1 \end{pmatrix} + \frac{1-\lambda}{\lambda -1+\mu} \begin{pmatrix} \delta \\ 1\end{pmatrix} \right)  \right|
= \lambda \left|\begin{pmatrix} \frac{1-\lambda}{1-\lambda-\mu}\delta \\- \frac{\mu}{1-\lambda-\mu}\end{pmatrix}  \right|\\
&\leq \sqrt{2} \min\{\lambda, (1-\lambda)\} \frac{\delta}{1-\delta} 
\leq 4\sqrt{2} \lambda (1-\lambda) \delta.
\end{align*}
This concludes the argument for \eqref{eq:close}. 

Finally, we observe that 
\begin{align*}
|\{x: \dist(\nabla u, B)< \dist(\nabla u, A)\}| &= 2\lambda \delta(1+(1-\lambda)\delta),\\
|\{x: \dist(\nabla u, A)\leq \dist(\nabla u, B)\}| &= 2 \delta(1-\lambda)(1-\lambda \delta),
\end{align*}
which for $\delta \in (0,1/2)$ also implies the claim on the volume fractions (with $v_A = 1-\lambda \delta$, $v_B = 1+(1-\lambda)\delta$). This concludes the argument.
\end{proof}

Next we extend the previous two-dimensional construction to a three-dimensional building block:

\begin{lem}
\label{lem:3D_O(n)}
Let $A,B \in \R^{3\times 3}$ with $A-B = a \otimes n$ for $a\in \R^3 \setminus \{0\}$, $n\in S^{2}$. Let further $\lambda \in (0,1)$ and assume that 
\begin{align*}
C = (1-\lambda)A + \lambda B.
\end{align*} 
Then for each $\delta \in (0,1/2)$ there exist
\begin{itemize}
\item[(i)] a domain $\Omega_{\delta}$ in the shape of a diamond of the aspect ratio $1:\delta:1$,
\item[(ii)] a map $u:\Omega_{\delta} \rightarrow \R^3$ with the properties
\begin{align*}
\dist(\nabla u, \{A,B\}) &\leq \epsilon,\\
|u-Cx| &\leq \epsilon,
\end{align*}
with $\epsilon$ as given in Lemma \ref{lem:MS}.
Furthermore, there exist $\Omega' \subset \Omega_{\delta}$ (which consists of a union of level sets of $\nabla u$) and constants $v_A,v_B \in (1/2,2)$ (which can be chosen as in Lemma \ref{lem:MS}) such that
\begin{align*}
  \dist(\nabla u, A) &\leq \epsilon \mbox{ on } \Omega' \mbox{ and } |\Omega'|=v_A (1-\lambda) |\Omega_{\delta}|, \\
  \dist(\nabla u, B) &\leq \epsilon \mbox{ on } \Omega \setminus \Omega' \mbox{ and } |\Omega\setminus \Omega'|= v_B\lambda|\Omega_{\delta}|.
\end{align*}
\end{itemize}
\end{lem}

\begin{proof}
By normalization we can assume that $C=0$, $n=e_1$, $|a|=1$ and $a\in \text{span}\{e_1,e_2\}$. Then we consider the two-dimensional deformation $v$ from Lemma \ref{lem:lin_Conti} in the two-dimensional domain $\conv(\pm e_1, \pm \delta e_2)$ and add the two points $P_{\pm}=(0,0,\pm 1)$. We extend the two-dimensional deformation $v$ to a three-dimensional one $u$ by linear interpolation between the corners of the triangles in $\conv(\pm e_1, \pm he_2)$ and $P_{\pm}=(0,0,\pm 1)$ by defining $u(P_{\pm})=0$. By linearity, $u_3 =0$ and hence
\begin{align*}
\nabla u = \begin{pmatrix} \nabla v & d \\ 0 & 0 \end{pmatrix}, \ d \in \R^{2}.
\end{align*}
Computing the value of $d\in \R^2$ in each of the tetrahedra, which are obtained as level sets of the interpolation, we infer that $|d| \leq \max\{\lambda \delta, (\lambda\delta)/(1-\delta)\}|a|\leq 4 \lambda (1-\lambda) \delta$. Moreover, the volume estimates follow by Cavalieri's principle with the same constants as in the two-dimensional setting of Lemma \ref{lem:MS}.
\end{proof}

\begin{rmk}
  \label{rmk:singv_close}
  Using Lemma \ref{lem:cont}, the condition
  \begin{align*}
    \dist(\nabla u, \{A,B\}) &\leq \epsilon
  \end{align*}
  also implies that the singular values of $\nabla u$ are locally within an $\epsilon$ neighbourhood of $A$ or $B$.
  More precisely, let $x^{\nabla u}:=(\sigma_{1}^{\nabla u}, \dots, \sigma_{n}^{\nabla u})$, $x^{A}:=(\sigma_{1}^A,
  \dots, \sigma_{n}^A)$ and $x^{B}:=(\sigma_{1}^B, \dots, \sigma_{n}^B)$ denote the vectors of the
  singular values of $\nabla u, A$ and $B$, respectively. Then it holds that
  \begin{align}
    \label{eq:epsilon_singv}
    \min(\|x^{\nabla u}- x^A\|,\|x^{\nabla u}- x^B\|) \leq c \epsilon.
  \end{align}
\end{rmk}

In the following we apply the preceding construction of Lemma
~\ref{lem:MS} to our in-approximation and verify that a replacement construction, which satisfies the estimates from \ref{item:A3tilde}-\ref{item:A4}, can be obtained in specific diamond-shaped domains.
Based on these constructions on model domains, in Section \ref{sec:covering} we provide a covering argument
extending the result to the full class $\mathcal{C}$ of admissible domains, thus yielding the full statement of \ref{item:A3tilde}-\ref{item:A4}.

\begin{lem}
  \label{lem:On_replace}
  Let $n\in\{2,3\}$ and let $U_{k}^j$ be given as in Lemma \ref{lem:in_approx}. Assume that $M\in
\intconv(O(n))$ with $M\in U_{k}^j$ for some $k\in
\N$ and  $j\in \{0,1,\dots,n-1\}$.
Then there exist a domain $\Omega^{\Diamond}$ in the shape of a diamond of aspect ratio
$1:\delta$ if $n=2$ (or $1:\delta:1$ if $n=3$), where $\delta=\kappa_02^{-10}/n$ and a set
$(\Omega^{\Diamond})_{g}^{\star} \subset \Omega^{\Diamond}$ such that with the notation from \ref{item:A3} and \ref{item:A4}
\begin{itemize}
\item[(i)] $\nabla w(x) \in U_{k}^{j+1} \mbox{ for a.e. } x \in \Omega^{\Diamond}$.
\item[(ii)] $|\nabla w(x) -M| \leq C 2^{-k} \mbox{ for a.e. } x \in (\Omega^{\Diamond})_{g}^{\star}$.
\item[(iii)] $w(x) = Mx \mbox{ for } x\in \partial \Omega^{\Diamond}$.
\item[(iv)] $\Omega^{\Diamond}_g =\Omega^{\Diamond}$ and $|(\Omega^{\Diamond})_{g}^{\star}|\geq (1-C 2^{-k}) |\Omega^{\Diamond}_{g}|$, where $C\geq 1$ is independent of $k$.
\item[(v)] Let $\Omega_{1},\dots, \Omega_{N} \subset \Omega^{\Diamond}$ denote
  the level sets of $\nabla w$. Then it holds that
  \begin{align*}
    \sum_{i=1}^{N} \Per(\Omega_{i}) \leq 2^{n+2} \Per(\Omega^{\Diamond}).
  \end{align*}
\end{itemize}
\end{lem}

As we discuss in Section \ref{sec:2D}, the level sets $\Omega_{i}$ are
given by triangles in the two-dimensional setting and are hence contained in the
class $\mathcal{C}$ (which is defined by using triangles). In the three-dimensional setting, the definition of the
class $\mathcal{C}$ and the inclusion of the level sets is established in
Section \ref{sec:3D}. This then allows for an iterative application of the replacement construction, which is for instance used in Algorithm \ref{alg:convex_int}.

\begin{proof}
Let $M \in U_{k}^j$. By the singular value decomposition (Proposition
  \ref{prop:singular_val_decomp}), we then have that $M = V_1 D_M V_2 \sim D_M$ with $V_1, V_2 \in O(n)$. Here
$D_M = \diag(\sigma_1,\dots,\sigma_n)$ is a positive definite
diagonal matrix with $\sigma_{1}, \dots, \sigma_{j} \in
I_{k,\kappa_0 j/n}$ and $\sigma_{j+1}, \dots, \sigma_{n} \in I_{k-1,\kappa_0(1+j/n)}$.
Here, we without loss of generality assume that $\sigma_{j+1}=\min\limits_{i\in\{1,\dots,n\}} \sigma_i$.\\

For $c_k \in \R$ as in Lemma \ref{lem:in_approx}, we then define two matrices $A,B$ by
\begin{align*}
  A&= V_1\diag(\sigma_1,\dots, \sigma_{j}, c_k, \sigma_{j+2}, \dots,  \sigma_n)V_2 \in U_{k}^{j+1},\\
  B&= V_1\diag(\sigma_1,\dots, \sigma_{j}, -c_k, \sigma_{j+2}, \dots,  \sigma_n)V_2 \in U_{k}^{j+1},
\end{align*}
with $V_1,V_2$ as in the singular value decomposition of $M$.
We note that $A-B$ is rank one.
Since $\sigma_{j+1} \in I_{k-1,\kappa_0(1+j/n)} \subset
(-c_k,c_k)$, there exists $\lambda \in (0,1)$ such that
\begin{align}
\label{eq:convex}
  M=(1-\lambda) A+ \lambda B.
\end{align}
A direct computation further shows that 
\begin{align*}
  \lambda= \frac{c_k-\sigma_{j+1}}{2c_k}.
\end{align*}
The inclusion $\sigma_{j+1} \in I_{k-1,\kappa_0(1+j/n)} \subset I_{k-1,1}$
then implies that 
\begin{align*}
  d_k=2^{-(k+2)}\leq c_k-\sigma_{j+1} \leq 3\cdot 2^{-(k+1)} ,
\end{align*}
and that hence $\lambda \leq c 2^{-k}$ and $\lambda(1-\lambda) \leq c 2^{-k}$.

Applying Lemma \ref{lem:MS} we thus obtain a domain $\Omega_{\delta}$ and a function $w:\Omega_{\delta} \rightarrow \R^n$ satisfying
\begin{align*}
\dist(\nabla w, \{A,B\}) &\leq \epsilon_k :=C\delta \lambda(1-\lambda) |a| \leq c 2^{-k}\delta ,\\
  |w-Mx| &\leq \epsilon_k :=C\delta \lambda (1-\lambda)|a| \leq c 2^{-k}\delta.
\end{align*}
Denoting $\Omega'\subset \Omega_{\delta}$ as the set in which $\dist(\nabla w, A)\leq \dist(\nabla w, B)$, we in particular also infer
\begin{align*}  
  |\nabla w - A| & \leq \epsilon_k   \mbox{ on } \Omega'.
\end{align*}
Here, we note that by Lemma \ref{lem:cont}, we have that the singular values of $\nabla w$
are in a $c\epsilon_k$ neighbourhood of the singular values of $A$ and $B$.
Using that $\kappa_0 j/n +2^{k}\epsilon_k \leq \kappa_0 (j+1)/n$, we thus
conclude that $\nabla w \in U_{k}^{j+1}$ almost everywhere in $\Omega_{\delta}$, which shows (i).

Furthermore,
\begin{align}
\label{eq:A5}
  |A-M| = \lambda |A-B| \leq C \lambda  \leq C 2^{-k} \mbox { on } \Omega',
\end{align}
where $\Omega'$ satisfies
\begin{align}
\label{eq:A51}
  |\Omega'| =(1-\lambda)|\Omega_{\delta}| \geq (1-C 2^{-k})|\Omega_{\delta}|.
\end{align}
Setting $\Omega^{\Diamond}:= \Omega_{\delta}$ and combining \eqref{eq:A5}, \eqref{eq:A51} then also proves (ii). As by Lemma \ref{lem:MS} we also have that $w(x)=Mx$ on $\partial \Omega_{\delta}=\partial \Omega^{\Diamond}$, we also directly conclude the validity of (iii). Condition (iv) follows from \eqref{eq:A51} and the boundedness of $O(n)$ by defining $(\Omega^{\Diamond})_g^{\star}:=\Omega'$ (where $\Omega'$ was defined in Lemmas \ref{lem:MS}, \ref{lem:3D_O(n)}).
Finally, we note that $\Omega^{\Diamond}$ is composed of at most $2^{n}$ level sets of
$\nabla w$ (which are all in the shape of tetragons), whose perimeter is controlled
by that of $\Omega^{\Diamond}$. Summing the perimeter bounds over all these level sets, we obtain the
desired perimeter estimate stated in (v).
\end{proof}

\begin{lem}
\label{lem:On_replace_in}
Let $n\in\{2,3\}$ and let $\tilde{U}_{k}^j$ be given as in Lemma \ref{lem:in_approx}. Assume that $M\in
\intconv(O(n))$ with $M\in \tilde{U}_{k}^j$ for some $k\in
\N$ and  $j\in \{0,1,\dots,n-1\}$.
Then there exists a domain $\Omega^{\Diamond}$ in the shape of a diamond of aspect ratio
$1:\delta_{k}$ (or $1:\delta_k:1$), where $\delta_{k}=\kappa_02^{-k-10}/n$ such that 
\begin{enumerate}[label=(\roman*)]
\item $\nabla w(x) \in \tilde{U}_{k}^{j+1} \mbox{ for a.e. } x \in \Omega^{\Diamond}$.
\item $w(x) = Mx  \mbox{ for } x\in \partial \Omega^{\Diamond}$.
\item $\Omega^{\Diamond}_g = \Omega^{\Diamond}$.
\item Let $\Omega_{1},\dots, \Omega_{N} \subset \Omega^{\Diamond}$ denote
  the level sets of $\nabla w$. Then it holds that
  \begin{align*}                                       
    \sum_{i=1}^{N} \Per(\Omega_{i}) \leq 2^{n+2} \Per(\Omega^{\Diamond}).
  \end{align*} 
\end{enumerate}
\end{lem}

As remarked after Lemma \ref{lem:On_replace}, in Section \ref{sec:covering} we
verify that $\Omega_i \in \mathcal{C}$.

\begin{proof}
The proof follows as the one of Lemma \ref{lem:On_replace}, the only difference being in the proof of the inclusion $\nabla w \in \tilde{U}^{j+1}_k$, which differs from the argument in Lemma \ref{lem:On_replace}, since in the setting of Lemma \ref{lem:On_replace_in},
  we can in general only estimate
  \begin{align*}
    \lambda (1-\lambda)\leq \frac{1}{4}.
  \end{align*}
 But, due to the requirement that $\delta_k=\kappa_02^{-k-10}/n$, this then implies that
  \begin{align*}
   0< \epsilon_k \leq c  \lambda (1-\lambda) \delta_{k} \leq c\kappa_0 2^{-k-10}/n.
  \end{align*}
  As in the previous lemma, we may thus conclude that
  \begin{align*}
    \kappa_0 j/n + 2^k \epsilon_k \leq \kappa_0 (j+1)/n,
  \end{align*}
  and $\nabla w(x) \in \tilde{U}_{k}^{j+1}$ on $\Omega^{\Diamond}$.
\end{proof}

\subsubsection{Proof of Theorem \ref{thm:reg} for $K=O(2)$ and $K=O(3)$}
\label{sec:proof1}

\begin{proof}
Using the results of Lemmas \ref{lem:sing_val_char}, \ref{lem:in_approx}, \ref{lem:On_replace}, \ref{lem:On_replace_in} and the boundedness of $O(n)$ we conclude that for $K=O(n)$ the conditions \ref{item:A1}, \ref{item:A2} and the estimates in \ref{item:A3tilde}-\ref{item:A4} are satisfied for a specific class of diamond-shaped domains $\Omega^{\Diamond}$. The construction for these specific domains as such is however not sufficient, as the class of diamond-shaped domains is not ``closed" in the sense that it produces level sets for the replacement construction, which can not be covered by finitely many diamond-shaped domains. In order to iterate the construction and to apply Theorem \ref{thm:reg_gen}, we hence need to enlarge our class of admissible domains.
In Section \ref{sec:covering} we show that there are replacement constructions satisfying conditions \ref{item:A3tilde}-\ref{item:A4} for the full class of domains $\mathcal{C}$, if we restrict to $n=2,3$. In particular, combining the results of the present section with the covering arguments from Lemmas \ref{lem:covering_reduction} and \ref{lem:covering_reduc_3D} implies the proof of Theorem \ref{thm:reg}, first for domains which consist of a finite union of elements of $\mathcal{C}$ and then, by applying Lemma \ref{lem:Lip}, for general Lipschitz domains.
\end{proof}

\subsection[Geometrically linearised constructions]{The hexagonal-to-rhombic and the cubic-to-orthorhombic phase transformation}
\label{sec:lin_trans}

We consider the sets $K=\tilde{K} + \Skew(n)$, where $n=2$ and $\tilde{K}=\tilde{K}_h$ is given by the (geometrically linearised) matrices of the hexagonal-to-rhombic transformation,
\begin{equation}
\label{eq:K3}
\begin{split}
&\tilde{K}_{h}:=\{e^{(1)}, e^{(2)}, e^{(3)}\} \subset \R^{2\times 2}_{sym,0} \mbox{ with }\\
&e^{(1)}:= \begin{pmatrix}
1 & 0 \\
0& -1
\end{pmatrix},
e^{(2)}:= \frac{1}{2}\begin{pmatrix}
-1 & \sqrt{3} \\
\sqrt{3}& 1
\end{pmatrix},
e^{(3)}:= \frac{1}{2}\begin{pmatrix}
-1 & -\sqrt{3}\\
-\sqrt{3}& 1
\end{pmatrix},
\end{split}
\end{equation}
or $n=3$ and $\tilde{K}=\tilde{K}_{co}$ corresponding to the cubic-to-orthorhombic transformation,
\begin{equation}
\label{eq:K3b}
\begin{split}
&\tilde{K}_{co}:=\{e^{(1)}, \dots, e^{(6)}\} \subset \R^{3\times 3}_{sym,0} \mbox{ with }\\
&e^{(1)}:= \begin{pmatrix}
1 & \delta & 0\\
\delta& 1 & 0\\
0 & 0 & -2
\end{pmatrix},
e^{(2)}:= \begin{pmatrix}
1 & -\delta & 0\\
-\delta& 1 & 0\\
0 & 0 & -2
\end{pmatrix},
e^{(3)}:= \begin{pmatrix}
1 & 0 & \delta\\
0& -2 & 0\\
\delta & 0 & 1
\end{pmatrix},\\
&e^{(4)}:= \begin{pmatrix}
1 & 0 & -\delta\\
0& -2 & 0\\
-\delta & 0 & 1
\end{pmatrix},
e^{(5)}:= \begin{pmatrix}
-2 & 0 & 0\\
0& 1 & \delta\\
0 & \delta & 1
\end{pmatrix},
e^{(6)}:= \begin{pmatrix}
-2 & 0 & 0\\
0& 1 & -\delta\\
0 & -\delta & 1
\end{pmatrix}.
\end{split}
\end{equation} 
Here $\R^{2\times 2}_{sym,0}$ and $\R^{3\times 3}_{sym,0}$ denote the vector space of symmetric, trace free matrices in two and three dimensions.
We seek to verify the assumptions \ref{item:A1}-\ref{item:A4} for these inclusion problems. 
More generally, the described properties remain valid, if 
\begin{align*}
\tilde{K}_{l} := \{e^{(1)},\dots,e^{(m)}\},
\end{align*}
with $e^{(i)}, e^{(j)}$ being pairwise rank-one connected for $i\neq j$, and if $\conv(\{e^{(1)},\dots,e^{(m)}\})$ is sufficiently large (e.g. $\dim\conv(\{e^{(1)},\dots,e^{(m)}\})=\frac{n(n+1)}{2}$ or $\dim\conv(\{e^{(1)},\dots,e^{(m)}\})=\frac{n(n+1)}{2}-1$ if there is an additional trace constraint as in our examples). As this only requires minor modifications, we focus on the examples $\tilde{K}_{h}$ and $\tilde{K}_{co}$ in the sequel.

We note that, by definition, as required in the condition \ref{item:A1} the sets $K$ are of the form 
\begin{align*}
K= \bigcup\limits_{j=1}^m \{e^{(j)}+\Skew(n)\},
\end{align*}
with $G=\Skew(n)$. This group acts on $K$ (and $K^{lc}$) by addition.

In order to simplify notation, we use the following convention: For a matrix $M\in \R^{n\times n}$ we set
\begin{align*}
e(M):=\frac{1}{2}(M+M^t), \ \omega(M):=\frac{1}{2}(M-M^t).
\end{align*}
Thus $M=e(M)+\omega(M)$ is the unique decomposition of $M$ into its symmetric and antisymmetric parts. Based on this, we also adopt the point of view that $\R^{n\times n}:=\R^{n\times n}_{sym}\times\Skew(n)$.

\subsubsection{Convex hull and in-approximation}

Since we are working in the context of symmetrized matrices, we replace the notation of the lamination convex hull by the notion of symmetrized lamination convex hull. It is defined as follows:

\begin{defi}
\label{defi:symm_lam_hull}
Let $\tilde{K} \subset \R^{n\times n}_{sym}$. Then $\tilde{K}^{lc,sym}:= \bigcup\limits_{l=0}^{\infty} R_{l,sym}(\tilde{K})$, where $R_{0,sym}(\tilde{K})=\tilde{K}$ and for $l\geq 1$
\begin{align*}
R_{l,sym}(\tilde{K})&:= \{M \in \R^{n\times n}_{sym}: M = \lambda A + (1-\lambda) B \mbox{ for some } \lambda \in (0,1), \ A,B \in R_{l-1}(\tilde{K}) \\
& \quad \quad \mbox{and } A-B = a \odot b \mbox{ for some } a\in \R^{n}\setminus \{0\}, b\in S^{n-1} \},
\end{align*}
where $a \odot b = \frac{1}{2}(a\otimes b + b \otimes a)$.
\end{defi}

Next we recall that it is always possible to pull-up a symmetrized rank-one connection for symmetric matrices with the same trace.

\begin{lem}[Rank-one vs symmetrized rank-one connectedness, \cite{R16}, Lemma 9]
\label{lem:sym_rk1}
Let $e_1,e_2 \in \R^{n\times n}_{sym}$ with $\tr(e_1)=0=\tr(e_2)$. Then the following statements are equivalent:
\begin{itemize}
\item[(i)] There exist vectors $a\in \R^{n}\setminus\{0\}, b \in S^{n-1}$ such that 
\begin{align*}
e_1-e_2 = a\odot b.
\end{align*}
\item[(ii)] There exist matrices $M_1,M_2 \in \R^{n\times n}$ and vectors $a\in \R^{n}\setminus\{0\}, b \in S^{n-1}$ such that
\begin{align*}
M_1 - M_2 &= a\otimes b,\\
e(M_1) &= e_1, e(M_2) = e_2.
\end{align*}
\item[(iii)] $\rank(e_{1}-e_2)\leq 2$.
\end{itemize}
\end{lem}

\begin{rmk}
\label{rmk:skew}
For later reference, we note that if for some matrices $e,e_1,e_2 \in \R^{n\times n}_{sym}$ with $e_1-e_2= a\odot b$ for $a\in \R^{n}\setminus \{0\}$, $b\in S^{n-1}$ we have 
\begin{align*}
e= \lambda e_1 + (1-\lambda)e_2, \ \lambda \in (0,1),
\end{align*}
and $e=e(M)$ for some $M \in \R^{n\times n}$, then there exist matrices $M_1, M_2$ such that
\begin{align*}
M = \lambda M_1 + (1-\lambda) M_2, \ e(M_1)=e_1, \ e(M_2)=e_2, \ \rank(M_1- M_2)=1.
\end{align*}
Indeed, setting
\begin{align*}
M_1 &= e_1 + \omega(M) + (1-\lambda) S,\\
M_2 & = e_2 + \omega(M) - \lambda S,
\end{align*}
with $S=\pm \omega(a\otimes b)$ yields
\begin{align*}
M_1 - M_2 = a\odot b + S = \pm a\otimes b.
\end{align*}
Thus, given a symmetrized rank-one connection for symmetric matrices, it is always possible to ``pull this up" to obtain a rank-one connection for the associated non-symmetric matrices. Here we have the choice between $\pm \omega(a\otimes b)$ for the matrix $S$.
\end{rmk}

We note that by a well-known result (c.f. for instance \cite{R16}, Lemma 4, where this is detailed) we obtain that the symmetrized lamination convex hulls of $\tilde{K}_{h}$ and $\tilde{K}_{co}$ coincide with the corresponding convex hulls. Combined with Lemma \ref{lem:sym_rk1} this also allows us to characterise the full convex hulls of the sets $K$. 

\begin{lem}
\label{lem:hulls}
Let $\tilde{K}=\tilde{K}_{h}$ or $\tilde{K}= \tilde{K}_{co}$ be as above. Then, $\tilde{K}^{lc,sym}= \conv(\tilde{K})$. More specifically,
\begin{align*}
R_{2,sym}(\tilde{K}_h)=\tilde{K}_h^{lc,sym}, \ R_{5,sym}(\tilde{K}_{co})=\tilde{K}^{lc,sym}_{co}.
\end{align*}
In particular, $\dim(K_{h}^{lc,sym})=2$ and $\dim(K_{h}^{lc,sym})=5$.

Furthermore, for $K=\tilde{K}_h \times \Skew(2)$ or $K=\tilde{K}_{co}\times \Skew(3)$ we have 
\begin{align*}
\conv(K) = K^{lc} = \tilde{K}^{lc,sym}\times \Skew(n),
\end{align*}
with $n=2$ (if $K=\tilde{K}_h \times \Skew(2)$) or $n=3$ (if $K=\tilde{K}_{co}\times \Skew(3)$).

\end{lem}

In particular, we remark that in the both examples, i.e. for $K=K_{h}$ and $K=K_{co}$ we are working with an additional trace constraint. Hence, in the sequel, the conditions stated in \ref{item:A1}-\ref{item:A4} are understood as conditions on the relative interior of $\R^{n\times n}$, i.e. on the interior of $\R^{n\times n}$ with an additional trace constraint.
Based on the observation of Lemma \ref{lem:hulls} and in order to simplify notation, we introduce barycentric coordinates: As any element $e\in \conv(\tilde{K})$ with $\tilde{K}$ as above can be written as 
\begin{align*}
e= \sum\limits_{j=1}^{m} \mu_j e^{(j)}, \quad \sum\limits_{j=1}^{m}\mu_j = 1, \quad \mu_j \in [0,1], 
\end{align*}
where the matrices $e^{(1)},\dots,e^{(m)}$ are the ones from $\tilde{K}_{h}$ or from $\tilde{K}_{co}$ (in which case $m=3$ or $m=6$).
We identify $e$ with the coordinates $\mu=(\mu_1,\dots,\mu_m)$.\\

With these barycentric coordinates at hand, we can describe a possible and convenient in-approximation. This makes use of the following sets:

\begin{defi}
\label{defi:in_approx_sym}
Let $m\in \N$ be as above.
For $k\in \N$ and $\kappa \in (0,1)$ define the sets $J_{k,\kappa}=[ 2^{-(k+2m+2)}(3-\kappa), 2^{-(k+2m+2)}(3+\kappa)]$. 
Let $\kappa_0=1/4$ and $j\in\{0,\dots,m-1\}$. Then define
\begin{align*}
U_{k,l}^{j}&:= \{\mu=(\mu_1,\dots,\mu_m): \mbox{ There exist } i_1,\dots,i_j \in \{1,\dots,m\}\setminus \{l\} \mbox{ s.t. }\\
&\quad  \mu_{i_1},\dots,\mu_{i_j} \in J_{k,\kappa_0(j/(m-1))}
\mbox{ and for } i_{j+2},\dots,i_{m}\in \{1,\dots,m\}\setminus \{l,i_1,\dots,i_j\} \\
&\quad \mu_{i_{j+2}},\dots,\mu_{i_m} \in J_{k-1,\kappa_0(1+j/(m-1))}, \ \sum\limits_{j=1}^{m}\mu_j =1 \},\\
\hat{U}_{k}^{j}&:= \{\mu=(\mu_1,\dots,\mu_m): \mbox{ There exist } i_1,\dots,i_j \in \{1,\dots,m\} \mbox{ s.t. }\\
&\quad  \mu_{i_1},\dots,\mu_{i_j} \in J_{k,\kappa_0(j/(m-1))}
\mbox{ and for } i_{j+1},\dots,i_{m}\in \{1,\dots,m\}\setminus \{i_1,\dots,i_j\} \\
&  \quad  \mu_{i_{j+1}},\dots, \mu_{i_m} \geq 2^{-(k+2m)}(1+\kappa_0 j/(m-1)), \ \sum\limits_{j=1}^{m}\mu_j =1 \}.
\end{align*}
Based on this, we define for $k\in \N$, $j\in\{0,\dots,m-1\}$
\begin{align*}
U_{k}^j = \bigcup\limits_{l=1}^{m} U_{k,l}^j \times  \Skew(n), \quad
\tilde{U}_{k}^j = \hat{U}_{k}^j \times  \Skew(n).
\end{align*}
\end{defi}

\begin{rmk}
\label{rmk:inapprox}
We note that for $m\geq 2$ in the definition of $U^{j}_{k,l}$ we have that $\mu_l \geq 1- \sum\limits_{j\neq l}\mu_j \geq 1-m2^{-k-2m+2}\geq \frac{3}{4}$.
\end{rmk}

We claim that these sets form a convenient in-approximation.

\begin{lem}
\label{in_approx_sym}
Let $n=2$ and $m=3$ or $n=3$ and $m=6$.
Let $k\in\N$, $j\in\{0,\dots,m-1\}$ and let $U^j_k, \tilde{U}^j_k$ be the sets from Definition \ref{defi:in_approx_sym}. Then the sequence of sets
\begin{align*}
\tilde{U}_{k}^j, \tilde{U}_k^{j+1},\dots, \tilde{U}_k^{m-1}=U_{k+1}^0, \dots, U_{k+1}^{m-1}=U_{k+2}^0, U_{k+2}^1,\dots
\end{align*}
forms an in-approximation for $K$, where $K=K_{h}=\tilde{K}_{h} \times \Skew(n)$ or $K=K_{co}=\tilde{K}_{h} \times \Skew(n)$.
Moreover,
\begin{align*}
\tilde{U}^{j}_k, U_k^j \subset \inte(K^{lc}), \ \inte(K^{lc})= \bigcup\limits_{k,j} \tilde{U}^{j}_k.
\end{align*}
\end{lem}

\begin{proof}
As the last statement is a direct consequence of the definition of the sets $U_{k}^j, \tilde{U}^k_j$, it suffices to prove the statement on the in-approximation.
To this end, we first note that the symmetrized rank-one connectedness of the matrices $e^{(1)},\dots,e^{(m)}$ implies
$(U_k^{j})^{lc,sym}=\conv(U_k^{j})$ and $(\tilde{U}_k^{j})^{lc,sym}=\conv(\tilde{U}_k^{j})$
by an argument similar to the one from Lemma \ref{lem:hulls}. By translating in skew space and recalling Lemma \ref{lem:sym_rk1}, this also yields $(U_k^{j})^{lc}=\conv(U_k^{j})\times \Skew(n)$ and $(\tilde{U}_k^{j})^{lc}=\conv(\tilde{U}_k^{j})\times \Skew(n)$. This shows condition (i) in the definition of the in-approximation. Property (ii) is a direct consequence of the definition of the sets $\tilde{U}_k^j$ and $U_k^j$.
\end{proof}

Noting that $e^{(j)}+\Skew(n)$ and $e^{(i)}+\Skew(n)$ with $i\neq j$ are parallel, disjoint surfaces in $\R^{n\times n}$ and invoking Lemma \ref{lem:hulls} implies the condition \ref{item:A1}. Combining Lemma \ref{lem:in_approx} with the invariance of $\tilde{U}^j_k$ and $U^j_k$ with respect to actions of $\Skew(2)$ (if $n=2$) and $\Skew(3)$ (if $n=3$) also yields the property \ref{item:A2}.

\subsubsection{Replacement constructions}

In the sequel, we seek to verify the conditions \ref{item:A3tilde}, \ref{item:A3} in the geometrically linearised setting. Here we rely on an analogue of Lemma \ref{lem:MS}, which however takes the additional trace constraint into account. This is achieved by linearising a construction due to Conti \cite{C} (see also \cite{R16}, Lemmas 5 and 6).\\

We begin by discussing the two-dimensional construction.

\begin{lem}[\cite{C}, Lemma 2.3, and \cite{R16}, Lemma 5]
\label{lem:lin_Conti}
Let $A,B\in \R^{2\times 2}$ with $\tr(A)=\tr(B)=0$ be such that
\begin{align*}
A-B = a \otimes n \mbox{ for } a\in \R^2\setminus\{0\},\ n\in S^1.
\end{align*}
Assume that $M= \lambda A + (1-\lambda)B$ for some $\lambda \in(0,1)$. There exist a diamond shaped domain $\Omega_{\delta} \subset \R^2$ with side ratio $1:\delta$ and a piecewise affine function $u:\Omega_{\delta} \rightarrow \R^2$ such that
\begin{itemize}
\item[(i)] $u(x)=M x$ on $\partial \Omega_{\delta}$.
\item[(ii)] $\nabla \cdot u = 0 \mbox{ in } \Omega_{\delta}$.
\item[(iii)] $\dist(\nabla u, A\cup B)  \leq C \lambda (1-\lambda)h |a|$ and 
\begin{align*}
|\{x\in \Omega_{\delta}:\dist(\nabla u, A)<\dist(\nabla u, B)\}|&=v_A \lambda |\Omega_{\delta}|, \\
|\{x\in \Omega_{\delta}:\dist(\nabla u, B)\geq \dist(\nabla u, A)\}|&= v_B (1-\lambda) |\Omega_{\delta}|,
\end{align*}
for some constants $v_A, v_B \in (1/2,2)$.
\item[(iv)] $\nabla u$ attains at most five different values.
\item[(v)] The domain is given as $\Omega_{\delta}=Q(\conv(\{\pm e_1, \pm \delta e_2\}))$, $\delta= \delta(\lambda, |a|) \in (0,1/2)$, $Q \in SO(2)$. It can be divided into 10 triangles on which $u$ is affine.
\item[(vi)] $\|u-Mx\|_{L^{\infty}(\Omega)} \leq C \delta \lambda (1-\lambda)|a|$.
\end{itemize}
\end{lem}

\begin{figure}[h]
  \centering
  \includegraphics[width=0.55\linewidth, page=4]{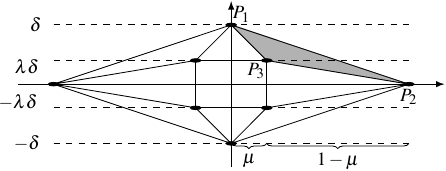}
  \caption{The diamond shaped domain of the construction of Lemma
    \ref{lem:lin_Conti}. Here, the triangle $P_1P_2P_3$ is highlighted.}
  \label{fig:conti}
\end{figure}

\begin{proof}
By a translation in matrix space, without loss of generality, we may assume that
$M=0$. By a further rotation and scaling (depending on $A-M$, $B-M$) in real
space we may assume that $n=e_2$ and that $a=e_1$ (this in particular involves the passage $u \rightarrow \frac{u}{|a|}$), where we used that $a\perp n$ (by the imposed trace constraint). Hence, we have that
\begin{align*}
A= \begin{pmatrix}
0 & 1-\lambda \\
0 & 0 
\end{pmatrix}, \ B= \begin{pmatrix}
0 & -\lambda \\
0 & 0 
\end{pmatrix}, 
\end{align*}
Without loss of generality we may further assume that $0<\lambda \leq 1-\lambda$ (else replace $A$ by $-A$ and $B$ by $-B$ and rename $\tilde{\lambda}=1-\lambda$).
As in Conti's original construction \cite{C} we construct a solution in the diamond with length scales $\delta,\lambda,\mu,1$ as depicted in Figure \ref{fig:conti}. As in the original construction, we focus on the deformation in the first quadrant and then extend it to the full diamond by symmetry afterwards. We define the affine functions
\begin{align*}
&v^{M_0}(x):= \begin{pmatrix} (1-\lambda)x_2\\0\end{pmatrix}, \
v^{M_1}(x):= \begin{pmatrix} - \lambda x_2 + \lambda \delta \\ 0  \end{pmatrix},\\
&v^{M_2}(x):= \begin{pmatrix} 0 \\ - q(1-\mu) x_1 \end{pmatrix}, \
v^{M_3}(x):= \begin{pmatrix} 0\\ q \mu x_1 - q \mu  \end{pmatrix},
\end{align*}
which have the gradients
\begin{align*}
&M_0 = \begin{pmatrix} 0 & 1-\lambda \\ 0 & 0 \end{pmatrix}, \ 
M_1 = \begin{pmatrix} 0 & -\lambda \\ 0 & 0 \end{pmatrix},\\
&M_2 = \begin{pmatrix} 0 & 0 \\ -q(1-\mu) & 0 \end{pmatrix}, \ M_3 = \begin{pmatrix} 0 & 0 \\ q \mu & 0 \end{pmatrix},
\end{align*}
where $q\in \R$ is to be specified in the sequel.
As $M_0, M_1$ and $M_2, M_3$ are each respectively rank-one connected, we can define the following 
\begin{align*}
 \tilde{v}(x):= \left\{ \begin{array}{ll} 
 &v^{M_0}(x) + v^{M_2}(x) \mbox{ in }   [0,\mu] \times [0, \delta \lambda],\\
 & v^{M_0}(x) + v^{M_3}(x) \mbox{ in }  [\mu,1] \times [0, \delta \lambda] ,\\
 &v^{M_1}(x) + v^{M_2}(x) \mbox{ in }  [0,\mu] \times [\delta \lambda, \delta],\\
 &v^{M_1}(x) + v^{M_3}(x) \mbox{ in }  [\mu,1] \times [\delta \lambda, \delta],
                        \end{array} \right.   
\end{align*}
as a piecewise affine (in particular continuous) function.
Setting $P_1 :=(0,\delta)$ and $P_2:=(1, 0)$, we infer that
\begin{align*}
\tilde{v}(P_1)= 0 = \tilde{v}(P_2).
\end{align*}
Furthermore, setting $P_3:=(\mu,\delta \lambda)$ and interpolating linearly in the shaded triangle in Figure \ref{fig:conti}, which is defined as the convex hull of $P_1P_3P_2$, we deduce that on the line segment $P_1,P_2$ the zero boundary conditions are satisfied. We define a new function $u$ on $\conv(P_1,P_2,(0,0))$ by setting it equal to $\tilde{v}$ outside the triangle $P_1 P_2 P_3$ and defining it as the interpolated function in this triangle. In order to ensure that $\nabla \cdot u = 0$, we choose $q= \frac{\lambda(1-\lambda)}{\mu(1-\mu)}\delta^2$ (indeed, this can for instance be seen by computing the gradient in the interpolated region. It is given by
\begin{align*}
D = \frac{q(1-\mu)\mu}{\delta(1-\lambda-\mu)} \begin{pmatrix}
- \frac{\lambda (1-\lambda)\delta^2}{q(1-q)\mu} &
- \frac{\lambda (1-\lambda)\delta}{q(1-q)\mu} \\ \delta  &  1 
 \end{pmatrix}.
\end{align*}
Alternatively, it is possible to argue by Gauß's theorem). Inserting the value of $q$ into the expression for the gradient hence yields
\begin{align*}
D = \frac{\lambda(1-\lambda)}{1-\lambda-\mu} \begin{pmatrix} -\delta & -1\\ \delta^2 & \delta \end{pmatrix}.
\end{align*}
Recalling that $\lambda \in(0,1/2]$ and that $\delta \in (0,1/2)$, we infer that for $\mu=(1-\lambda)\delta$
\begin{align*}
|D-M_1|
&\leq 3\frac{\delta \lambda (1-\lambda)}{1-\lambda -\mu} + \frac{\lambda\mu(1-\lambda)}{1-\lambda-\mu}\\
&\leq 3\frac{\delta \lambda (1-\lambda)}{(1-\lambda)(1 -\delta)} + \frac{\lambda (1-\lambda)^2 \delta}{(1-\lambda)(1 -\delta)}
 \leq 5\frac{\delta \lambda }{(1 -\delta)} .
\end{align*}
Similarly, we infer the closeness condition
\begin{align*}
\dist(\nabla u, A\cup B) \leq 20 \delta \lambda (1-\lambda).
\end{align*}
Using these observations and the distribution of the gradients then also entails the result on the volume fractions which is stated in (iii). More precisely, as in Lemma \ref{lem:MS} we have that

\begin{align*}
|\{x: \dist(\nabla u, B)< \dist(\nabla u, A)\}| &= 2 \delta(1-\lambda)(1-\lambda \delta),\\
|\{x: \dist(\nabla u, A)\leq \dist(\nabla u, B)\}|  &= 2\lambda \delta(1+(1-\lambda)\delta),
\end{align*}
Finally, we estimate the $L^{\infty}$ error. We have
\begin{align*}
|M x - u(x)| &\leq |(M-B)|(1-\lambda)\delta + \dist(\nabla u, A\cup B) ((1-\lambda)\delta +1)\\
& \quad + |M-A|\lambda \delta\\
& \leq C\lambda (1-\lambda) \delta.
\end{align*}
Undoing the rescaling with $|a|$ then implies the claims.
This concludes the proof.
\end{proof}

We next show that as in Lemma \ref{lem:3D_O(n)} it is possible to pass from the
two-dimensional to the  three-dimensional variant of Lemma \ref{lem:lin_Conti}.

\begin{lem}[\cite{C}, Lemma 2.4, and \cite{R16}, Lemma 6]
\label{lem:lin_Conti_3D}
Let $A,B\in \R^{3\times 3}$ with $\tr(A)=\tr(B)=0$ be such that
\begin{align*}
A-B = a \otimes n \mbox{ for } a\in \R^3\setminus\{0\},\ n\in S^2.
\end{align*}
Assume that $M= \lambda A + (1-\lambda)B$ for some $\lambda \in(0,1)$. There exist a diamond shaped domain $\Omega_{\delta} \subset \R^3$ with side length ratio $1:\delta:1$ (with $\delta \in (0,1/2)$) and a piecewise affine function $u:\Omega_{\delta} \rightarrow \R^3$ such that
\begin{itemize}
\item[(i)] $u(x)=M x$ on $\partial \Omega_{\delta}$.
\item[(ii)] $\nabla \cdot u = 0 \mbox{ in } \Omega_{\delta}$.
\item[(iii)] $\dist(\nabla u, A\cup B)  \leq C \lambda (1-\lambda)\delta |a|$ and 
\begin{align*}
|\{x\in \Omega_{\delta}:\dist(\nabla u, A)<\dist(\nabla u, B)\}|&=v_A\lambda |\Omega_{\delta}|, \\
|\{x\in \Omega:\dist(\nabla u, B)\geq \dist(\nabla u, A)\}|&=v_B(1-\lambda) |\Omega_{\delta}|,
\end{align*}
for some constants $v_A,v_B\in(1/2,2)$ which can be chosen as in Lemma \ref{lem:lin_Conti}.
\item[(iv)] $\nabla u$ attains at most ten different values.
\item[(v)] The domain is given as $\Omega_{\delta}= Q \conv(\pm e_1, \pm \delta e_2, \pm e_3)$, $Q \in SO(3)$, can be divided into 20 triangles on which $u$ is affine.
\item[(vi)] $\|u-Mx\|_{L^{\infty}(\Omega_{\delta})} \leq C \delta \lambda (1-\lambda)|a|$.
\end{itemize}
\end{lem}

\begin{proof}
The proof and error bounds follow as in Lemma \ref{lem:3D_O(n)} and by noting that as 
\begin{align*}
\nabla \tilde{u}= \begin{pmatrix} \nabla' v & d \\ 0& 0  \end{pmatrix}, \ d \in \R^2,
\end{align*}
the trace constraint is still satisfied.
\end{proof}

The previous replacement construction in combination with the covering results from Section \ref{sec:covering} allow us to deduce the properties (A3) and (A4), which were formulated in Section \ref{sec:assumptions}:

\begin{lem}
  \label{lem:K_lin_replace}
 Let $n\in\{2,3\}$ and let $K=K_{h}$ (with $m=3$) or $K=K_{co}$ (with $m=5$).
  Let $U_{k}^j$ be given by Definition \ref{defi:in_approx_sym}.
Assume that $M\in \inte(K^{lc})$ with $M\in U_{k}^j$ for some $k\in
\N$ and  $j\in \{0,1,\dots,m-1\}$.
Let $\delta=\kappa_02^{-(10+2m)}/m$.
Then there exist a domain $\Omega^{\Diamond}$ of aspect ratio $1:\delta$ (or $1:\delta:1$), a piecewise affine map $w:\Omega^{\Diamond} \rightarrow \R^n$ and a domain $(\Omega^{\Diamond})_g^{\star} \subset \Omega^{\Diamond}$ such that with the notation from \ref{item:A3}, \ref{item:A3}
\begin{enumerate}[label=(\roman*)]
\item $\nabla w(x) \in U_{k}^{j+1} \mbox{ for a.e. } x \in \Omega^{\Diamond}_g$.
\item $|\nabla w(x) - M|\leq C 2^{-k}$ for a.e. $x\in (\Omega^{\Diamond})_g^{\ast}$.
\item $w(x) = Mx \mbox{ for } x\in (\Omega^{\Diamond} \setminus \Omega_g^{\Diamond}) \cup \partial \Omega^{\Diamond}$.
\item $\Omega^{\Diamond}_g = \Omega^{\Diamond}$ and $|(\Omega_{g}^{\Diamond})^{\star}|\geq (1-C 2^{-k}) |\Omega_{g}^{\Diamond}|$.
\item Let $\Omega_{1},\dots, \Omega_{N} \subset \Omega^{\Diamond}$ denote
  the level sets of $\nabla w$. Then it holds that
  \begin{align*}
    \sum_{i=1}^{N} \Per(\Omega_{i}) \leq 2^{n+2} \Per(\Omega^{\Diamond}).
  \end{align*}
\end{enumerate}
\end{lem}

As remarked after Lemma \ref{lem:On_replace}, in Section \ref{sec:covering} we further
verify that $\Omega_i \in \mathcal{C}$.

\begin{proof}
The proof follows similarly as in the analogous case of $O(n)$ and is based on a suitable application of the replacement construction of Lemma \ref{lem:lin_Conti}.
By assumption we have that $e(M) =(\mu_1,\dots,\mu_m) \in U_{k,l}^j$ for some $l\in\{1,\dots,m\}$. Without loss of generality assume that $l=1$ and that $\mu_2 \in J_{k-1,\kappa(1+j/(m-1))}$. Let
\begin{align*}
\tilde{e}&= (3 \cdot 2^{-k-2-2m}, \mu_1 + \mu_2 - 3\cdot 2^{-k-2-2m},\mu_3 ,\dots,\mu_{m}),\\
\hat{e}&= (\mu_1+ \mu_2 -3\cdot 2^{-k-2-2m}, 3 \cdot 2^{-k-2-2m}, \mu_3 ,\dots,\mu_{m}).
\end{align*}
We note that 
\begin{align}
\label{eq:rank_one}
\begin{split}
\tilde{e}-\hat{e}
&= (3 \cdot 2^{-k-1-2m}-\mu_1-\mu_2)(e^{(1)}-e^{(2)})\\
&= (3 \cdot 2^{-k-1-2m}-\mu_1-\mu_2) a_{12} \odot n_{12}.
\end{split}
\end{align}
Moreover the assumption that $e(M)\in U^{j}_{k,1}$ combined with the construction of $\tilde{e},\hat{e}$ implies that $\tilde{e}\in U_{k,2}^{j+1}$ and $\hat{e}\in U_{k,1}^{j+1}$.
Also,
\begin{align*}
e(M)= \lambda \tilde{e} + (1-\lambda) \hat{e},
\end{align*}
where $\lambda = \frac{\mu_2 - 3 \cdot 2^{-k-2-2m}}{\mu_1+ \mu_2 - 3 \cdot
  2^{-k-1-2m}}$. 
Recalling Remark \ref{rmk:inapprox}, we infer that $\lambda \in (0,1)$ and $\lambda \leq C  2^{-k}$.

Lemma \ref{lem:sym_rk1} and Remark \ref{rmk:skew} ensure that there exist matrices $M_1,M_2 \in \R^{n\times n}$ such that
\begin{align*}
M= \lambda M_1 + (1-\lambda) M_2, \ e(M_1)= \tilde{e}, \ e(M_2)=\hat{e}, \ \rank(M_1-M_2)=1.
\end{align*}
Since $\tilde{e}\in U_{k,2}^{j+1}$ and $\hat{e}\in U_{k,1}^{j+1}$, we have that $M_1, M_2 \in U_{k}^{j+1}$. 
We apply the construction from Lemmas \ref{lem:lin_Conti} or \ref{lem:lin_Conti_3D} with $\delta = \kappa_0 2^{-10-2m}/m$.
This yields a domain $\Omega_{\delta}$ and deformation $w:\Omega_{\delta}\rightarrow \R^n$ (with $n=2,3$). We remark that in the application of Lemma \ref{lem:sym_rk1} we have the choice between two possible skew directions, the matrices $\pm S$ in the notation of Remark \ref{rmk:skew}. In order to ensure that the gradient remains bounded (i.e. to satisfy our assumption (A5)), we have to prescribe the skew part carefully. The discussion of this is however postponed to Lemma \ref{lem:skew} in Section \ref{sec:skew}.

Ignoring for the moment the issue of choosing the skew part in such a way that the gradient remains bounded, thus leaves us to verify that $\nabla w \in U^{j+1}_{k}$. 
By construction of $\tilde{e}, \hat{e}$ we have that
\begin{align*}
\dist(e(\nabla w), e(M_1)\cup e(M_2)) \leq \frac{\epsilon_k}{2}:=c \lambda (1-\lambda) \delta.
\end{align*}
Recalling the construction of the skew part from Remark \ref{rmk:skew} (independently of which sign is chosen) then also yields that 
\begin{align*}
\dist(\nabla w, M_1 \cup M_2)  \leq \epsilon_k.
\end{align*}
Thus, it suffices to show that the error $\epsilon_k$ is sufficiently small. This however follows from the fact that
\begin{align*}
0<\epsilon_k \leq 2c \lambda (1-\lambda) \delta \leq c 2^{-k+2m} \kappa_0 2^{-9}/m .
\end{align*}
Defining 
\begin{align*}
\Omega^{\Diamond}:=\Omega_{\delta}, \
(\Omega^{\Diamond})_g^{\star}:=\{x\in \Omega^{\Diamond}: \dist(e(\nabla w), \hat{e})\leq \dist(e(\nabla w),\tilde{e})\},
\end{align*}
 and using the notation introduced in \ref{item:A3tilde}, \ref{item:A3} then gives properties (i), (ii), (iii), (v) and the first property in (iv).
The result on the volume fractions of $\Omega^{\star}_g$ which is stated in (iv) follows from the volume fraction estimates in Lemmas \ref{lem:lin_Conti_3D} and \ref{lem:lin_Conti} and the explicit expression for $\lambda\in(0,1)$ from above.
\end{proof}

Analogously, we infer the replacement construction for the sets $\tilde{U}_k^j$:

\begin{lem}
  \label{lem:K_lin_replace_b}
  Let $K=K_{h}$ or $K=K_{co}$.
  Let $\tilde{U}_{k}^j$ be given by Definition \ref{defi:in_approx_sym}.
Assume that $M\in
\inte(K^{lc})$ with $M\in \tilde{U}_{k}^j$ for some $k\in
\N$ and  $j\in \{0,1,\dots,m-1\}$.
Let $\delta_k=\kappa_02^{-(10+k+2m)}/m$.
Then there exist a domain $\Omega^{\Diamond}$ of aspect ratio $1:\delta_k$ (or $1:\delta_k:1$) and a piecewise affine map $w:\Omega^{\Diamond} \rightarrow \R^n$ such that
\begin{enumerate}[label=(\roman*)]
\item $\nabla w(x) \in \tilde{U}_{k}^{j+1} \mbox{ for a.e. } x \in \Omega^{\Diamond}_g$.
\item $\Omega^{\Diamond}_g = \Omega^{\Diamond}$.
\item $w(x) = Mx \mbox{ for } x\in \partial \Omega^{\Diamond}$.
\item Let $\Omega_{1},\dots, \Omega_{N} \subset \Omega^{\Diamond}$ denote
  the level sets of $\nabla w$. Then,
  \begin{align*}
    \sum_{i=1}^{N} \Per(\Omega_{i}) \leq 2^{n+4} \Per(\Omega^{\Diamond}).
  \end{align*}
\end{enumerate} 
\end{lem}

\begin{proof}
The proof proceeds analogous to the previous one, but as we can in general only bound $\lambda(1-\lambda)\leq \frac{1}{4}$, we use the smallness assumption for the ratio $\delta_k$ to control the error $\epsilon_k$. 
Indeed, as above we have that
\begin{align*}
0\leq \epsilon_k \leq C\lambda (1-\lambda) \delta_k \leq \frac{1}{4}\kappa_0 2^{-(k+10+2m)}/m \leq 2^{-k-2-2m},
\end{align*}
which then also concludes the argument.
\end{proof}

\subsubsection{Skew control}
\label{sec:skew}

We seek to verify the assumption (A5). While the condition on the size of $\Omega^{\star}_g$ follows similarly as in the $O(n)$ case, the unboundedness of the sets $K$ implies that we additionally have to argue that it is possible to construct \emph{uniformly bounded} sequences $\nabla u_k$, in order to ensure the validity of the assumption (A5). To this end we use the flexibility in the choice of the skew part (c.f. Remark \ref{rmk:skew}). Heuristically, we give ourselves a ball of a fixed radius such that as long as the skew part of our constructions remains in this ball, we choose the skew part freely. If the skew part leaves this radius, we choose the sign of the skew part, such that we move from the exterior of the ball back into its interior. While this can directly be made rigorous in the case of $\Skew(2)$ (which is a one-dimensional linear space), the case of matrices in $\Skew(3)$ requires a little more care.

\begin{lem}
\label{lem:skew}
Consider the constructions from Lemma \ref{lem:K_lin_replace} and
\ref{lem:K_lin_replace_b}. Then in the application of Algorithm
  \ref{alg:convex_int} for any initial data $M \in \inte(K^{lc})$, we can choose
  the signs of the skew-parts in the application of the Lemmata such that  for
  any $k \in \N$ it holds that
\begin{align*}
\omega(\nabla u_k) \in B_{CR}(\omega(M)),
\end{align*}
where $R=R(m,n)>0$ is independent of $M$ and  $B_{CR}(\omega(M)):=\{\omega \in
\Skew(n): \ |\omega - \omega(M)|\leq CR\}$ with a constant $C=C(n,c_2, \diam(K))>1$.
\end{lem}

\begin{proof}
  Before coming to the formal proof, we give a brief overview of our strategy:
  
  As noted in Remark \ref{rmk:skew}, we may freely choose the sign of the skew part $S$
  in our rank-one connection.
  On any level set $\tilde{\Omega} \in \hat{\Omega}_{k}$ of our construction,
  we thus obtain a net change of the skew part that is (up to a controlled, geometrically decaying error) a sum of the terms
  \begin{align*}
    \pm \lambda S \mbox{ or } \pm (1-\lambda) S, 
  \end{align*}
  depending on how $\tilde{\Omega}$ was constructed.

  Since by definition of the in-approximation, we may assume that along the construction $(1-\lambda)$ converges to zero geometrically and
  each jump $S$ is bounded by two times the diameter $R$ of $\tilde{K}\subset \R^{n\times n}_{sym}$, we note that any
  series of the form
  \begin{align}
  \label{eq:sum_a}
    \sum_{i}  \pm (1-\lambda_{i})S_{i} \leq \sum c_{2}^{k} 2R \leq \frac{2R}{1-c_2} \leq C R
  \end{align}
  is absolutely convergent and hence all its partial sums are uniformly bounded.

  In contrast, the remaining series is of the form
  \begin{align}
  \label{eq:sum_b}
    \sum_{i} \pm \lambda_{i} S_{i}.
  \end{align}
 Here in general neither the coefficients $\lambda_i$ nor the matrices $S_i$ tend to zero, so there is no hope
  for the series \eqref{eq:sum_b} to be absolutely convergent.
  Instead we rely on a good choice of signs to ensure that all partial sums
  are uniformly bounded. The argument for this essentially reduces to a one-dimensional argument.
  Since any series of jumps in the skew part decomposes into the two cases \eqref{eq:sum_a}, \eqref{eq:sum_b}, this then
  establishes the desired bound.\\

  More formally, we make the following claim: 
\begin{claim}
\label{claim:skew}  
  There is a choice of signs
  \begin{align*}
    \epsilon: (\hat{\Omega}_{k})_{k \in \N} \mapsto \{-1,1\}
  \end{align*}
  with the following property:
  Let $\Omega^{1} \supset \Omega^{2} \supset \dots \supset
  \Omega^{k}$ be any sequence of descendants in our construction and denote by $\epsilon_{i}S_{i}$, where $\epsilon_i:=\epsilon(\Omega^{i})$, the
  choice of the respective skew part and by $\lambda_{i}$ the choice of the associated coefficient $\lambda_i$ in the convex splitting.
  Let further $\Lambda_{g} \subset \{1, \dots, k\}$ denote the iteration steps in which the
  descendant picked up the change by $(1-\lambda_{i}) \epsilon_{i}S_{i}$ and
  $\Lambda_{b}=\{1, \dots, k\} \setminus \Lambda_{g}$, the ``bad'' set, its complement.
  Then, it holds that 
  \begin{align}
      \label{eq:geometric}
    \sum_{i \in \Lambda_{g}} |(1-\lambda_{i}) \epsilon_{i} S_{i}| \leq \frac{2R}{1-c_2},
  \end{align}
  and 
  \begin{align}
    \label{eq:kinchin}
    |\sum_{i \in \Lambda_{b}} \lambda_{i} \epsilon_{i} S_{i} | \leq 2 R m^{2}. 
  \end{align}
  In particular, the net change of the skew part is uniformly bounded by
  \begin{align*}
    \frac{2R}{1-c} + 2 R m^{2}.
  \end{align*}
  \end{claim}
  We remark that \eqref{eq:geometric} is independent of the choice of $\epsilon$ and indeed
  follows by the geometric convergence of $(1-\lambda_{i})$ and the boundedness of $S_i$ as explained above.
  
  It hence remains to verify \eqref{eq:kinchin}.
  We note that $S_{i}$ always satisfies $|S_i|\leq 2\diam(\tilde{K})$ and is (up to normalization) always given by the skew part of a rank-one connection
  between the wells $e^{(1)}, \dots, e^{(m)}$ and can thus only take the $N=m(m-1)$
  values $\omega(a_{kl}\otimes n_{kl})$ (c.f. Lemma \ref{lem:skew}) for some $k,l \in \{1,\dots,m\}$, which we denote by $v_{1}, \dots, v_{N}$.
  
  We then claim that for any $k_{0} \in \N$ and any $\Lambda_{b}$ corresponding to a sequence of length
  up to $k_{0}$, $\epsilon_i$ can be chosen such that
  \begin{align}
    \label{eq:kinchin_sum}
    \begin{split}
    \sum_{i \in \Lambda_{b}} \lambda_{i} \epsilon_{i} S_{i}
 & \in \left\{ S \in \R^{n\times n}: \ S = \sum\limits_{j=1}^{N} \mu_j v_j, \ \mu_j \in [-1,1] \mbox{ for all } j \in \{1,\dots,N\} \right\}  \\
& \subset B_{2NR}(0),
\end{split}
  \end{align}
  which implies \eqref{eq:kinchin}.

  We establish \eqref{eq:kinchin_sum} by induction over $k_{0}$ and note that for
  $k_{0}=0,1$ it is satisfied (for any choice of $\epsilon$ by the triangle
  inequality).
  Thus suppose that the statement is true for any sequence up to length $k_{0}-1$
  and consider a sequence of length $k_{0}$.
  If in the $k_{0}$-th step, we chose
  $(1-\lambda_{k_{0}})\epsilon_{k_{0}}S_{k_{0}}$, then $k_{0} \not \in
  \Lambda_{b}$ and hence the statement follows by considering the truncation of
  the sequence to length $k_{0}-1$.
  Thus suppose $k_{0} \in \Lambda_{b}$, we then claim that for a suitable choice
  of $\epsilon_{k_{0}} (\Omega^{k_{0}})$ we achieve \eqref{eq:kinchin_sum} for
  this sequence. Since we did not modify $\epsilon_i$ on any other sequence up to level $k_{0}$
   (ancestors are unique), we can use this choice of $\epsilon_{k_{0}}
   (\Omega^{k_{0}})$ as our definition of the value of $\epsilon|_{\Omega^{k_0}}$ and thus conclude the induction step.
   It hence remains to show that such a choice is possible.
   We note that $\Lambda_{b}= \Lambda_{b}' \cup \{k_{0}\}$, where $\Lambda_b'$ is the
   ``bad'' set of the truncated sequence. Hence, by the induction assumption
   \begin{align}
   \label{eq:induc_kin}
   \begin{split}
     \sum_{i \in \Lambda_{b}} \lambda_{i} \epsilon_{i} S_{i} = \lambda_{k_{0}} \epsilon_{k_{0}} S_{k_{0}} +  \sum_{i \in \Lambda_{b}'} \lambda_{i} \epsilon_{i} S_{i} \\
     = \lambda_{k_{0}} \epsilon_{k_{0}} S_{k_{0}} + \sum_{i=1}^{N} \mu_{i} v_{i},
     \end{split}
   \end{align}
   for some $\mu \in [-1,1]^{N}$.
   Let further $l \in \{1, \dots, N\}$ be such that $S_{k_{0}}=v_{l}$ and write \eqref{eq:induc_kin}
   as
   \begin{align*}
     \lambda_{k_{0}} \epsilon_{k_{0}} v_{l} + \mu_{l} v_{l} + \sum_{i=1,i\neq l}^{N} \mu_{i} v_{i}.
   \end{align*}
  Choosing $\epsilon_{k_{0}}= -\sgn(\lambda_{k_{0}}) \sgn(\mu_{l})$ then implies
  that $\lambda_{k_{0}}\epsilon_{k_{0}} +\mu_{l} \in [-1,1]$ and thus establishes the claim.
\end{proof}

We summarize that the arguments in this section show that for any $M \in \inte(K^{lc})$ the constructions in Lemmas \ref{lem:K_lin_replace}, \ref{lem:K_lin_replace_b} can always be prescribed such that the full statement of \ref{item:A4} holds.

\subsubsection{Proof of Theorem \ref{thm:reg} in the case $K=K_h$ or $K=K_{co}$}
\label{sec:proof2}

\begin{proof}
The proof follows as in Section \ref{sec:proof1}:
The constructions of Lemmas \ref{lem:K_lin_replace} and \ref{lem:K_lin_replace_b} together with the covering arguments from Section \ref{sec:covering} (which are needed in order to work with a ``closed" class of domains) show the validity of assumptions (A3) and (A4) as well as the estimate 
\begin{align*}
|e(\nabla u) - e(M)|<C c_2^k \mbox{ on } (\Omega^{\Diamond})^{\star}_g.
\end{align*}
The argument in Section \ref{sec:skew} proves the boundedness of the gradient,
which hence yields \ref{item:A4}. As a consequence, combining these observations
with Lemmas \ref{lem:covering_reduction} and \ref{lem:covering_reduc_3D} implies that Theorem \ref{thm:reg_gen} is applicable. This thus provides the proof of Theorem \ref{thm:reg} for domains, which are given as a finite union of elements of $\mathcal{C}$. Invoking Lemma \ref{lem:Lip} then also yields the claim for general Lipschitz domains. 
\end{proof}

\section{The Precise Covering Constructions}
\label{sec:covering}

\subsection{The 2D covering construction}
\label{sec:2D}

In this section we introduce several coverings which allow us to extend our
constructions on the model domains $\Omega^{\Diamond}$ introduced in Section
\ref{sec:Ex} to the more general class $\mathcal{C}$ as defined below. 

\begin{defi}
The class $\mathcal{C}$ consists of arbitrary triangles.  The class
$\mathcal{C}^1$ is given by certain isosceles triangles whose aspect ratio (axis of symmetry to base) is $1:\delta$ for a given $\delta\in (0,1/2)$ and which are oriented in the same way as the current diamond-shaped replacement construction $\Omega^{\Diamond}$ (c.f. Lemma \ref{lem:covering_reduction} (i) for a precise statement on this).
\end{defi}

In Section \ref{sec:Ex} we showed that it is possible to verify the estimates
from Section \ref{sec:assumptions} for special diamond-shaped domains
$\Omega^{\Diamond}$. 
The class of diamond-shaped domains $\Omega^{\Diamond}$ is however not ``closed" under the replacement construction, in that it for instance produces level sets, which are triangles, which cannot be covered by finitely many diamonds. In particular, with only the class of diamonds at hand, we cannot directly apply Algorithm \ref{alg:convex_int} and iterate the replacement constructions.
Hence, it is necessary to extend the set of domains to a larger collection of sets denoted as $\mathcal{C}$ and to provide constructions which can be used on any domain in the class $\mathcal{C}$.
As a consequence, the class $\mathcal{C}$ is chosen in such a way that it is ``closed" under the replacement construction in the sense that if $\Omega \in \mathcal{C}$, all the level sets of the replacement deformation are again elements of $\mathcal{C}$. In particular, it contains all types level sets of the deformations arising in the replacement constructions from Section \ref{sec:Ex}. 

In order to achieve this, we begin with the construction, which we have at hand already, i.e. the diamond-shaped domains $\Omega^{\Diamond}$ and show that this can be used as a building block to construct more general domains with the desired properties stated in \ref{item:A3tilde}-\ref{item:A4}.
This is the content of Lemma \ref{lem:covering_reduction} where we argue by a step by step construction:
We first introduce specific isosceles triangles and square-shaped
domains $\Omega^{\Box}$ such that $\Omega^{\Box}\setminus \Omega_{g} \in
\mathcal{C}^{1}$ has an (essentially) self-similar geometry allowing for a very simple covering
construction. Here and in the following, we will call a set {\em self-similar} if it is similar to the set we started with. 
With this at hand, we then prove that for any $\Omega \in \mathcal{C}$ there exists a suitable
covering/partition involving $\Omega^{\Box}$ such that \ref{item:A3tilde}-\ref{item:A4} are satisfied.

\begin{figure}[t]
  \includegraphics[width=0.6\linewidth, page=5]{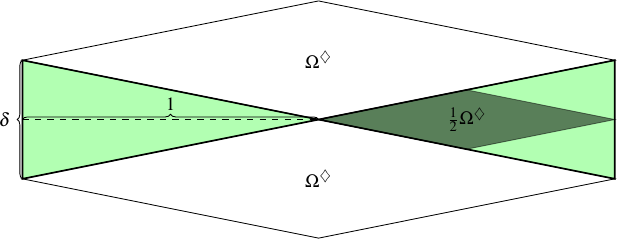}
  \includegraphics[width=0.4\linewidth, page=6]{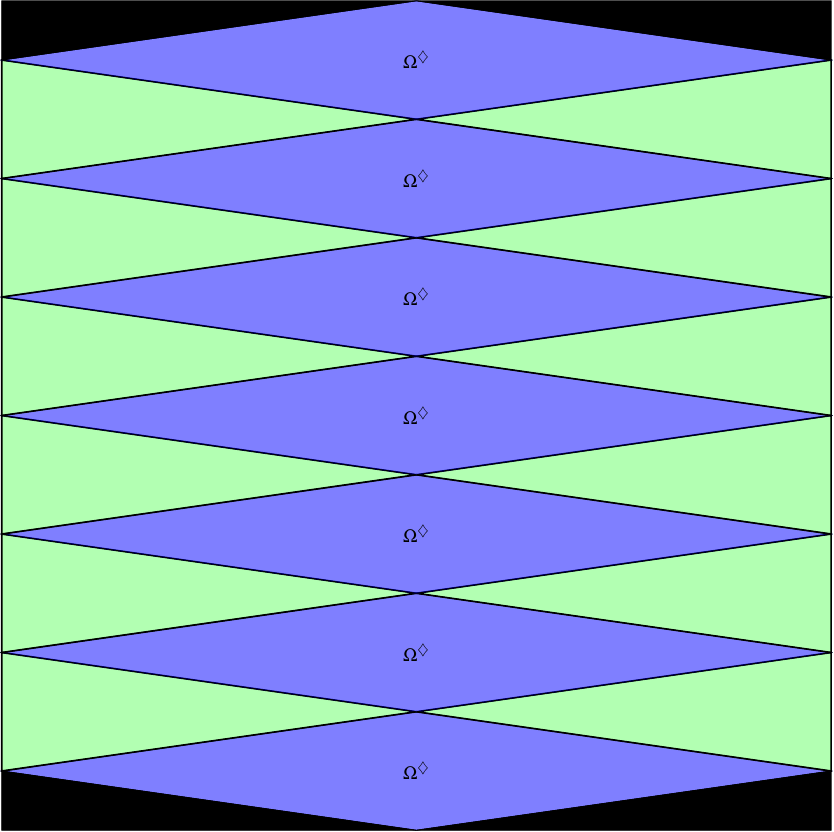}
  \caption{
  Illustration of the constructions used in Lemma \ref{lem:covering_reduction} (i), (ii). The figure on the top shows two isosceles triangles between two diamonds $\Omega^{\Diamond}$ (in white), to whose geometry these triangles are adapted. The triangles in $\mathcal{C}^1$ are thus constructed to fit into the gaps between vertical stackings of the diamonds $\Omega^{\Diamond}$. The right isosceles triangle is covered as explained in the proof of Lemma \ref{lem:covering_reduction} (i): We insert a self-similar diamond, and obtain as remainders two self-similar isosceles triangles. Due to the self-similarity of the covering, it is obvious that this covering strategy can be iterated.
The figure on the bottom illustrates the argument from the proof of Lemma \ref{lem:covering_reduction} (ii):
 Stacking $\lfloor\frac{1}{\delta}\rfloor$ many diamonds, we can
    construct a rectangle of aspect ratio similar to a square.
  Here, the light green areas are of type $\mathcal{C}^{1}$ and the black right-angled triangles are of controlled perimeter. }
  \label{fig:self-similar}
\end{figure}

\begin{figure}[t]
  \centering
  \includegraphics[width=0.6\linewidth, page=7]{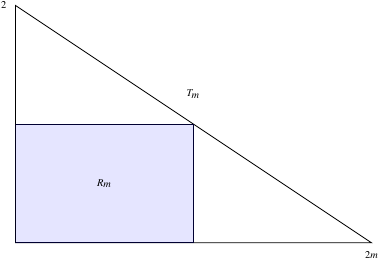}
  \includegraphics[width=0.35\linewidth, page=8]{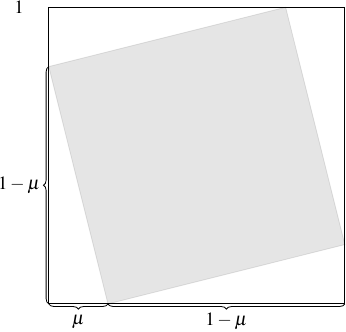}
  \includegraphics[width=0.7\linewidth, page=9]{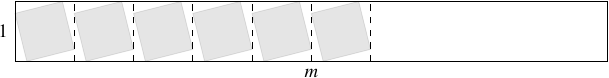}
  \caption{Any right-angled triangle $T_{m}$ can be partitioned into a rectangle
    $R_{m}$ and two self-similar triangles.
    At least half the volume of $R_{m}$ can then be filled by axis-parallel
    squares $S_i$, $i\in \{1,\dots,m\}$, and a rotated square $\Omega^{\Box}$ can be fitted inside each
    axis-parallel one.}
  \label{fig:fitting_square}
\end{figure}

\begin{lem}
  \label{lem:covering_reduction}
Assume that the estimates in conditions \ref{item:A3tilde}-\ref{item:A4} are satisfied, if $\Omega = \Omega^{\Diamond}$, where $\Omega^{\Diamond}$ is a diamond-shaped domain of aspect ratio $1: \delta$ for some $\delta \in (0,1/2)$. In the case \ref{item:A3tilde} the parameter $\delta$ may depend on
$k,j$, while in the case \ref{item:A4} it is required to be independent of
$k,j$. Suppose further that $\Omega_g = \Omega^{\Diamond}$, which in particular
yields that the conditions \ref{item:A3tilde} (ii), (iii) are empty for $\Omega
= \Omega^{\Diamond}$. Denote the constants from conditions \ref{item:A3tilde}
(i) and \ref{item:A3} by $C_{0}^{\star},C_{1}^{\star}$, if $\Omega = \Omega^{\Diamond}$, and assume that both are uniform in $j,k$.

\begin{enumerate}[label=(\roman*)]
  \item \label{item:selfsimilar_case}
If $\Omega$ is an isosceles triangle of aspect ratio $1:\delta$ oriented in in the same way as $\Omega^{\Diamond}$ (i.e. as depicted in Figure \ref{fig:self-similar} top), there exists a
replacement construction which satisfies \ref{item:A3tilde} (ii) or \ref{item:A3} with $C_{0}=C_{0}^{\star}$, $C_{1}=C_{1}^{\star}$, $C_{2}=2$ and $v_{1}=1/2$.
We include the corresponding isosceles triangle in the set $\mathcal{C}^1$.
\item \label{item:box_case}
If $\Omega = \Omega^{\Box}$ is a rectangle of aspect ratio $1: \delta \cdot \lfloor{\frac{1}{\delta}} \rfloor$, there is a replacement construction such that the conditions \ref{item:A3tilde}-\ref{item:A4} are satisfied with constants
  $C_{0}=C_{0}^{\star}/\delta$, $C_{1}=C_{1}^{\star}/\delta$, $C_{2}=8$ and $v_{1}=1/2$.
\item \label{item:general_case}
  For any $\Omega \in \mathcal{C}$ there exists a construction such that the conditions \ref{item:A3tilde}-\ref{item:A4} are satisfied with constants $C_{0}=100C_{0}^{\star}/\delta$,
  $C_{1}=100C_{1}^{\star}/\delta$, $C_{2}=100$ and $v_{1}\geq \frac{1}{100}$. 
  \end{enumerate}
\end{lem}

\begin{rmk}
\label{rmk:constants}
Due to the uniform dependence of the constants $C_0^{\star}, C_1^{\star}$ on $k,j$ and the $k,j$ (in-)dependence of $\delta$, we in particular infer the required $k,j$ (in-)dependence of the constants $C_0, C_1, C_2$.
\end{rmk}

We recall that the replacement construction for the model domains $\Omega^{\Diamond}$ has been
established for several differential inclusion problems in Section \ref{sec:Ex}.
The following construction makes use of self-similar tiling properties of these
model domains (c.f. Figure \ref{fig:self-similar}) in order to derive sufficiently strong perimeter
estimates. We recall that this allows us to establish the good $BV$ estimates from
Section \ref{sec:BV} and to deduce $W^{s,p}$ regularity of our
convex integration solution with $s,p$ \emph{independent} of our initial data.
If we instead accepted such a dependence, we could allow $C_{2}=C_{0}$ in
\ref{item:A3tilde} and $\mathcal{C}^{1}=\emptyset$ and employ a simpler covering
strategy.

\begin{proof}
In this proof we proceed step by step and construct each covering/partition using the preceding ones.
We remark that in all these constructions $\Omega_{g}$ is composed of copies of $\Omega^{\Diamond}$.  Denoting the replacement deformation in $\Omega^{\Diamond}$ by $w$, we always obtain the new replacement construction $u:\Omega \rightarrow \R^2$ as follows:
The function $u$ is always constructed to be piecewise affine, such that $u(x)=Mx+b$ in $\Omega \setminus \Omega_g$ (where $Mx+b$ with $M \in U^{k,j}$ or $\tilde{U}^{k,j}$ is the deformation that is to be replaced) and such that $u$ is a translated and rescaled version of $w$ in the respective diamonds, into which $\Omega_g$ is decomposed.
In particular, $\nabla u \in U^{j}_{k}$ is constant and unchanged outside $\Omega_{g}$. As $|(\Omega^{\Diamond})^{\star}_g| \geq (1-C2^{-k})|\Omega^{\Diamond}|$, condition \ref{item:A4} is always satisfied. It hence remains to ensure that $v_1=\frac{|\Omega_g|}{|\Omega|}$ is sufficiently large and that the perimeter bounds are true.

In the sequel, we will hence not discuss the construction of $u$, but only describe the underlying partitioning of the respective domains, the function $u$ being understood to be constructed as just outlined.\\
  
\underline{Case \ref{item:selfsimilar_case}:}
Let $\Omega \in \mathcal{C}^{1}$ be an isosceles triangle compatible with the geometry of
  $\Omega^{\Diamond}$ as depicted in Figure \ref{fig:self-similar}, top. Here when writing that the isosceles triangle is compatible with the geometry of $\Omega^{\Diamond}$, we mean that if $\Omega^{\Diamond}=\conv(\{(0,\pm\delta),(\mp 1,0)\})$, then the isosceles triangle is up to rescaling and translation given by $\conv(\{(0,\pm \delta),(-1,0)\})$ or $\conv(\{(0,\pm \delta),(-1,0)\})$. If $\Omega^{\Diamond}$ is rotated, then any compatible isosceles triangle is analogously rotated.
We cover $\Omega$ by a (by a factor $t\in \R_+$ rescaled) version $\Omega_{t}^{\Diamond}$ of $\Omega^{\Diamond}$ and two rescaled copies of $\Omega$, 
  rescaled by a factor $1/2$ (these are the dark green diamond and the two light green triangles in Figure \ref{fig:self-similar} on the top right). Setting $\Omega_g:= \Omega_t^{\Diamond}$, all conditions on $\Omega\setminus \Omega_{g}$ are thus satisfied by self-similarity with $C_{2}\in [1,4)$, while all conditions
  on $\Omega_{g}$ are satisfied by assumption on $\Omega^{\Diamond}$ and
  $|\Omega_{g}|=\frac{1}{2}|\Omega|$.
In addition to \ref{item:A4}, this shows the validity of \ref{item:A3tilde} (i), (ii) and \ref{item:A3} (the other cases in \ref{item:A3tilde} being empty).  \\

\underline{Case \ref{item:box_case}:}
Given a rectangle $\Omega^{\Box}$, we can cover half of its area by stacking $\lfloor{\frac{1}{\delta}}\rfloor$ copies of $\Omega^{\Diamond}$ as depicted in Figure \ref{fig:self-similar} bottom. Denoting the union of these copies of
$\Omega^{\Diamond}$ by $\Omega_{g}$, we
  observe that $\Omega^{\Box}\setminus \Omega_{g}$ consists of $2
  (\lfloor{\frac{1}{\delta}} \rfloor -1)$ triangles of type $\mathcal{C}^{1}$ and four right-angled triangles
  at the top and bottom. 
  In the notation of \ref{item:A3tilde}, these collections correspond to
  $\Omega^{[1]}$ (consisting of the $2
  (\lfloor{\frac{1}{\delta}} \rfloor -1)$ triangles of type $\mathcal{C}^{1}$) and $\Omega^{[2]}$ (consisting of the four remaining triangles), respectively.

  Then it holds that
  \begin{align*}
    \sum_{\tilde{\Omega} \in \Omega^{[1]}} \Per(\tilde{\Omega}) &\leq \sum_{\tilde{\Omega} \in \Omega^{[1]}} 2 \Per(\Omega) \leq 2 \lfloor{ \frac{1}{\delta} \rfloor} \Per(\Omega)=: C_{0} \Per(\Omega), \\
    \sum_{\tilde{\Omega} \in \Omega^{[2]}}\Per(\tilde{\Omega}) &\leq 8 \Per(\Omega) \leq  C_2 \Per(\Omega), \\
    |\Omega_{g}|&=\frac{1}{2}|\Omega^{\Box}|.
  \end{align*}
  In addition to \ref{item:A4}, this shows the validity of \ref{item:A3tilde} (i), (iii) and \ref{item:A3} (the other cases in \ref{item:A3tilde} being empty).\\

\underline{Case \ref{item:general_case}:} Let $\Omega \in \mathcal{C}$ be a given triangle, then drawing a perpendicular,
  we may partition this triangle into two right-angled triangles whose perimeters are controlled by $\Per(\Omega)$. We may thus without loss of generality assume that $\Omega$ is a right-angled triangle and
  may also choose a coordinate system such that $\Omega$ is axis-parallel. Further by scaling and rotation, we may assume that it has side-lengths $2$ and $2m$ with $m\in \R_+, m \geq 2$ as depicted in Figure \ref{fig:fitting_square} (left).

Then, we can partition $\Omega$ into two self-similar triangles of half the lengths and a rectangle $R_m$, as depicted in Figure \ref{fig:fitting_square} (left).
  Since the perimeter of both triangles is controlled by $1/2 \Per(\Omega)$, the perimeter of $R$ by $\Per(\Omega)$, and $|R_m|=\frac{1}{2}|\Omega|$, it suffices to show that there exists a replacement construction on $R_m$.
We can cover at least
  half the volume of $R_m$ by $\lfloor{\frac{m}{2}}\rfloor$ squares $S_{i}$ of side length
  $1$ as illustrated in Figure \ref{fig:fitting_square} (bottom).
  The remainder of $R_m$ can then be split into two (self-similar) right-angled triangles, whose
  perimeter is controlled by $\Per(R_m)$ and the combined perimeter of all squares
  can be controlled by $2 \Per(R_m)$. Thus, again it suffices to provide a
  construction on each square $S_{i}$.
  As illustrated in Figure \ref{fig:fitting_square} (right), a suitable rescaling
  $\Omega^{\Box}_{i}$ of $\Omega^{\Box}$ (i.e. of the domain from \ref{item:box_case}) can be fitted inside $S_{i}$ such that
  $|\Omega^{\Box}_{i}|\geq \frac{1}{4}|S_{i}|$ and such that $S_{i}\setminus \Omega^\Box$
  consists of four right-angled triangles of comparable perimeter.
  Finally, we apply the construction on $\Omega^{\Box}_{i}$ inside each
  $S_{i}$ and note that $\Omega\setminus \bigcup \Omega^{\Box}_{i}$ consists of right-angled triangles that we collect in  $
  \Omega^{[1]}$.  The conditions \ref{item:A3tilde}-\ref{item:A3} are then
  satisfied with the claimed constants. 
    In addition to \ref{item:A4}, this hence shows the validity of \ref{item:A3tilde} (i), (iii) and \ref{item:A3}.
\end{proof}

\subsection{The 3D covering construction}
\label{sec:3D}

In this section we introduce several coverings and partitions in the three-dimensional setting to extend the replacement
construction on a symmetric, three-dimensional diamond $\Omega^{\Diamond}$ to more
general domains $\Omega \in \mathcal{C}$, as defined in Definition \ref{defi:C3D}.
In this context, the outline of the argument is analogous to the two-dimensional setting. Given a certain class of model domains, we first argue that it always suffices to construct a covering of a box by the diamond shaped domains from Lemmas \ref{lem:K_lin_replace_b}, \ref{lem:K_lin_replace}, \ref{lem:On_replace_in}, \ref{lem:On_replace}. This is the content of Section \ref{sec:boxes}, where we introduce a number of auxiliary domains, which are needed to obtain a ``closed" collection of sets $\mathcal{C}$. In Section \ref{sec:3d_self_similar} we then show that these boxes can indeed be covered by diamonds and provide a self-similar refinement which is analogous to Lemma \ref{lem:covering_reduction} (ii). Combined with the results in Section \ref{sec:boxes} this then concludes the three-dimensional covering construction in such a way that we obtain $W^{s,p}$ estimates, whose differentiability and integrability exponents $sp$ are \emph{uniform} in the position of the boundary data in $\inte(K^{lc})$. We again emphasize that if were willing to give up the uniformity of these estimates, a simpler covering argument would suffice.

\subsubsection{Reduction to boxes}
\label{sec:boxes}

We begin by introducing an auxiliary set $\mathcal{C}^0$ of model domains. Together with the class $\mathcal{C}^1$, which is defined in Section \ref{sec:3d_self_similar} this then constitutes the class of domains $\mathcal{C}$.

\begin{defi}
  \label{defi:3D_shapes}
In the three-dimensional setting the class $\mathcal{C}^0$ is built from several model shapes (c.f. Figure \ref{fig:3D_shapes}):
\begin{enumerate}
\item[(i)] A \emph{ramp} or \emph{triangle prism} is a domain of the type $T' \times I$, where
  $T'$ is a triangle and $I$ is an interval. Bisecting the triangle using a perpendicular, we may
  without loss of generality assume that $T'$ is a right-angled triangle. 
\item[(ii)] A \emph{triangle pyramid} is the convex hull of a right-angled triangle $T'$ and a
  point $P$, which we require to be vertically above one of the corners. That
  is, there exists a corner $Q$ of the triangle such that $P-Q$ is orthogonal to
  the plane containing $T'$. 
\item[(iii)] A \emph{rectangle pyramid} is the convex hull of a rectangle $R$ and a
  point $P$, which we require to be vertically above one of the corners.
\end{enumerate}
A set $\Omega$ is then an element of the class $\mathcal{C}^0$ if it can be expressed as the
union of up to $30$ of these model sets.
\begin{figure}[t]
  \centering
  \includegraphics[width=0.2\linewidth, page=10]{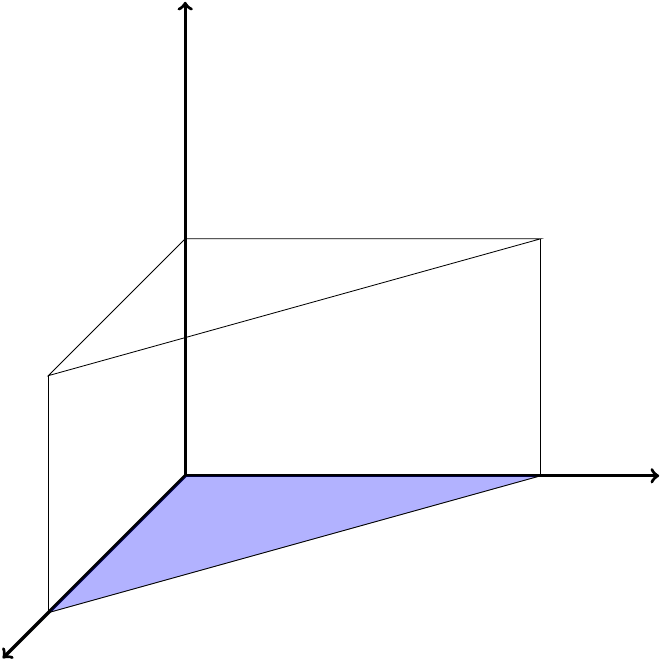}
  \includegraphics[width=0.2\linewidth,page=11]{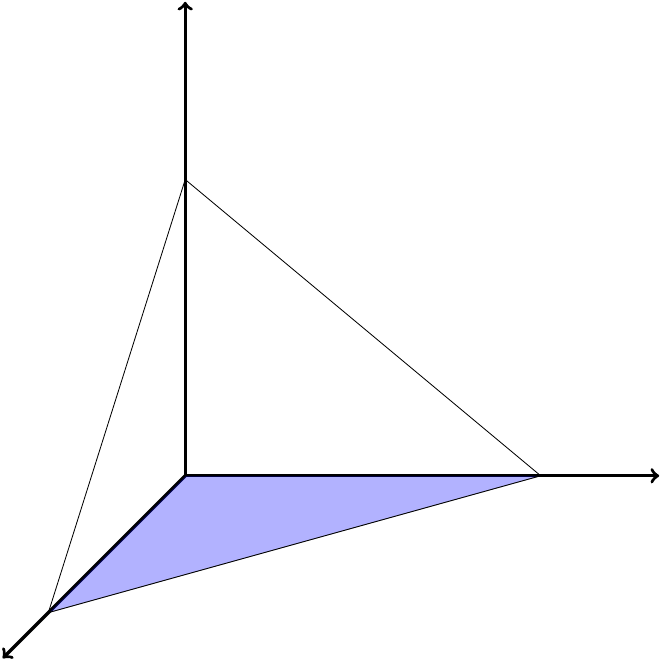}
   \includegraphics[width=0.2\linewidth,page=12]{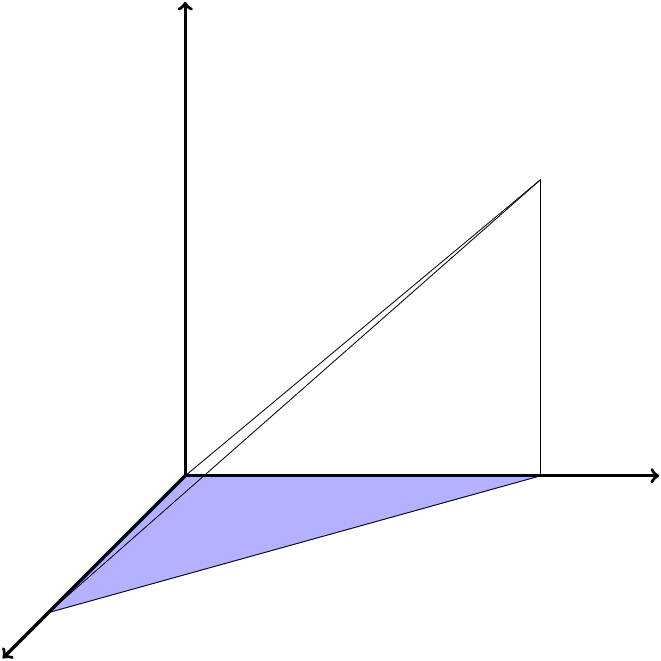}
   \includegraphics[width=0.2\linewidth, page=13]{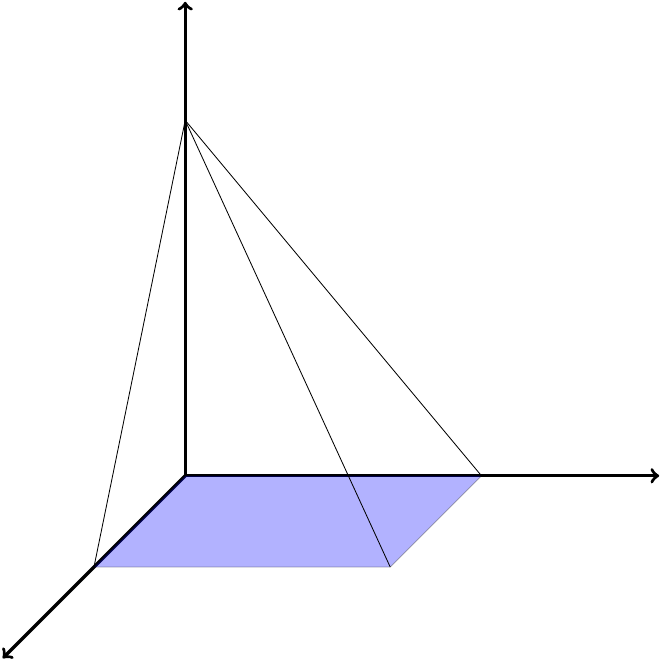}
  \caption{The model domains for our 3D construction are ramps, triangle
    pyramids and rectangle pyramids.}
  \label{fig:3D_shapes}
\end{figure}
\end{defi}

In the sequel, we discuss how the existence of the replacement construction for $\Omega^{\Diamond}$ (c.f. Lemmas \ref{lem:On_replace_in}, \ref{lem:On_replace}, \ref{lem:K_lin_replace_b}, \ref{lem:K_lin_replace}) implies the existence of suitable replacement constructions satisfying the requirements \ref{item:A3tilde}-\ref{item:A4} for the building blocks from Definition \ref{defi:3D_shapes}. To this end, it suffices to concentrate on the necessary covering. Indeed, constructing this in a way such that for any $\Omega \in \mathcal{C}$ we always have that $\Omega_g$ consists of a union of translated and rescaled model domains $\Omega^{\Diamond}$, we implicitly understand the associated replacement deformation $u:\Omega \rightarrow \R^3$ to be a piecewise affine function, which is a translated and scaled model deformation on $\Omega_g$ (i.e. one of the deformations from Lemmas \ref{lem:On_replace}, \ref{lem:On_replace_in}, \ref{lem:K_lin_replace}, \ref{lem:K_lin_replace_b}), and which is unchanged in $\Omega \setminus \Omega_g$. Provided that $|\Omega_g|$ is sufficiently large compared to $|\Omega|$, all the requirements on $u$ are therefore satisfied. Hence, we only discuss the covering construction in the sequel.

\begin{lem}
  \label{lem:3D_reduction2}
  Let $R$ be a rectangle and let $P_{5}$ be a point such that there exists
  $\tilde{P}_{5} \in R$ so that $P_{5}-\tilde{P}_{5}$ is orthogonal to $R$. Then 
  the shape obtained as the convex hull of $R$ and $P_{5}$ can be partitioned in
  four rectangle pyramids as described in Definition \ref{defi:3D_shapes}.

  Similarly, let $T'$ be a triangle and let $P_{4}$ be a point such that there
  exists $\tilde{P}_{4} \in T'$  so that $P_{4}-\tilde{P}_{4}$ is orthogonal to
  $T'$. Then the tetrahedron obtained as the convex hull of $T'$ and $P_{4}$ can
  be partitioned into six triangle pyramids as described in Definition \ref{defi:3D_shapes}.
\end{lem}

\begin{proof}
  We partition $R$ into four rectangles $R_{1}, R_{2}, R_{3}, R_{4}$, which all
  have $\tilde{P}_{5}$ and one of the corners of $R$ as opposite corners.
  Then the convex hulls of each rectangles and $P_{5}$ are rectangle pyramids
  due to the assumed orthogonality and yield a partition of the convex hull of
  $R$ and $P_{5}$.
  \\

  Let $P_{1}, P_{2}, P_{3}$ denote the corners of $T'$. Then we may draw
  perpendiculars from $\tilde{P}_{4}$ to the sides of the triangle opposite of $P_{1}, P_{2}$ or $P_{3}$
  and obtain points $Q_{1}, Q_{2},Q_{3}$, respectively. In this way we
  may partition $T'$ into six right-angled triangles $\tilde{P}_{4}Q_{i}P_{j}$
  with $i, j \in \{1,2,3\}$ and $i\neq j$.
  Then, by the assumed orthogonality of $P_{4}-\tilde{P}_{4}$, each tetrahedron $\tilde{P}_{4}Q_{i}P_{j} P_{4}$ is a triangle pyramid.
\end{proof}

We next show that given a possibly rotated box (which should be thought of as the convex hull of a stacking of diamonds analogous to the two-dimensional stacking described in the proof of Lemma \ref{lem:covering_reduction} (ii), c.f. Section \ref{sec:3d_self_similar}), it is possible to fit this in a controlled way into a slightly larger axis parallel cube.

\begin{lem}
\label{lem:Euler_angles}
There are constants $c\in (0,1)$ and $C>1$ with the following properties: 
Let $W_\delta\subset\R^3$ be a box of side length ratio $1:1:\delta \lfloor \frac{1}{\delta}\rfloor$. Then there is an axis-parallel cube $W$ such that 
\begin{itemize}
\item $W_\delta\subset W$, $|W\setminus W_\delta|\leq c|W|$,
\item $W\setminus W_\delta$ is either empty or can be written as a union of at most $20$ elements $B_i \in \mathcal{C}^0$, and 
\[\sum_j \Per(B_j)+\Per(W_\delta)\leq C\Per (W). \]
\end{itemize}
\end{lem}

\begin{figure}[t]
  \centering
  \includegraphics[width=0.3\linewidth, page=14]{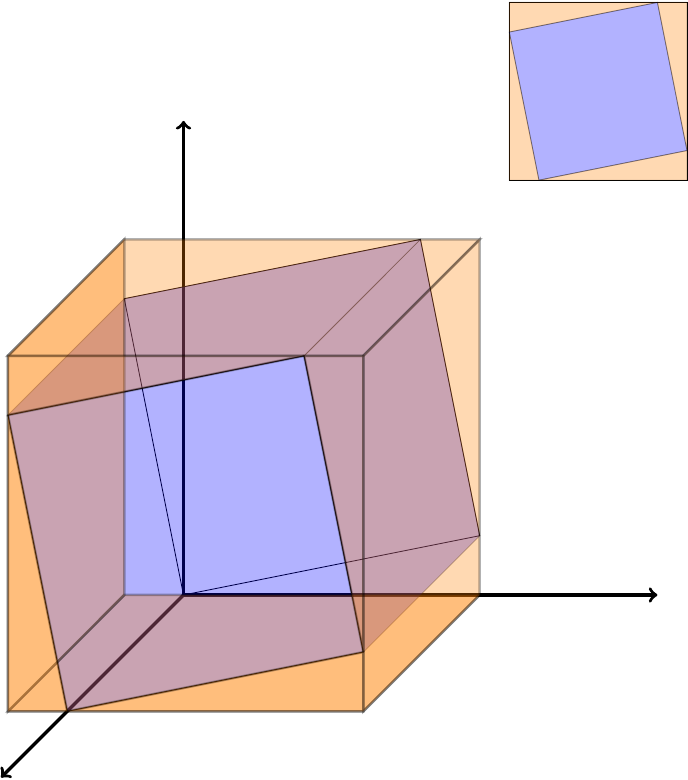}
  \caption[Euler angles]{Fitting a rotated cube into an axis-parallel one in the
    case of an axis-parallel rotation. Decomposing a general rotation into Euler
    angles allows to reduce to this situation. As the original box $W_{\delta}$ is in general not exactly a box, we potentially need to fill the remainder of the inner-most cube up with ramps.}
  \label{fig:Euler_angles}
\end{figure}

\begin{proof}
There exists a rotation $R\in SO(3)$ such that $R^t W_\delta$ is
axis-parallel. In the case that the rotation axis is given by a coordinate axis, we can use the $2$D construction to place the box into a slightly larger axis parallel cube, where the remainder is given by ramps, see Figure \ref{fig:Euler_angles}. 

Consider now the case of a general rotation $R\in SO(3)$. Using Euler angles,
every rotation can be decomposed into three rotations $R_x$, $R_y$, $R_z$ around
the respective coordinate axes, $R=R_zR_yR_x$. We thus iterate the above
construction to build an axis-parallel box: First start with an axis parallel
box $\tilde{W}$. Consider the rotated box $R_x\tilde{W}$. As described above,
this box can be placed in a slightly larger axis-parallel  cube $W_1$, where the
complement consists of ramps. Now consider $R_y W_1$. Note that inside
$W_1$ the box $\tilde{W}$ is now rotated by $R_yR_x$. Again, there exists
a slightly bigger axis-parallel cube such that $R_yW_1$ lies inside it and
the complement is filled by ramps. Finally rotating this cube by $R_z$ and
repeating the construction yields the desired cube $W$. Here, the triangle
inequality ensures that in every step the exterior box is at most increased by a
factor of $\sqrt{2}$. Again, we fill the remainder up with ramps. 
\end{proof}

As in the two-dimensional case, given an axis-parallel rectangle, we can cover at least half of its volume by axis-parallel cubes. 
We hence next address the problem of fitting a rectangle into an arbitrary element of $\mathcal{C}^0$.

\begin{lem}\label{lem:box3Dint}
Every building block $B$ as described in Definition \ref{defi:3D_shapes}, contains a three-dimensional rectangle $R$ of volume fraction at least $3/16$, such that the complement $B\setminus R$ is (up to null-sets) a disjoint union of at most seven building blocks $U_\ell \in \mathcal{C}^{0}$, and  $ \sum_{\ell} \Per(U_\ell)+\Per (R)\leq 8\Per(B)$.
\end{lem}
\begin{proof}
We consider the three cases of Definition \ref{defi:3D_shapes} separately (c.f. Figure \ref{fig:3D_fit}). 
\begin{figure}[t]
  \centering
  \includegraphics[width=0.2\linewidth, page=15]{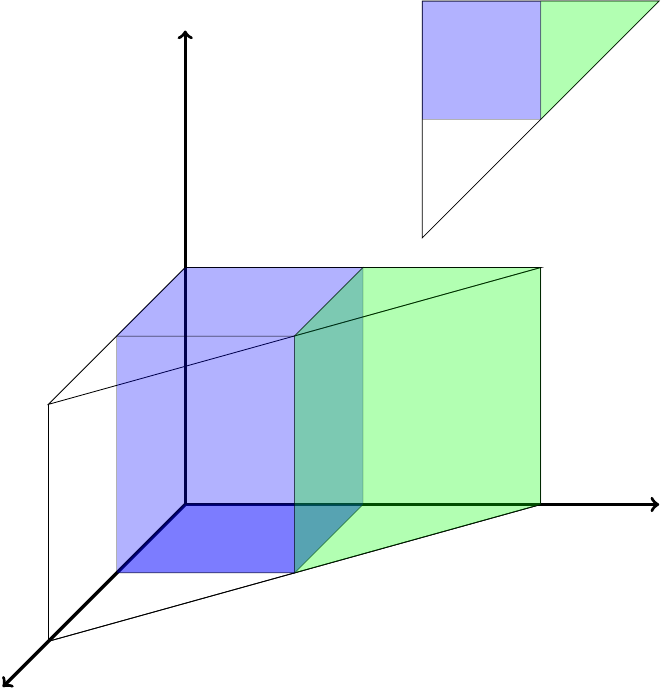}
   \includegraphics[width=.2\linewidth, page=18]{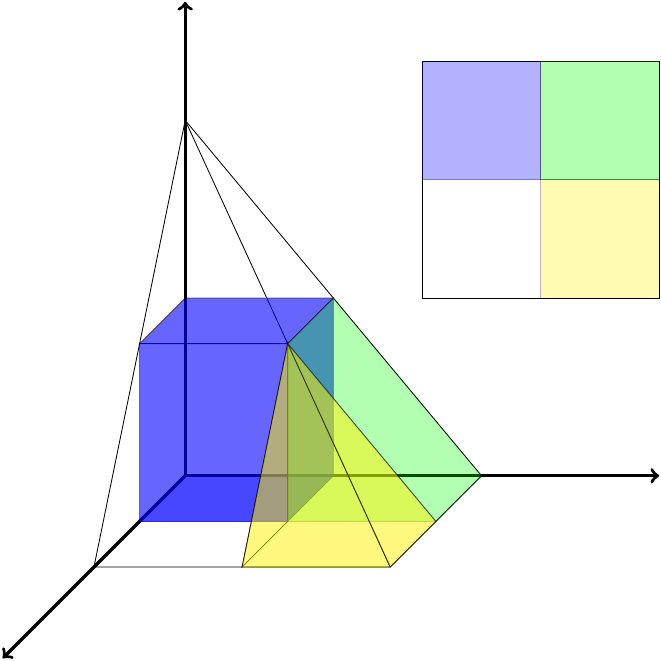}
   \includegraphics[width=.2\linewidth, page=17]{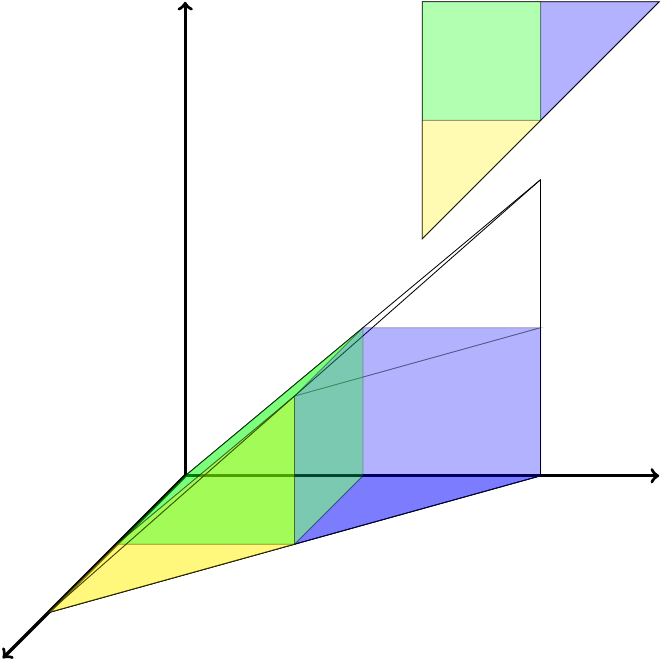}
   \includegraphics[width=.2\linewidth, page=16]{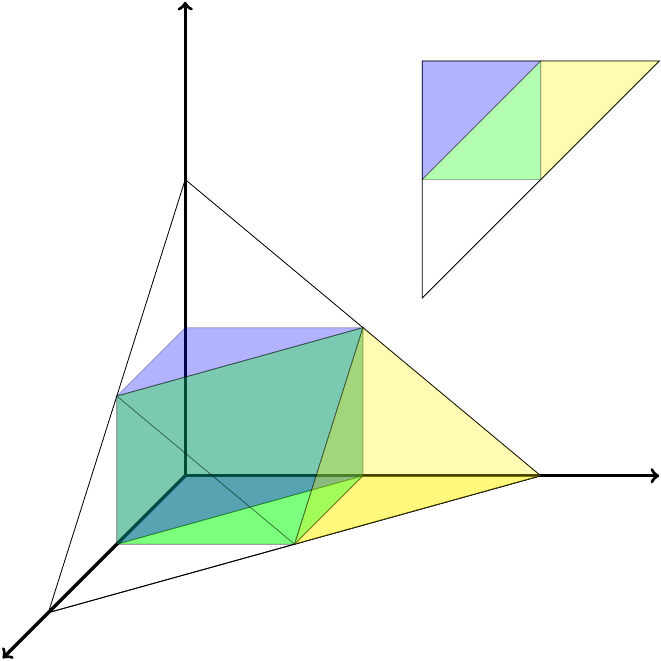}
  \caption{Fitting a rectangle into a building block.}
  \label{fig:3D_fit}
\end{figure}
\begin{itemize}
\item[(i)] 
For ramps $B=T' \times I$, we use the $2$D-construction, see left panel of Figure \ref{fig:3D_fit}. Precisely, in the right-angled base triangle $T'$ we fit a rectangle $\tilde{R}$ of half the area, and define the $3$D-rectangle $R$ as $\tilde{R}\times I$. Then this rectangle has half the volume of the ramp, and the complement consists of two ramps $U_1,U_2$ that are similar to the original ramp. Denoting by $a$ and $b$ the lengths of the legs of the base triangle, and by $h$ the height of the ramp, we have $\Per (B)=ab+h\sqrt{a^2+b^2}+ah+bh$ and $\Per(R)=\frac{1}{2}ab+ah+bh$. Since the two ramps in $B\setminus R$ have only faces which are (up to translation) contained in the surfaces of $B$ and $R$, we can very roughly estimate $ \sum_{\ell=1}^{2}\Per(U_\ell)+\Per (R)\leq 3 \Per(B)$.
\item[(ii)] Consider now a rectangle pyramid $B$, see second panel of Figure \ref{fig:3D_fit}. The base rectangle contains a rectangle of half the length and half the width which has as one corner the vertex of the pyramid above which the top is located. Then the $3$D rectangle which has this smaller rectangle as one face and is of half the height of the pyramid is contained in the pyramid, and has $1/8$ of its volume. Further, denoting the side lengths of the base rectangle by $a$ and $b$, and the height by $h$, we have $\Per(B)=ab+\frac{1}{2}ah+\frac{1}{2}bh+\frac{1}{2}a\sqrt{b^2+h^2}+\frac{1}{2}b\sqrt{a^2+h^2}$, and $\Per (R)=\frac{1}{2}(ab+ah+bh)$. \\
Consider now the complement, which has four components $U_1,\dots,U_4$. We claim that two of them are ramps and the other two are rectangle-based pyramids which are similar to the original one. Consider first the rectangle-based pyramids (the yellow and the top part in the figure). By construction, their base is a rectangle of half the side lengths and the height is half the height of the original rectangle-based pyramid. The remaining two parts which lie above the green and the white part of the base rectangle are ramps by construction. For every set $U_\ell$, we have $\Per(U_\ell)\leq \Per (B)$, and the assertion follows.
\item[(iii)] Consider now triangle-based pyramids, see the two right panels of Figure
  \ref{fig:3D_fit}.

 We begin by discussing the case when the corner $Q$ is not above the right angle of the
  triangle $T'$, which is illustrated in the third panel of Figure \ref{fig:3D_fit}.
  Then the base triangle can be partitioned into two self-similar triangles of
  half the side lengths and a rectangle of half the side lengths.
  This partition of the base then yields a partitioning of the pyramid into two
  self-similar triangle pyramids (yellow and white) and two ramps (green and
  purple).
The volume of the ramps is then $3/8$ of the volume of the triangle-pyramid, and the perimeters of all four objects are not bigger than the perimeter of the triangle-pyramid. Hence, the assertion follows by (i).

  Finally, consider the case when the top $Q$ is above the right angle of the
  base triangle $T'$, which is depicted in the fourth panel of Figure \ref{fig:3D_fit}. We first introduce a ramp of half the side lengths of the
  pyramid around the corner of the base triangle at which the right angle is
  located (violet in the figure). The pyramid on top of this ramp also has a
  base of half the side lengths of the original one, and half its height, and is
  therefore similar to it with $1/8$ of its volume. Similarly, for the
  yellow ones. It remains to consider the green part, which is a rectangle-based
  pyramid whose base is the face that it shares with the ramp. Its top is not
  above a corner but in view of Lemma \ref{lem:3D_reduction2}, we can split this
  pyramid into shapes of class (ii) in Definition \ref{defi:3D_shapes}.
  Since the top corner lies already above an edge, we only need two pyramids. We
  now use the construction from (i) to fit a box into the ramp. Note that this
  box has $1/2$ of the volume of the ramp, and thus $1/2\cdot 3/8=3/16$ of the
  volume of the original triangle-based pyramid. Denoting the legs of the base triangle by $a,b$ and the height of the pyramid by $h$, we have
  $\Per(B)=\frac{1}{2}(ab+ah+bh +\sqrt{a^2+b^2} \sqrt{a^2/4+b^2/4+h^2})$ and $\Per(R)= \frac{1}{8} ab + \frac{1}{4}a h + \frac{1}{4} bh$. The assertion follows.
\end{itemize}
\end{proof}

As an additional step compared to the two-dimensional setting, in three dimensions we have to
verify that each level set of the construction in $\Omega^{\Diamond}$ (c.f. Lemmas \ref{lem:3D_O(n)}, \ref{lem:lin_Conti_3D}) can be
partitioned into (a small number of) the above model domains.
\begin{lem}
  Let $\Omega_{1}, \dots, \Omega_{N}$ denote the level sets of the
  three-dimensional replacement construction from Lemmas \ref{lem:3D_O(n)}, \ref{lem:lin_Conti_3D} (c.f. Figure \ref{fig:3D_diamond}).
  Then each set is contained in the class $\mathcal{C}^0$ given in Definition \ref{defi:3D_shapes}.
\end{lem}

\begin{figure}[t]
  \centering
      \includegraphics[width=0.4\linewidth, page=19]{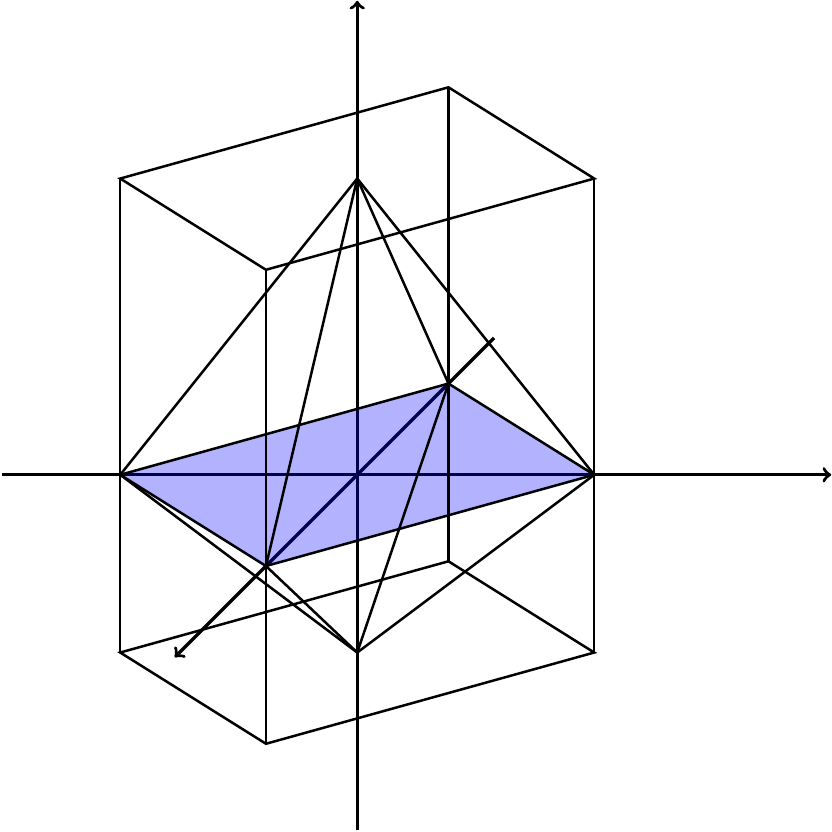}
  \includegraphics[width=0.55\linewidth, page=4]{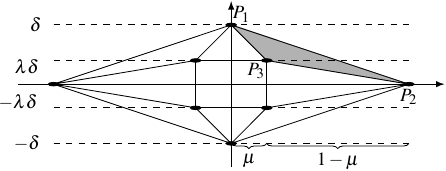}
  \caption{The 3D diamond and the 2D diamond. The three-dimensional diamond is obtained by embedding the two-dimensional construction which is depicted on the right into the $(x,y,0)$-plane and adding the points $P_4^{\pm}:=(0,0,\pm 1)$. Setting $u(P_4^{\pm})=0$ and interpolating linearly between the boundary points of the two-dimensional construction then yields the piecewise affine three-dimensional construction (c.f. Lemmas \ref{lem:3D_O(n)}, \ref{lem:lin_Conti_3D}).}
  \label{fig:3D_diamond}
\end{figure}

\begin{proof}
We recall that the set $\Omega^{\Diamond}$ was obtained by
interpolating the 2D construction from Lemmas \ref{lem:MS}, \ref{lem:lin_Conti} with a point vertically
above the centre of the 2D diamond (c.f. Figure \ref{fig:3D_diamond} and also Figure \ref{fig:kite}).

Let us thus fix a coordinate system such that $P_{4}=(0,0,1)$ is the vertical
point at the top and $P_{2}=(1,0,0)$, $P_{1}=(0,\delta,0)$, $P_{3}=(\mu,
\lambda\delta,0)$ and $O=(0,0,0)$ are corners of the level sets of the 2D construction.
Furthermore, we introduce $\tilde{P}_{3}=(\mu,-\lambda \delta,0)$ and
$\hat{P}_{3}=(-\mu, \lambda \delta,0)$.

We note that the tetrahedron $OP_{3}\tilde{P}_{3}P_{4}$ can be bisected into
two triangle pyramids by introducing the point $P_{6}= \frac{1}{2}P_{3} +
\frac{1}{2}\tilde{P}_{3}$.
Hence, by symmetry it only remains to discuss the three tetrahedra
$P_{3}\tilde{P}_{3}P_{2}P_{4}$, $P_{3}\hat{P}_{3}P_{1}P_{4}$ and
$P_{1}P_{2}P_{3}P_{4}$.\\

\underline{The tetrahedron $P_{3}\tilde{P}_{3}P_{2}P_{4}$:}
We claim that there exists a point $P_{5}$ contained in the line $P_{2}P_{4}$
  such that the triangle $P_{3}\tilde{P}_{3}P_{5}$ is orthogonal to the vector
  $P_{2}-P_{4}$ (with slight abuse of notation, we do not distinguish carefully between vectors and points in the sequel).
  If this is the case, the tetrahedron $P_{3}\tilde{P}_{3}P_{2}P_{4}$  can be
  expressed as the disjoint (up to null sets) union of two triangle pyramids
  $P_{3}\tilde{P}_{3}P_{5}P_{4}$ and $P_{3}\tilde{P}_{3}P_{5}P_{1}$.

  Indeed, let $l=\frac{\mu+1}{2} \in (0,1)$ and define $P_{5}= l P_{2}
  + (1-l)P_{4}$ and $P_{6}=(\mu,0,0)=\frac{1}{2}P_{3}+\frac{1}{2}\tilde{P}_{3}$.
  Then a short computation yields
  \begin{align*}
    P_{5}-P_{6}&= \frac{(1-\mu)}{2} (1,0,1), \\
    P_{3}-\tilde{P}_{3} &= 2 \lambda \delta (0,1,0), \\
    P_{4}-P_{2}&=(-1,0,1).
  \end{align*}
  We note that the first two vectors are linearly independent and hence span the
  plane containing $P_{5}P_{3}\tilde{P_{3}}$.
  Since both vectors are orthogonal to $P_{4}-P_{2}$, it follows that
  $P_{4}-P_{2}$ is orthogonal to $P_{5}P_{3}\tilde{P_{3}}$, as claimed.\\

\underline{The tetrahedron $P_{3}\hat{P}_{3}P_{1}P_{4}$:}  
We argue similarly as in the previous case and introduce $P_{7}=l P_{1}
+(1-l)P_{4}$ with $l=\frac{\delta^{2}\lambda+1}{\delta^{2}+1} \in (0,1)$
and $P_{8}=(0,\lambda \delta, 0)= \frac{1}{2}P_{3}+\frac{1}{2}\hat{P}_{3}$.
  Then a short calculation yields
  \begin{align*}
    P_{7}-P_{8}&=\frac{(1-\lambda)\delta }{\delta^{2}+1} (0,1,\delta), \\
    P_{3}-\hat{P}_{3}&= (2 \mu, 0, 0), \\
    P_{4}-P_{1}&= (0, -\delta , 1).
  \end{align*}
  We note that $P_{4}-P_{1}$ is orthogonal to both other vectors and hence
  orthogonal to $P_{3}\hat{P}_{3}P_{7}$.
  Thus, the tetrahedra $P_{3}\hat{P}_{3}P_{7}P_{4}$ and
  $P_{3}\hat{P}_{3}P_{7}P_{1}$ are triangle pyramids.\\

\underline{The tetrahedron $P_{1}P_{2}P_{3}P_{4}$:}
 We claim that there exists $P_{3}^{\star} \in P_{1}P_{2}P_{4}$ such that
 $P_{3}-P_{3}^{\star}$ is orthogonal to $P_{1}P_{2}P_{4}$.
 If this is the case, then Lemma \ref{lem:3D_reduction2} yields a partition of
 $P_{1}P_{2}P_{3}P_{4}$ into six triangle pyramids.

Indeed, let $P_{9}=(-1,-\delta^{-1},0)$, then it holds that $P_{9}-P_{4}$ is orthogonal to the plane containing $P_{4}P_{2}P_{1}$, since
  \begin{align*}
    P_{4}-P_{9}&= (1,\delta^{-1}, 1), \\
    P_{2}-P_{1}&=(1,-\delta, 0), \\
    P_{4}-P_{2}&= (-1,0,1),
  \end{align*}
  satisfy $(P_{9}-P_{4}) \bot (P_{2}-P_{1})$ and $(P_{9}-P_{4}) \bot (P_{4}-P_{2})$.
  Thus $P_{4}$ is the orthogonal projection of $P_{9}$ onto $P_{4}P_{2}P_{1}$.
  Finally, we note that the set of all points in the $xy$-plane such that their
  orthogonal projection onto the plane containing $P_{1}P_{2}P_{4}$ is contained in
  $P_{1}P_{2}P_{4}$ is a convex set and contains $P_{9}, P_{1}$ and $P_{2}$.
  The claim thus follows by noting that $P_{3}$ is contained in  $P_{9}P_{1}P_{2}$. 
\end{proof}

\subsubsection{The self-similar case in 3D}
\label{sec:3d_self_similar}

Following a similar approach as in the two-dimensional setting which was considered in Section \ref{sec:2D} and which is depicted in Figure
\ref{fig:self-similar}, we explain how to cover a box by diamonds and a controlled remainder. As in the two-dimensional situation, we here introduce a self-similar covering by diamonds of
dyadic sizes. This gives rise to the cases \ref{item:A3tilde} (ii), (iii) involving the class of domains in $\mathcal{C}^1$.
Due to the more complicated geometry in three dimensions we introduce several
additional building blocks:
\begin{itemize}
\item We consider a \emph{symmetric diamond} as constructed in Lemmas \ref{lem:3D_O(n)}, \ref{lem:lin_Conti_3D}. Without loss of generality after rotation and scaling, we may assume that it is the convex hull of the points $(0,0,\pm \delta), (\pm 1, \pm 1, 0), (\mp 1, \pm 1,0)$. In particular, its base is a square.
We call this shape, as well as translates and rescalings of
it, a \emph{symmetric diamond}.
\item Since our covering by diamonds leaves some gaps, we further introduce
  \emph{symmetric tetrahedra} $T$, which up to rescaling, translation and switching of the $x$- and $y$- axes
  have corners $(-2,0,-\delta),(2,0,-\delta),(0,-2,\delta)$ and $(0,2,\delta)$
  (c.f. Figure \ref{fig:tetra}).
\item In analogy to the self-similar triangles in the 2D setting, we introduce a
  \emph{ring} shaped domain $\mathcal{R}_j$ for $j \in \N_0$.
  These domains are given by non-convex polygons with corners $(\pm
  1, \pm 1, \pm \delta), (\mp
  1, \pm 1, \pm \delta)$ and $(\pm(1-2^{-j}), \pm (1-2^{-j}),0), (\mp(1-2^{-j}), \pm (1-2^{-j}),0)$ and are
  depicted in Figure \ref{fig:ring}.
  Their projections onto the $yz$-plane then include isosceles triangles as in
  Figure \ref{fig:self-similar}, while the $xy$-projection is a square annulus.
  We remark that in the special case $j=0$ the ring consists of $8$ rectangle
  pyramids and instead of an annulus the projection is a rectangle.
\end{itemize}

Given a symmetric diamond, we define the class $\mathcal{C}^1$ as the collection of tetrahedra $T$ and rings $\mathcal{R}_j$ described above.
For a rotated diamond, we also rotate the tetrahedra and rings correspondingly.\\

\begin{figure}[t]
  \centering
  \includegraphics[width=1\linewidth, page=20]{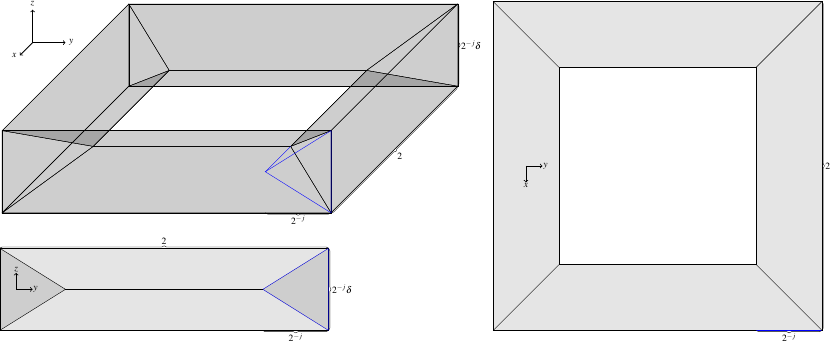}
  \includegraphics[width=1\linewidth, page=21]{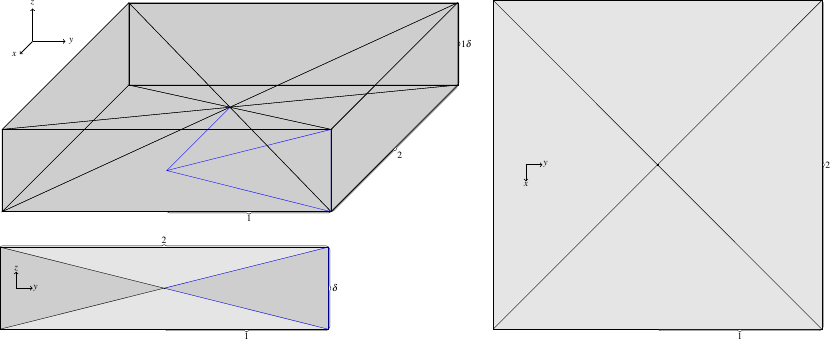}
  \caption{Ring-shaped domains in various projections. Top: $j>0$, Bottom: $j=0$.}
  \label{fig:ring}
\end{figure}

With these additional building blocks at hand, we formulate our main covering lemma, which should be viewed as analogous to the two-dimensional statement of Lemma \ref{lem:covering_reduction}.

\begin{lem}
\label{lem:covering_reduc_3D}
Assume that the estimates from \ref{item:A3tilde}-\ref{item:A4} are satisfied,
if $\Omega = \Omega^{\Diamond}$, where $\Omega^{\Diamond}$ is a symmetric diamond for some $\delta \in (0,1/2)$. In the case $\ref{item:A3tilde}$ the parameter $\delta$ may depend on $
k,j$, while in the case $\ref{item:A3}$ it is required to be independent of $k,j$. Suppose further that $\Omega_g = \Omega^{\Diamond}$, which in particular yields that the conditions \ref{item:A3tilde} (ii), (iii) are empty for $\Omega = \Omega^{\Diamond}$. Denote the constants from conditions \ref{item:A3tilde} (i) and \ref{item:A3} by $C_{0}^{\star},C_{1}^{\star}$, if $\Omega=\Omega^{\Diamond}$, and assume that both are uniform in $j,k$.
  \begin{enumerate}[label=(\roman*)]
  \item \label{item:selfsimilar_case_3D}
If $\Omega=\mathcal{R}_j$ is a ring or if $\Omega=T$ is a symmetric tetrahedron, there is a replacement construction, which
satisfies \ref{item:A3tilde} (ii) or \ref{item:A3} with
  $C_{0}=C_{0}^{\star}$, $C_{1}=C_{1}^{\star}$, $C_{2}=2$.
  The volume fractions are given by $v_{1}=1/3$ in the case of a ring and by
  $v_1=\frac{1}{2}$ in the case of a tetrahedron. We define the set $\mathcal{C}^1$ to consist of these symmetric tetrahedra or rings, which are oriented in the same way as the relevant current symmetric diamond.
\item \label{item:box_case_3D}
If $\Omega = \Omega^{\Box}$ is a box of aspect ratio $1: 1:  \delta \cdot
  \lfloor{\frac{1}{\delta}} \rfloor$, there exists a replacement construction
  such that \ref{item:A3tilde} and \ref{item:A3} are satisfied with constants
  $C_{0}=C_{0}^{\star}/\delta$, $C_{1}=C_{1}^{\star}/\delta$, $C_{2}=8$ and $v_{1}=1/3$.
\item \label{item:general_case_3D}
  For any $\Omega \in \mathcal{C}$ there exists a construction such that \ref{item:A3tilde}-\ref{item:A4} are satisfied with constants $C_{0}=100C_{0}^{\star}/\delta$,
  $C_{1}=100C_{1}^{\star}/\delta$, $C_{2}=100$ and $v_{1}\geq 10^{-6}$. 
  \end{enumerate}
\end{lem}

Using the additional building blocks from above, we define the full class of objects, which are used in three-dimensions:

\begin{defi}
\label{defi:C3D}
Let $\mathcal{C}^0$ be as in Definition \ref{defi:3D_shapes} and let $\mathcal{C}^1$ be as described at the beginning of this subsection.
Then we define the class $\mathcal{C}$ as $\mathcal{C}=\mathcal{C}^0 \cup \mathcal{C}^1$.
\end{defi}

\begin{proof}[Proof of Lemma \ref{lem:covering_reduc_3D}]
  We largely follow the same strategy as in the proof of Lemma
  \ref{lem:covering_reduction}.
  \\
  
  \underline{Case \ref{item:selfsimilar_case_3D}:}
  We begin by discussing the case of a tetrahedron $T$, which after rescaling,
  translation and possibly relabelling of the $x$- and $y$-axes we may assume to have
  corners $(-2,0,-\delta),(2,0,-\delta),(0,-2,\delta)$ and
$(0,2,\delta)$.
  We note that the cross-section at $z=0$ is given by the square with corners
$(\pm 1, \pm 1, 0), (\pm 1, \mp 1,0)$, since $\frac{1}{2} (\pm 2, 0, -\delta) + \frac{1}{2}(0,\pm 2,
\delta)=(\pm 1, \pm 1, 0)$.
We then introduce the points $(0,0,-\delta)$ and $(0,0,\delta)$. Connecting the square
with these points, we obtain a diamond $D$ with this square as base and as a
remainder we obtain four self-similar copies of the tetrahedron
$T_1,T_2,T_3,T_4$ whose lengths are rescaled by a factor $1/2$ (c.f. Figure
\ref{fig:tetra}).
\begin{figure}[h]
  \centering
  \includegraphics[width=0.7\linewidth, page=22]{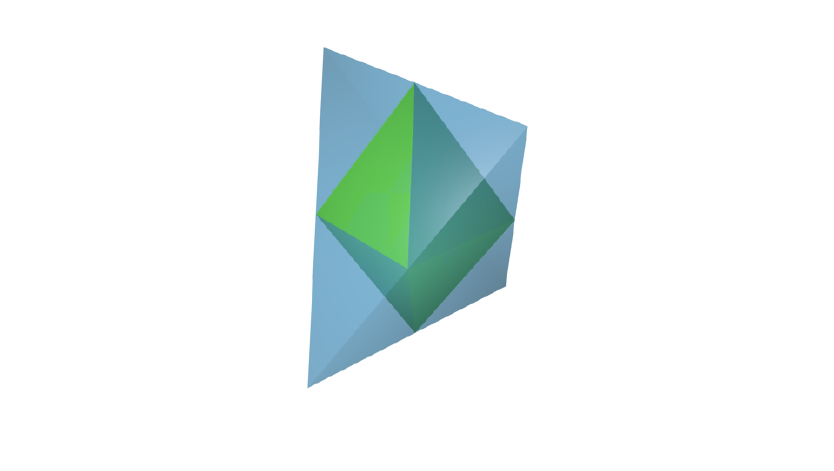}
  \caption{Fitting a symmetric diamond into a symmetric tetrahedron.}
  \label{fig:tetra}
\end{figure}

It follows that $|T_1|+|T_2|+|T_3|+|T_4|=4 \cdot \left( \frac{1}{2}
\right)^3|T|=\frac{1}{2}|T|$ and thus $|D|=\frac{1}{2}|T|$. Furthermore, by
direct computation
\begin{align*}
  \Per(D)+\sum_{i=1}^4\Per(T_i) \leq 2 \Per(T).
\end{align*}
Hence, the lemma is proven for the case of a symmetric tetrahedron.
\\

Next let $j \in \N_0$ and consider the ring shaped domain $\mathcal{R}_j$.
We then may cover half of this domain's volume using $d_j:=4\cdot(2^{j+1}-1)$ symmetric
diamonds, $D_1,\dots,D_{d_j}$, of size $2^{-j-1}$ as well as $t_j:=4\cdot(2^{j+1}-2)$ symmetric tetrahedra, $T_1,\dots,T_{t_j}$
(c.f. Figure \ref{fig:ring_a}).
\begin{figure}[t]
  \centering
  \includegraphics[width=0.24\linewidth, page=23]{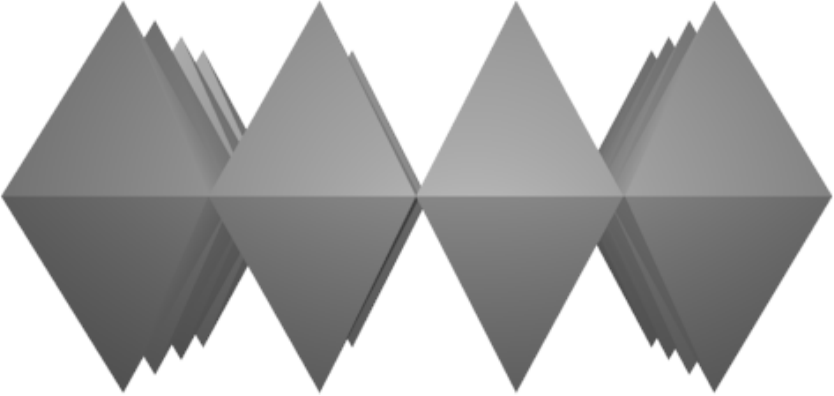}
  \includegraphics[width=0.24\linewidth, page=24]{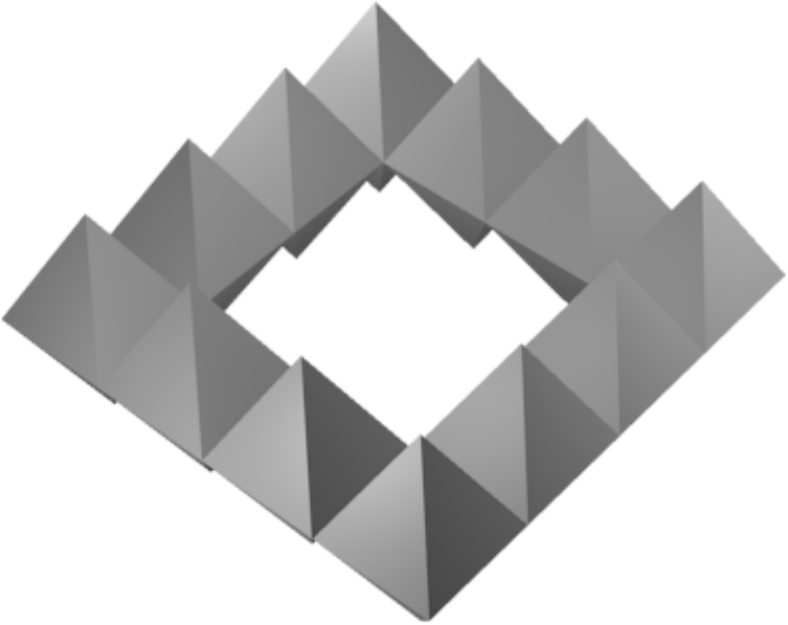}
  \includegraphics[width=0.24\linewidth, page=25]{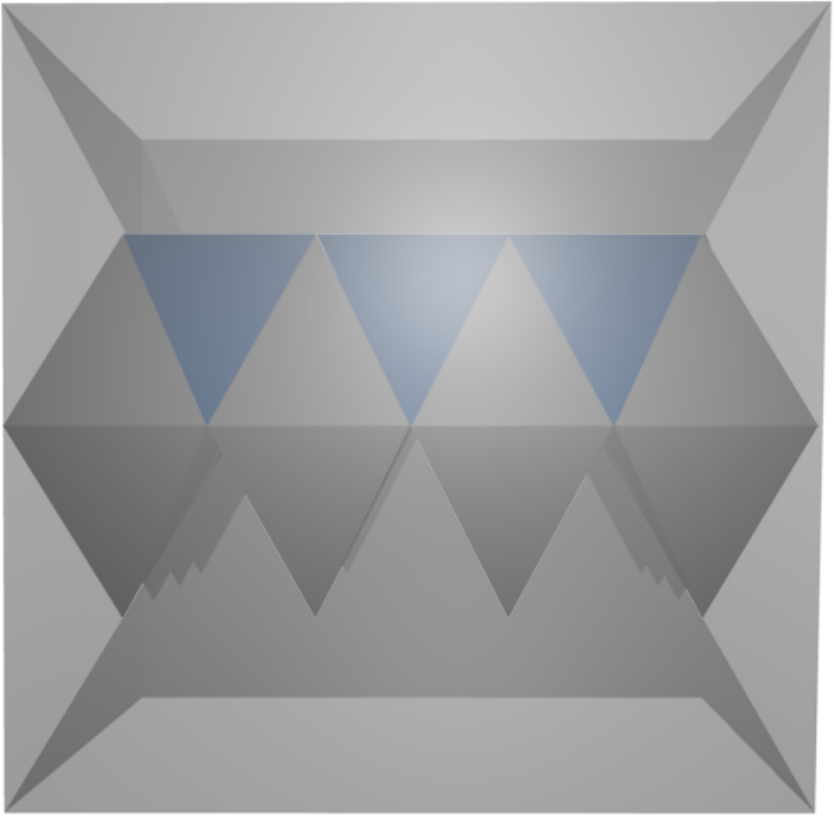}
  \includegraphics[width=0.24\linewidth, page=26]{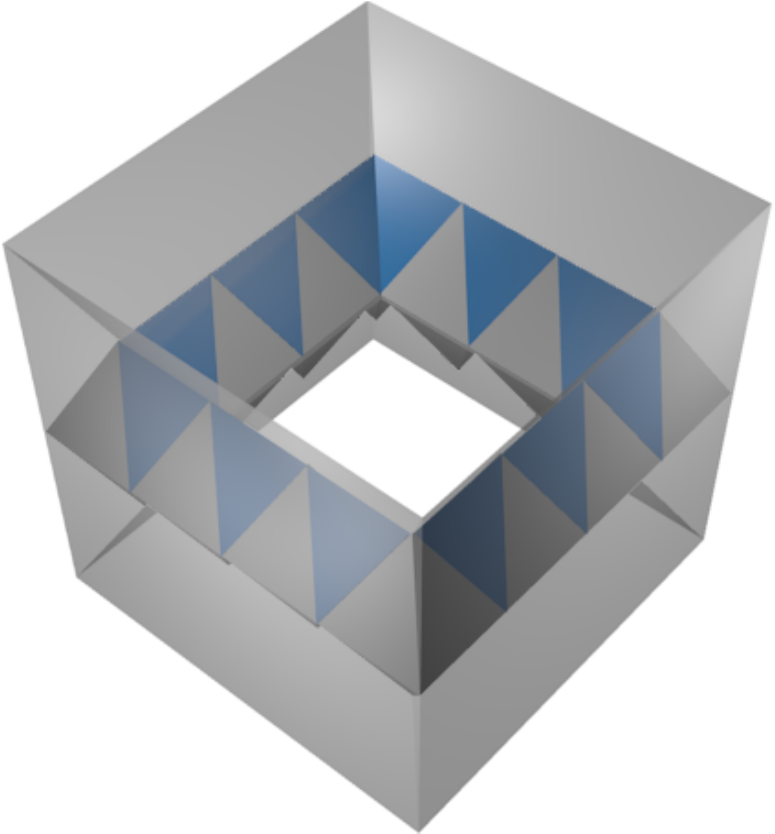}
  \caption{Combining symmetric diamonds and tetrahedra, we can cover half the volume
    of a ring $\mathcal{R}_j$. The complement then is given by two copies of $\mathcal{R}_{j+1}$. More precisely, we begin with a ring $\mathcal{R}_j$, which is shown on the two right hand side figures as the domain shaded in grey. We seek to cover half of its volume by tetrahedra and diamonds. This is achieved by forming an inner square shaped annulus by first stacking diamonds (as shown in the two figures on the left) and then filling the gaps in between by tetrahedra (depicted in blue in the two figures on the right hand side).
    }
  \label{fig:ring_a}
\end{figure}

\begin{figure}[t]
  \includegraphics[width=0.9\linewidth, page=27]{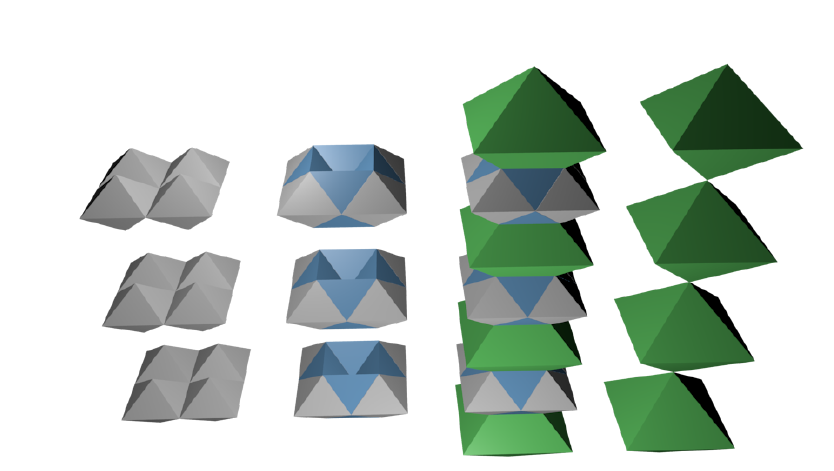}
  \caption{Steps of the box covering, here for $\delta=1/4$.
  We first stack the diamonds, which contain the replacement construction from Lemmas \ref{lem:K_lin_replace_b}, \ref{lem:K_lin_replace}, \ref{lem:On_replace_in}, \ref{lem:On_replace}. This corresponds to the stacking of the green diamonds on the very right. Then we iteratively seek to fill the gaps between the green diamonds in such a way that in each step, the volume of the remaining gaps is reduced by a fixed factor. To this end, we note that the gaps between the green diamonds are always rings $\mathcal{R}_j$, for which we can hence invoke the covering construction from Lemma \ref{lem:covering_reduc_3D} (i). The construction filling the gap in the first iteration step is depicted in the two left most figures: We first stack the grey diamonds which are of half the size of the green diamonds, fill these up by the blue tetrahedra, to create a band, which is then put around the green diamonds in the third figure.}
  \label{fig:box_3D}
\end{figure}

Scaling then yields a total perimeter bound 
\begin{align*}
&\sum\limits_{l=1}^{d_j}\Per(D_l)+ \sum\limits_{l=1}^{t_j}\Per(T_l)
+ \Per(R_j \setminus \left(\bigcup\limits_{l=1}^{t_j}T_l \cup \bigcup\limits_{l=1}^{d_j}D_l  \right)) \leq C \left( 2^{-j-1}
\right)^3 2^{j+1}\\
& =C 2^{-2j-2} \leq C \Per(\mathcal{R}_j).
\end{align*}
Concerning the volume fractions, we note that the diamonds cover twice as much
volume as the symmetric tetrahedra and hence $v_1=\frac{1}{3}$ corresponds to
the volume fraction covered by the diamonds.
Finally, we note that the complement of this cover is given by two copies of
$\mathcal{R}_{j+1}$ with $\Per(\mathcal{R}_{j+1})\leq \Per(\mathcal{R}_j)$, so we we can
iterate self-similarly on this complement, as well as on the symmetric tetrahedra.
\\

\underline{Case \ref{item:box_case_3D}:}
As in the $2D$ setting, we vertically stack $\left\lfloor \frac{1}{\delta} \right\rfloor$ copies of the square diamond (c.f. Figure \ref{fig:box_3D}).
This then covers a volume fraction $v_1=\frac{1}{3}$ of the enveloping
axis-parallel box $\Omega^{\Box}$ and its perimeter is controlled by $1/\delta$
times the perimeter of the box.
As in the 2D setting, the top and bottom can be decomposed into $8$ rectangle
pyramids of small perimeter, while the remainder is given by
$\left\lfloor \frac{1}{\delta} \right\rfloor -1$ copies of $\mathcal{R}_0 \in \mathcal{C}^1$.
\\

\underline{Case \ref{item:general_case_3D}:}
Let $\Omega \in \mathcal{C}$ be one of the building blocks as described in
Definition \ref{defi:3D_shapes}. Then using Lemma \ref{lem:box3Dint} it suffices to
provide a covering for a box $R$ contained in $\Omega$.
Furthermore, as sketched in Figure \ref{fig:fitting_square} in 2D, we can cover
at least a quarter of the volume of a box of lengths $1:a:b$ with $a,b\geq 1$ by
$\lfloor a \rfloor \cdot \lfloor b \rfloor \approx
\Per(R)$ many unit cubes.
Using Lemma \ref{lem:Euler_angles}, we can cover a large volume fraction of each
of these cubes by an axis-parallel cube.
Finally, in each axis-parallel cube, we make use of case \ref{item:box_case_3D}
and thus conclude our proof.
\end{proof}

\subsection{General Lipschitz domains}

In this section briefly comment on the ideas which are used to extend the setting from domains in the class $\mathcal{C}$ to general bounded Lipschitz domains. As the argument proceeds as in \cite{RZZ16}, we omit most details and refer to Section 6 in \cite{RZZ16}.

\begin{lem}
\label{lem:Lip}
Let $\theta_0 \in (0,1)$ and $n\in \N$.
Assume that for any cube $Q \subset \R^n$ and for any limit $u\in W^{1,\infty}(Q, \R^n)$ of a sequence $u_k:Q \rightarrow \R^n$, obtained through Algorithm \ref{alg:convex_int}, it holds $\nabla u \in W^{s,p}(Q,\R^n)$ for all $s\in(0,1)$, $p \in (1,\infty)$ with $sp <\theta_0$. Moreover assume that the bound
\begin{align}
\label{eq:est_cube}
\|\nabla u_{k+1} - \nabla u_{k}\|_{L^1(Q)}^{1-\theta}\|\nabla u_{k+1} - \nabla u_{k}\|_{BV(Q)}^{\theta} \leq C_Q \mu^{k}
\end{align}
holds, where $\mu=\mu(s,p)\in(0,1)$ and $\theta=\theta(s,p)\in(0,1)$ is the interpolation exponent for the $W^{s,p}$ interpolation from \cite{CDDD03} and $C_Q>1$ is independent of $k$.
Then for any Lipschitz domain $\Omega \subset \R^n$ and any $s\in(0,1)$, $p \in (1,\infty)$ with $sp <\theta_0$ there are deformations $u \in W^{1,\infty}(\Omega)$ solving \eqref{eq:incl} with the additional property that $\nabla u \in W^{s,p}_{loc}(\Omega, \R^n)$ and
\begin{align*}
\|\nabla u\|_{W^{s,p}(\Omega)} \leq C(\Omega, \mu, sp, C_Q)<\infty.
\end{align*}
\end{lem}

\begin{proof}
The proof follows as in Section 6 in \cite{RZZ16}. Its consists of two main ingredients, for whose details we refer to \cite{RZZ16}. In a first step we note that locally the boundary of $\Omega$ can be written as a Lipschitz graph. In a second step, we exhaust the set below the graph by cubes. Using \eqref{eq:est_cube}, it is thus possible to obtain a telescope sum estimate on $\Omega$, which is similar to \eqref{eq:est_cube}.
\end{proof}

As any cube can be partitioned into two right-angled triangles or two ramps, we always have that cubes are contained in our admissible class $\mathcal{C}$. Hence Lemma \ref{lem:Lip} is applicable, extending the result from polygonal domains $\Omega$ which can be decomposed into finitely many elements of $\mathcal{C}$ to general Lipschitz domains.

\bibliographystyle{alpha}
\bibliography{citations1}

\end{document}